\documentclass[12pt]{amsart}

\usepackage{latexsym}
\usepackage{psfrag}
\usepackage{amsmath}
\usepackage{amssymb}
\usepackage{epsfig}
\usepackage{amsfonts}
\usepackage{amscd}
\usepackage{mathrsfs}
\usepackage{graphicx}
\usepackage{enumerate}

\newtheorem{theorem}{Theorem}%[section]

\newtheorem{prop}[theorem]{Proposition}
\newtheorem{conj}[theorem]{Conjecture}
\newtheorem{defn}[theorem]{Definition}

\dedicatory{Dedicated to George E. Andrews for his 80th birthday.}

\begin{document}%%%%%%%%%%%%%%%%%%%%%%%%%%%%%%%%%%%%%%%%%%%%%%%%%%%%%%%%
%%%%%%%%%%%%%%%%%%%%%%%%%%%%%%%%%%%%%%%%%%%%%%%%%%%%%%%%%%%%%%%%%%%%%%%%

\title[Andrews-Gordon Type Series]{Andrews-Gordon Type Series for Kanade-Russell Conjectures}

\author[Kur\c{s}ung\"{o}z]{Ka\u{g}an Kur\c{s}ung\"{o}z}
\address{Faculty of Engineering and Natural Sciences, Sabanc{\i} University, \.{I}stanbul, Turkey}
\email{kursungoz@sabanciuniv.edu}

%    For articles to be published after 1 January 2010, you may use
%    the following version:


\date{August 2018}

\begin{abstract}
\noindent
We construct Andrews-Gordon type evidently positive series 
as generating functions of partitions satisfying certain difference conditions 
in six conjectures by Kanade and Russell.  
We construct generating functions for missing partition enumerants, 
naturally without claiming new partition identities.  
Thus, we obtain $q$-series conjectures as companions to 
Kanade and Russell's combinatorial conjectures.  
\end{abstract}

\maketitle

%%%%%%%%%%%%%%%%%%%%%%%%%%%%%%%%%%%%%%%%%%%%%%%%%%%%%%%%%%%
%%%%                                                   %%%%
%%%%                   Introduction                    %%%%
%%%%                                                   %%%%
%%%%%%%%%%%%%%%%%%%%%%%%%%%%%%%%%%%%%%%%%%%%%%%%%%%%%%%%%%%

\section{Introduction}
\label{secIntro}

In November of 2014, Kanade and Russell announced six new 
partition identities using some computer help \cite{KR-conj}.  
The difference conditions on partitions are inspired by Capparelli's identities 
\cite{Capparelli-Conj, AAG-Capparelli}.  

The first of the conjectures is given below.  
\begin{conj}[Kanade-Russell conjecture $I_1$]
  The number of partitions of a non-negative integer 
  into parts $\equiv \pm 1, \pm 3 \pmod{9}$ 
  is the same as the number of partitions 
  with difference at least three at distance two 
  such that if two successive parts differ by at most one, 
  then their sum is divisible by three.  
\end{conj}
Here, difference at distance two means the difference between $i$th and $(i+2)$th parts.  
The former condition in the conjecture is a congruence condition, 
and the latter is a difference condition.  
For example, $n = 9$ has seven partitions satisfying the first constraint: 
\begin{align*}
 1+1+\cdots+1, & \quad 1+1+\cdots+1+3, \quad 1+1+1+3+3, \\
 1+1+1+6, & \quad 1+8, \quad 3+3+3, \quad 3+6, 
\end{align*}
as well as seven partitions satisfying the second constraint: 
\begin{align*}
 9, \quad 1+8, \quad 2+7, \quad 3+6, \quad 1+3+5, \quad 4+5, \quad 1+2+6.  
\end{align*}

A quote attributed to the late A.O.L. Atkin 
asserts that it is often easier to prove identities in the 
theory of $q$-series than to discover them.  
Kanade and Russell's conjectures have been counterexamples, 
since they evaded proof for more than three years so far.  
This paper, unfortunately, is no attempt to prove them.  

The goal of this paper is to construct Andrews-Gordon type series 
as generating functions of the partitions in the conjectures.  
In particular, generating functions for partitions satisfying the difference conditions 
in them will be constructed.  
Gordon marking of a partition and clusters will be utilized \cite{K-Clusters}.   

The next section lists the definitions and a small result 
that will be used throughout the paper.  
Section \ref{secKR1-4} deals with the first four or the ``$\pmod{9}$'' conjectures 
and some missing cases.  
Section \ref{secKR5-6} treats the last two or the ``$\pmod{12}$'' conjectures 
and some missing cases.  
Section \ref{secKR5-6Alt} lists alternative generating functions of section \ref{secKR5-6}.  
We do not assert any partition identities 
for the missing cases in sections \ref{secKR1-4}-\ref{secKR5-6Alt}.  
In the short section \ref{secKR-qSeries}, we collect some of the constructed series thus far, 
and state $q$-series conjectures as analytic companions for the 
Kanade-Russell's combinatorial conjectures.  
We conclude with some commentary, a few open problems, 
and some directions for further research in section \ref{secConclusion}.  
The appendix by Emre Erol contains a metaphor and explanation 
for parts of a construction in section \ref{secKR5-6} and related terminology.  

\section{Definitions and Preliminary Results}
\label{secDefns}

An integer partition $\lambda$ of a natural number $n$ is a non-decreasing sequence of positive integers 
that sum up to $n$.  
\begin{align*}
 n = & \lambda_1 + \lambda_2 + \cdots + \lambda_m , \\
 & \lambda_1 \leq \lambda_2 \leq \cdots \leq \lambda_m 
\end{align*}
The $\lambda_i$'s are called parts.  
The number of parts $m$ is called the length of the partition $\lambda$, denoted by $l(\lambda)$.  
The number being partitioned is the weight of the partition $\lambda$, denoted by $\vert \lambda \vert$.  
One could also reverse the weak inequalities and take non-decreasing sequences, 
but we will stick to this definition for purposes of this note.  
The point is that reordering the same parts will not give us a new partition.  
For example, the five partitions of $n = 4$ are
\begin{align*}
 4, \quad 1+3, \quad 2+2, \quad 1+1+2, \quad 1+1+1+1.  
\end{align*}

We sometimes allow zeros to appear in the partition.  
Clearly, they have no contribution to the weight of the partition, 
but the length changes as we add or take out zeros.  

Given a partition $\lambda$, if there exists positive integers $d$ and $k$ such that
$\lambda_{j+d} - \lambda_{j} \geq k$ for all $j = 1, 2, \ldots, l(\lambda)-d$, 
we say that $\lambda$ \emph{has difference at least $k$ at distance d}.  

Many partition identities have the form 
``the number of partitions of $n$ satisfying condition A 
= the number of partitions of $n$ satisfying condition B'' \cite{Andrews-bluebook}.  
We can abbreviate this as $p(n \vert \textrm{ cond. } A)$ 
$= p(n \vert \textrm{ cond. } B)$.  
Any form of the series
\begin{align*}
 F(q) = \sum_{n \geq 0} p(n \vert \textrm{ cond. } A) q^n
\end{align*}
is called a \emph{partition generating function}.  
Or $F(q)$ is said to generate $p(n \vert \textrm{ cond. } A)$.  

The definitions below are taken from \cite{K-Clusters}.  
Although they are lengthy, they are included here for self-containment.  

\begin{defn}
\label{defGordonMarking}
 The \emph{Gordon marking} of a partition $\lambda$ is an assignment of positive integers (marks) to 
 $\lambda$ such that parts equal to any given integer $a$ are assigned distinct marks from the set 
 $\mathbb{Z}_{>0} \backslash \{ r \vert \exists \; r\textrm{-marked } \lambda_j = a-1 \}$ 
 such that the smallest possible marks are used first.  
 We can repsesent the Gordon marking by a two-dimensional array, 
 where the row index, counted from bottom to top indicates the mark.  
\end{defn}

{\bf Example: } For the partition 
\begin{align*}
 \lambda = 2+2+3+4+5+6+6+7+9+11+13+13+15+15+16+17+18, 
\end{align*}
the Gordon marking is 
\begin{align*}
 \lambda = & 2_1+2_2+3_3+4_1+5_2+6_1+6_3+7_2+9_1+11_1 \\ 
 & +13_1+13_2+15_1+15_2+16_3+17_1+18_2, 
\end{align*}
or 
\begin{align*}
  \left\{
    \begin{array}{cc} & 3 \\ 2 & \\ 2 & \end{array} 
    \begin{array}{cccc} & & 6 & \\ & 5 & & 7 \\ 4 & & 6 & \end{array} 
    \begin{array}{cc} \\ \\ 9 \end{array} 
    \begin{array}{cc} \\ \\ 11 \end{array} 
    \begin{array}{cc} \\ 13 \\ 13 \end{array} 
    \begin{array}{cc} & 16 \\ 15 & \\ 15 & \end{array} 
    \begin{array}{cc} \\ & 18 \\ 17 & \end{array} 
  \right\}
    \begin{array}{cc} \\ \\ . \end{array} 
\end{align*}

This last representation of partitions will be used throughout the note.  

\begin{defn}
\label{defFwdMove}
  Given a partition $\lambda$, 
  let $\lambda_j$ be an $r$-marked part such that 
  
  \begin{itemize}
    
   \item[(a)] there are no $r+1$ or higher marked parts $= \lambda_j$ or $= \lambda_j + 1$; 
   
   \item[(b1)] either there is an $r_0$ marked part $\lambda_{j_0} = \lambda_j - 1$, 
    $r_0 < r$ such that there are no $r_0$-marked parts $= \lambda_j + 1$, 
    and no $r_0 + 1$ or higher marked parts equal to $\lambda_j-1$, 
   
   \item[(b2)] or there are $1, 2, \ldots, (r-1)$-marked parts $= \lambda_j$ or $= \lambda_j+1$, 
    and no $r$-marked parts $= \lambda_j+2$.  
   
  \end{itemize}
  A \emph{forward move of the $r$th kind} is replacing the $r_0$-marked $\lambda_{j_0}$ 
  with an $r_0$ marked $\lambda_{j_0} + 1$ if (a) and (b1) hold; 
  and replacing the $r$-marked $\lambda_j$ with an $r$-marked $\lambda_j+1$ 
  if (a) and (b2) hold, but (b1) fails.  
\end{defn}

{\bf Example: }
A forward move of the 3rd kind on the 3-marked 16 (in boldface)
of the partition in the above example makes the partition 
\begin{align*}
  \left\{
    \begin{array}{cc} & 3 \\ 2 & \\ 2 & \end{array} 
    \begin{array}{cccc} & & 6 & \\ & 5 & & 7 \\ 4 & & 6 & \end{array} 
    \begin{array}{cc} \\ \\ 9 \end{array} 
    \begin{array}{cc} \\ \\ 11 \end{array} 
    \begin{array}{cc} \\ 13 \\ 13 \end{array} 
    \begin{array}{cc} & \mathbf{16} \\ & 16 \\ 15 & \end{array} 
    \begin{array}{cc} \\ & 18 \\ 17 & \end{array} 
  \right\}
    \begin{array}{cc} \\ \\ . \end{array} 
\end{align*}

\begin{defn}
\label{defBackwdMove}
 For a partition $\lambda$, 
 let $\lambda_j \neq 1$ be an $r$-marked part such that 
 \begin{itemize}
  
  \item[(c)] there are no $(r+1)$ or greater marked parts that are $=\lambda_j$ or $= \lambda_j+1$, 
  
  \item[(d)] there is an $r_0 \leq r$ such that there is an $r_0$-marked $\lambda_{j_0}=\lambda_j$, 
    but no $r_0$-marked parts $=\lambda_j-2$.  
  
 \end{itemize}
  Choose the smallest $r_0$ described in (d).  
  A \emph{backward move of the $r$th kind} on $\lambda_j$ 
  is replacing the $r_0$-marked $\lambda_{j_0}$ with an $r_0$-marked $\lambda_{j_0}-1$.  
\end{defn}

{\bf Example: }
A backward move of the 3rd kind on the 3-marked 6 of the last displayed partition makes it 
\begin{align*}
  \left\{
    \begin{array}{cc} & 3 \\ 2 & \\ 2 & \end{array} 
    \begin{array}{cccc} & \mathbf{5} & & \\ & 5 & & 7 \\ 4 & & 6 & \end{array} 
    \begin{array}{cc} \\ \\ 9 \end{array} 
    \begin{array}{cc} \\ \\ 11 \end{array} 
    \begin{array}{cc} \\ 13 \\ 13 \end{array} 
    \begin{array}{cc} & 16 \\ & 16 \\ 15 & \end{array} 
    \begin{array}{cc} \\ & 18 \\ 17 & \end{array} 
  \right\}
    \begin{array}{cc} \\ \\ . \end{array} 
\end{align*}
The 6 becomes 5 (in boldface).  

\begin{defn}
\label{defCluster}
 An \emph{$r$-cluster} in $\lambda = \lambda_1 + \lambda_2 + \cdots + \lambda_m$ 
 is a sub partition $\lambda_{i_1} \leq \lambda_{i_2} \leq \cdots \leq \lambda_{i_r}$
 such that $\lambda_{i_j}$ is $j$-marked for $j = 1, 2, \ldots, r$, 
 $\lambda_{i_{j+1}} - \lambda_{i_j} = 0$ or 1 for $j = 1, 2, \ldots, r-1$, 
 and there are no $(r+1)$-marked parts $= \lambda_{i_r}$ or $= \lambda_{i_r}+1$.  
\end{defn}

{\bf Example: } 
\begin{align*}
  \left\{ 
    \begin{array}{cc} & 3 \\ 2 & \\ 2 & \end{array} 
    \begin{array}{cccc} & & 6 & \\ & 5 & & 7 \\ 4 & & 6 & \end{array} 
    \begin{array}{cc} \\ \\ 9 \end{array} 
    \begin{array}{cc} \\ \\ 11 \end{array} 
    \begin{array}{cc} \\ 13 \\ 13 \end{array} 
    \begin{array}{cc} & 16 \\ 15 & \\ 15 & \end{array} 
    \begin{array}{cc} \\ & 18 \\ 17 & \end{array} 
  \right\} 
\end{align*}
has the following clusters.  
\begin{align*}
  \left\{ \begin{array}{c} \\ \\ \phantom{0} \end{array} \right.
    \underbrace{
    \begin{array}{cc} & 3 \\ 2 & \\ 2 & \end{array} 
    }_\textrm{a 3-cluster} \;
    \underbrace{
    \begin{array}{ccc} & & 6\\ & 5 & \\ 4 & &\end{array} 
    }_\textrm{a 3-cluster} \;
    \underbrace{
    \begin{array}{cc} & \\ & 7 \\ 6 & \end{array} 
    }_\textrm{a 2-cluster} \;
    \underbrace{
    \begin{array}{cc} \\ \\ 9 \end{array} 
    }_\textrm{a 1-cluster} \;
    \underbrace{
    \begin{array}{cc} \\ \\ 11 \end{array} 
    }_\textrm{a 1-cluster} \;
    \underbrace{
    \begin{array}{cc} \\ 13 \\ 13 \end{array} 
    }_\textrm{a 2-cluster} \;
    \underbrace{
    \begin{array}{cc} & 16 \\ 15 & \\ 15 & \end{array} 
    }_\textrm{a 3-cluster} \;
    \underbrace{
    \begin{array}{cc} \\ & 18 \\ 17 & \end{array} 
    }_\textrm{a 2-cluster} 
  \left. \begin{array}{c} \\ \\ \phantom{0} \end{array} \right\}
\end{align*}

When we compare two clusters, not necessarily having the same number of parts, 
we compare the 1-marked parts in them.  
The \emph{largest 2-cluster} means the 2-cluster having the largest 1-marked part etc.  

We will also need the following result in section \ref{secKR5-6}.  

\begin{prop}
\label{propDistinctOdds}
  The partitions into at most $n$ parts, 
  where no odd part repeats is generated by 
  $\displaystyle \frac{(-q; q^2)_n}{(q^2; q^2)_n}$.  
\end{prop}

\begin{proof}
 By the $q$-binomial theorem \cite{Andrews-bluebook}, 
 \begin{align*}
%  \label{eqPropDistinctOdds}
  \frac{(-qt; q^2)_\infty }{ (t; q^2)_\infty } 
  = \sum_{n \geq 0} \frac{ (-q; q^2)_n }{ (q^2; q^2)_n } t^n.  
 \end{align*}
 The right hand side obviously generates partitions in which no odd part repeats, 
 and the exponent of $t$ accounts for the number of parts, zeros allowed.  
\end{proof}

Here, and throughout, 
\begin{align*}
 (a; q)_n = & \prod_{j = 0}^n (1 - aq^{j-1}), \\
 (a_1, a_2, \ldots, a_k; q)_n = & (a_1; q)_n (a_2; q)_n \cdots (a_k; q)_n 
\end{align*}
for $n \in \mathbb{N} \cup \{\infty\}$ and $\vert q \vert < 1$.  

It is also possible to give a purely combinatorial proof of Proposition \ref{propDistinctOdds}.   
However, it will just be a twist of a combinatorial proof of the $q$-binomial theorem.  

\section{Kanade and Russell's First Four Conjectures and Some Missing Cases}
\label{secKR1-4}

\begin{theorem}[cf. Kanade-Russell conjecture $I_1$]
\label{thmKR1}
  For $n, m \in \mathbb{N}$, 
  let $kr_1(n, m)$ be the number of partitions of $n$ into $m$ parts 
  with difference at least three at distance two such that 
  if two successive parts differ by at most one, 
  then their sum is divisible by 3.  
  Then, 
  \begin{align}
  \label{eqGF_KR1}
    \sum_{n, m \geq 0} kr_1(n, m) q^n x^m 
    = \sum_{n_1, n_2 \geq 0} \frac{ q^{3n_2^2 + n_1^2 + 3n_1n_2} x^{2n_2 + n_1} }
      { (q; q)_{n_1} (q^3; q^3)_{n_2} }.  
  \end{align}
\end{theorem}

\begin{proof}
 For any $\lambda$ enumerated by $kr_1(n, m)$, 
 we will construct a unique triple $(\beta, \mu, \eta)$ 
 meeting the following criteria.  
 \begin{itemize}
  \item 
    $\beta$ is the base partition into $m = 2n_2 + n_1$ parts 
    having $n_2$ 2-clusters and $n_1$ 1-clusters.  
    $\beta$ satisfies the difference conditions set forth by $kr_1(n, m)$.  
  \item 
    $\mu$ is a partition with $n_1$ parts (counting zeros).  
  \item 
    $\eta$ is a partition into multiples of three with $n_2$ parts (counting zeros).   
  \item 
    $\vert \lambda \vert = \vert \beta \vert + \vert \mu \vert + \vert \eta \vert$.  
 \end{itemize}
 Conversely, given a triple $(\beta, \mu, \eta)$ as described above, 
 we will construct a unique $\lambda$ counted by $kr_1(n, m)$, 
 where $m = 2n_2 + n_1$.  
 We will arrange constructions so that they are inverses of each other at each step.  
 This will give a 1-1 correspondence between the said $\lambda$ and $(\beta, \mu, \eta)$, 
 yielding
 \begin{align}
 \label{eqGFtr1}
  \sum_{n, m \geq 0} kr_1(n, m) q^n x^m 
  = \sum_{n_1, n_2 \geq 0} q^{\vert \beta \vert} x^{l(\beta)} 
    \sum_{\mu, \eta} q^{\vert \mu \vert + \vert \eta \vert}.  
 \end{align}
 
 $\beta$ is the partition with $n_2$ 2-clusters, $n_1$ 1-clusters, and having the smallest possible weight.  
 Notice that $\lambda$ cannot have $r$-clusters for $r \geq 3$, 
 since the existence of an $r$-cluster requires the existence of an $r$-marked part, 
 hence difference at most one at distance $r-1$.  
 
 In building $\beta$, we will place the 1- and 2-clusters, which are as small as possible, 
 one after the other without violating the difference conditions.  
 The 2-clusters may look like
 \begin{align*}
  \left\{ 
  \begin{array}{cc} \\ ( \textrm{ parts } \leq 3k-3) \end{array}
  \begin{array}{cc} 3k \\ 3k \end{array}
  \begin{array}{cc} \\ ( \textrm{ parts } \geq 3k+3) \end{array}
  \right\} 
  \begin{array}{cc} \\ , \end{array}
 \end{align*}
 or 
 \begin{align*}
  \left\{ 
  \begin{array}{cc} \\ ( \textrm{ parts } \leq 3k-1) \end{array}
  \begin{array}{cc} & 3k+2 \\ 3k+1 & \end{array}
  \begin{array}{cc} \\ ( \textrm{ parts } \geq 3k+4) \end{array}
  \right\} 
  \begin{array}{cc} \\ , \end{array}
 \end{align*}
 but not
 \begin{align*}
  \begin{array}{cc} 3k+1 \\ 3k+1 \end{array}
  \begin{array}{cc} \\ , \end{array}
  \begin{array}{cc} 3k+2 \\ 3k+2 \end{array}
  \begin{array}{cc} \\ , \end{array}
  \begin{array}{cc} & 3k+1 \\ 3k & \end{array}
  \begin{array}{cc} \\ \textrm{ , or } \end{array}
  \begin{array}{cc} & 3k+3 \\ 3k+2 & \end{array}
  \begin{array}{cc} \\ . \end{array}
 \end{align*}
 In the first two cases, 
 the sum of two successive displayed parts is divisible by 3.  
 In the last four, it is not.  
 
 One can check that the minimal weight of $\beta$ is attained 
 when all 2-clusters are smaller than the 1-clusters, 
 and all clusters are as small as possible.  
 We will give indications of this fact in the course of the proof.  
 Thus, $\beta$ is 
 \begin{align}
 \label{baseptnKR1}
  \left\{
    \begin{array}{cc} & 2 \\ 1 & \end{array}
    \begin{array}{cc} & 5 \\ 4 & \end{array}
    \begin{array}{cc} \\ \cdots \end{array}
    \begin{array}{cc} & 3n_2-1 \\ 3n_2-2 & \end{array}
    \begin{array}{cc} \\ 3n_2+1 \end{array}
    \begin{array}{cc} \\ 3n_2+3 \end{array}
    \begin{array}{cc} \\ \cdots \end{array}
    \begin{array}{cc} \\ 3n_2+2n_1-1\end{array}
  \right\}
    \begin{array}{cc} \\ . \end{array}
 \end{align}
 Here, $n_1, n_2 \geq 0$.  
 The weight of $\beta$ is 
 \begin{align*}
  \vert \beta \vert 
  = & [(1 + 2) + (4 + 5) + \cdots + ((3n_2-2) + (3n_2 - 1))] \\
  & + [ (3n_2 + 1) + (3n_2 + 1) + \cdots + (3n_2 + 2n_1 - 1) ] \\
  = & [3 + 9 + \cdots + 3(2n_2 - 1)] + 3n_2n_1 + n_1^2 \\
  = & 3n_2^2 + n_1^2 + 3n_2n_1.  
 \end{align*}
 Clearly, $\mu$ is generated by $1/(q; q)_{n_1}$, 
 and $\eta$ by $1/(q^3; q^3)_{n_2}$, so that 
 \begin{align}
 \label{eqGFtr2}
  \sum_{n_1, n_2 \geq 0} q^{\vert \beta \vert} x^{l(\beta)} 
    \sum_{\mu, \eta} q^{\vert \mu \vert + \vert \eta \vert} 
  = \sum_{n_1, n_2 \geq 0} \frac{ q^{3n_2^2 + n_1^2 + 3n_1n_2} x^{2n_2 + n_1} }
      { (q; q)_{n_1} (q^3; q^3)_{n_2} }.  
 \end{align}
 \eqref{eqGFtr1} and \eqref{eqGFtr2} prove the theorem.  
 
 Given a triple $(\beta, \mu, \eta)$, 
 we will first move the $i$th largest 1-cluster the $i$th largest part of $\mu$ times forward, 
 for $i = 1, 2, \ldots, n_1$, in this order.  
 And then, we move the $i$th largest 2-cluster $\frac{1}{3}\times$(the $i$th largest part of $\eta$) 
 times forward, for $i = 1, 2, \ldots, n_2$, in this order.  
 This will give us $\lambda$.  
 The forward and backward moves on the 2-clusters are not exactly 
 the forward or backward moves of the 2nd kind in Definitions \ref{defFwdMove}-\ref{defBackwdMove}.  
 
 Conversely, given $\lambda$, we first determine the number of 2- and 1-clusters, 
 $n_2$, and $n_1$, respectively.  
 We first move the $i$th smallest 2-cluster backward as many times as possible 
 for $i = 1, 2, \ldots, n_2$, in this order, 
 and record the number of moves as $\frac{1}{3}\eta_1$, $\frac{1}{3}\eta_2$, \ldots, $\frac{1}{3}\eta_{n_2}$.  
 Then we move the $i$th smallest 1-cluster backward as many times as possible 
 for $i = 1, 2, \ldots, n_1$, in this order, 
 and record the number of moves as $\mu_1$, $\mu_2$, \ldots, $\mu_{n_1}$.  
 Not only will we have obtained $\mu$ and $\eta$, but also $\beta$ in the end.  
 
 Notice that we perform the forward and backward moves in the exact reverse order.  
 
 Starting with $(\beta, \mu, \eta)$, 
 we simply add the $i$th largest part of $\mu$ to the $i$th largest 1-cluster in $\beta$.  
 This preserves the difference condition because the 1-clusters were at least two apart to start with, 
 and larger parts are added to larger 1-clusters, 
 keeping or increasing the gaps.  
 We now have the intermediate partition
 \begin{align}
 \label{intermPtnkr_1}
 \left\{
    \begin{array}{cc} & 2 \\ 1 & \end{array}
    \begin{array}{cc} & 5 \\ 4 & \end{array}
    \begin{array}{cc} \\ \cdots \end{array}
    \begin{array}{cc} & 3n_2-1 \\ 3n_2-2 & \end{array}
    \begin{array}{cc} \\ (\textrm{ parts } \geq 3n_2+1, \textrm{ all 1-clusters }) \end{array}
  \right\}
    \begin{array}{cc} \\ . \end{array}
 \end{align}
 This also adds the weight of $\mu$ to the weight of $\beta$.  
 
 We now describe the forward moves on the 2-clusters.  
 There are several cases.  
 \begin{align*}
  \left\{ 
  \begin{array}{cc} \\ ( \textrm{ parts } \leq 3k-1) \end{array}
  \begin{array}{cc} & \mathbf{3k+2} \\ \mathbf{3k+1} & \end{array}
  \begin{array}{cc} \\ ( \textrm{ parts } \geq 3k+6) \end{array}
  \right\} 
 \end{align*}
 \begin{align}
 \label{move_kr1_fwd2_m1}
  \big\downarrow \textrm{ one forward move on the displayed 2-cluster }
 \end{align}
 \begin{align*}
  \left\{ 
  \begin{array}{cc} \\ ( \textrm{ parts } \leq 3k-1) \end{array}
  \begin{array}{cc} \mathbf{3k+3} \\ \mathbf{3k+3} \end{array}
  \begin{array}{cc} \\ ( \textrm{ parts } \geq 3k+6) \end{array}
  \right\} 
 \end{align*}
 Here and elsewhere, 
 we highlight the cluster we move.  

 \begin{align*}
  \left\{ 
  \begin{array}{cc} \\ ( \textrm{ parts } \leq 3k-3) \end{array}
  \begin{array}{cc} \mathbf{3k} \\ \mathbf{3k} \end{array}
  \begin{array}{cc} \\ ( \textrm{ parts } \geq 3k+4) \end{array}
  \right\} 
 \end{align*}
 \begin{align}
 \label{move_kr1_fwd2_m2}
  \big\downarrow \textrm{ one forward move on the displayed 2-cluster }
 \end{align}
 \begin{align*}
  \left\{ 
  \begin{array}{cc} \\ ( \textrm{ parts } \leq 3k-3) \end{array}
  \begin{array}{cc} & \mathbf{3k+2} \\ \mathbf{3k+1} & \end{array}
  \begin{array}{cc} \\ ( \textrm{ parts } \geq 3k+4) \end{array}
  \right\} 
 \end{align*} 
 Observe that one forward move adds three to the weight of the intermediate partition.  
 This is why we require parts of $\eta$ to be multiples of three.  
 \begin{align*}
  \left\{ 
  \begin{array}{cc} \\ ( \textrm{ parts } \leq 3k-1) \end{array}
  \begin{array}{cc} & \mathbf{3k+2} \\ \mathbf{3k+1} & \end{array}
  \begin{array}{cc} \\ 3k+4  \end{array}
  \begin{array}{cc} \\ ( \textrm{ parts } \geq 3k+7) \end{array}
  \right\} 
 \end{align*}
 \begin{align*}
  \big\downarrow \textrm{ one forward move on the displayed 2-cluster }
 \end{align*}
 \begin{align*}
  \left\{ 
  \begin{array}{cc} \\ ( \textrm{ parts } \leq 3k-1) \end{array} \right.
  \underbrace{
  \begin{array}{cc} \mathbf{3k+3} \\ \mathbf{3k+3} \end{array}
  \begin{array}{cc} \\ 3k+4  \end{array}
  }_{!}
  \left. \begin{array}{cc} \\ ( \textrm{ parts } \geq 3k+7) \end{array} 
  \right\} \textrm{ (temporarily) }
 \end{align*}
 \begin{align*}
  \big\downarrow \textrm{ adjustment }
 \end{align*}
 \begin{align*}
  \left\{ 
  \begin{array}{cc} \\ ( \textrm{ parts } \leq 3k-1) \end{array} \right.
  \begin{array}{cc} \\ 3k+1  \end{array}
  \begin{array}{cc} & \mathbf{3k+5} \\ \mathbf{3k+4} & \end{array}
  \left. \begin{array}{cc} \\ ( \textrm{ parts } \geq 3k+7) \end{array} 
  \right\} 
 \end{align*}
 Notice that the adjustment does not change the weight, 
 and the terminal configuration satisfies the difference condition if the initial one does.  
 The adjustment here is simply subtracting three from the obstacle, namely the displayed 1-cluster, 
 and move the 2-cluster one more time forward as in \eqref{move_kr1_fwd2_m1} or \eqref{move_kr1_fwd2_m2}, 
 as if there are no obstacles.  
 
 There are four more cases in which a forward move on a 2-cluster is followed by one or more adjustments.  
 The idea is the same, so we skip the details.  
 \begin{align*}
  \left\{ 
  \begin{array}{cc} \\ ( \textrm{ parts } \leq 3k-1) \end{array}
  \begin{array}{cc} & \mathbf{3k+2} \\ \mathbf{3k+1} & \end{array}
  \begin{array}{cc} \\ 3k+4  \end{array}
  \begin{array}{cc} \\ 3k+6  \end{array}
  \begin{array}{cc} \\ ( \textrm{ parts } \geq 3k+9) \end{array}
  \right\} 
 \end{align*}
 \begin{align*}
  \big\downarrow \textrm{ one forward move on the displayed 2-cluster, followed by two adjustments }
 \end{align*}
 \begin{align*}
  \left\{ 
  \begin{array}{cc} \\ ( \textrm{ parts } \leq 3k-1) \end{array}
  \begin{array}{cc} \\ 3k+1  \end{array}
  \begin{array}{cc} \\ 3k+3  \end{array}
  \begin{array}{cc} \mathbf{3k+6} \\ \mathbf{3k+6} \end{array}
  \begin{array}{cc} \\ ( \textrm{ parts } \geq 3k+9) \end{array} 
  \right\} 
  \begin{array}{cc} \\ ,  \end{array}
 \end{align*}

 \begin{align*}
  \left\{ 
  \begin{array}{cc} \\ ( \textrm{ parts } \leq 3k-1) \end{array}
  \begin{array}{cc} & \mathbf{3k+2} \\ \mathbf{3k+1} & \end{array}
  \begin{array}{cc} \\ 3k+4  \end{array}
  \begin{array}{cc} \\ 3k+6  \end{array}
  \begin{array}{cc} \\ 3k+8  \end{array}
  \begin{array}{cc} \\ ( \textrm{ parts } \geq 3k+10) \end{array}
  \right\} 
 \end{align*}
 \begin{align*}
  \big\downarrow \textrm{ one forward move on the displayed 2-cluster, followed by three adjustments }
 \end{align*}
 \begin{align*}
  \left\{ 
  \begin{array}{cc} \\ ( \textrm{ parts } \leq 3k-1) \end{array}
  \begin{array}{cc} \\ 3k+1  \end{array}
  \begin{array}{cc} \\ 3k+3  \end{array}
  \begin{array}{cc} \\ 3k+5  \end{array}
  \begin{array}{cc} & \mathbf{3k+8} \\ \mathbf{3k+7} & \end{array}
  \begin{array}{cc} \\ ( \textrm{ parts } \geq 3k+10) \end{array} 
  \right\} 
  \begin{array}{cc} \\ ,  \end{array}
 \end{align*}

 \begin{align*}
  \left\{ 
  \begin{array}{cc} \\ ( \textrm{ parts } \leq 3k-3) \end{array}
  \begin{array}{cc} \mathbf{3k} \\ \mathbf{3k} \end{array}
  \begin{array}{cc} \\ 3k+3  \end{array}
  \begin{array}{cc} \\ ( \textrm{ parts } \geq 3k+6) \end{array}
  \right\} 
 \end{align*}
 \begin{align*}
  \big\downarrow \textrm{ one forward move on the displayed 2-cluster, followed by an adjustment }
 \end{align*}
 \begin{align*}
  \left\{ 
  \begin{array}{cc} \\ ( \textrm{ parts } \leq 3k-3) \end{array}
  \begin{array}{cc} \\ 3k  \end{array}
  \begin{array}{cc} \mathbf{3k+3} \\ \mathbf{3k+3} \end{array}
  \begin{array}{cc} \\ ( \textrm{ parts } \geq 3k+6) \end{array} 
  \right\} 
  \begin{array}{cc} \\ ,  \end{array}
 \end{align*}

 \begin{align*}
  \left\{ 
  \begin{array}{cc} \\ ( \textrm{ parts } \leq 3k-3) \end{array}
  \begin{array}{cc} \mathbf{3k} \\ \mathbf{3k} \end{array}
  \begin{array}{cc} \\ 3k+3  \end{array}
  \begin{array}{cc} \\ 3k+5  \end{array}
  \begin{array}{cc} \\ ( \textrm{ parts } \geq 3k+7) \end{array}
  \right\} 
 \end{align*}
 \begin{align*}
  \big\downarrow \textrm{ one forward move on the displayed 2-cluster, followed by two adjustments }
 \end{align*}
 \begin{align*}
  \left\{ 
  \begin{array}{cc} \\ ( \textrm{ parts } \leq 3k-3) \end{array}
  \begin{array}{cc} \\ 3k  \end{array}
  \begin{array}{cc} \\ 3k+2  \end{array}
  \begin{array}{cc} & \mathbf{3k+5} \\ \mathbf{3k+4} & \end{array}
  \begin{array}{cc} \\ ( \textrm{ parts } \geq 3k+7) \end{array} 
  \right\} 
  \begin{array}{cc} \\ .  \end{array}
 \end{align*}
 The above cases are excusive, there are no others.  
 One can easily verify that one forward move on the displayed 2-cluster 
 allows at least one forward move on the preceding 2-cluster.  
 Therefore, all parts of $\eta$ can be realized as forward moves on the 2-clusters, 
 registering the weight of $\eta$ on the weight of the intermediate partition.  
 In all cases above, the terminal configurations conform to the difference condition 
 provided that the respective initial configurations do.  
 This is due to the fact that the difference conditions can be checked locally 
 as the differences between successive parts, and as differences at distance two.  
 
 The final partition is the $\lambda$ we have been aiming at.  
 It is enumerated by $kr_1(n, m)$.  
 
 Now, given $\lambda$ counted by $kr_1(n, m)$, 
 having $n_2$ 2-clusters and $n_1$ 1-clusters, so that $m = 2n_2 + n_1$, 
 we will decompose it into the triple $(\beta, \mu, \eta)$ as described at the beginning of the proof.  
 
 We start by moving the smallest 2-cluster backward as many times as necessary to stow it as 
 \begin{align*}
  \left\{ 
  \begin{array}{cc} & \mathbf{2} \\ \mathbf{1} & \end{array}
  \begin{array}{cc} \\ ( \textrm{ parts } \geq 4) \end{array} 
  \right\} 
  \begin{array}{cc} \\ .  \end{array}
 \end{align*}
 We record the number of moves as $\frac{1}{3}\eta_1$, which gives us the first part of $\eta$.  
 If the smallest 2-cluster is already $\begin{array}{cc} & 2 \\ 1 & \end{array}$, 
 we set $\eta_1 = 0$.  
 
 We need to describe the backward moves on the 2-clusters.  
 Again, there are several cases.  
 \begin{align*}
  \left\{ 
  \begin{array}{cc} \\ ( \textrm{ parts } \leq 3k-4) \end{array}
  \begin{array}{cc} \mathbf{3k} \\ \mathbf{3k} \end{array}
  \begin{array}{cc} \\ ( \textrm{ parts } \geq 3k+3) \end{array}
  \right\} 
 \end{align*}
 \begin{align}
 \label{move_kr1_fwd2_m3}
  \big\downarrow \textrm{ one backward move on the displayed 2-cluster }
 \end{align}
 \begin{align*}
  \left\{ 
  \begin{array}{cc} \\ ( \textrm{ parts } \leq 3k-4) \end{array}
  \begin{array}{cc} & \mathbf{3k-1} \\ \mathbf{3k-2} & \end{array}
  \begin{array}{cc} \\ ( \textrm{ parts } \geq 3k+3) \end{array}
  \right\} 
  \begin{array}{cc} \\ , \end{array}
 \end{align*}

 \begin{align*}
  \left\{ 
  \begin{array}{cc} \\ ( \textrm{ parts } \leq 3k-3) \end{array}
  \begin{array}{cc} & \mathbf{3k+2} \\ \mathbf{3k+1} & \end{array}
  \begin{array}{cc} \\ ( \textrm{ parts } \geq 3k+4) \end{array}
  \right\} 
 \end{align*}
 \begin{align}
 \label{move_kr1_fwd2_m4}
  \big\downarrow \textrm{ one backward move on the displayed 2-cluster }
 \end{align}
 \begin{align*}
  \left\{ 
  \begin{array}{cc} \\ ( \textrm{ parts } \leq 3k-3) \end{array}
  \begin{array}{cc} \mathbf{3k} \\ \mathbf{3k} \end{array}
  \begin{array}{cc} \\ ( \textrm{ parts } \geq 3k+4) \end{array}
  \right\} 
  \begin{array}{cc} \\ .  \end{array}
 \end{align*} 
 Clearly, one backward move on a 2-cluster decreases the weight of $\lambda$ by three, 
 which is registered in parts of $\eta$.  
 Thus, parts of $\eta$ are evidently multiples of 3.  
 \begin{align*}
  \left\{ 
  \begin{array}{cc} \\ ( \textrm{ parts } \leq 3k-4) \end{array} \right.
  \begin{array}{cc} \\ 3k-2  \end{array}
  \begin{array}{cc} & \mathbf{3k+2} \\ \mathbf{3k+1} & \end{array}
  \left. \begin{array}{cc} \\ ( \textrm{ parts } \geq 3k+4) \end{array} 
  \right\} 
 \end{align*}
 \begin{align*}
  \big\downarrow \textrm{ one backward move on the displayed 2-cluster }
 \end{align*}
 \begin{align*}
  \left\{ 
  \begin{array}{cc} \\ ( \textrm{ parts } \leq 3k-4) \end{array} \right.
  \underbrace{
  \begin{array}{cc} \\ 3k-2  \end{array}
  \begin{array}{cc} \mathbf{3k} \\ \mathbf{3k} \end{array}
  }_{!}
  \left. \begin{array}{cc} \\ ( \textrm{ parts } \geq 3k+4) \end{array} 
  \right\} \textrm{ (temporarily) }
 \end{align*}
 \begin{align*}
  \big\downarrow \textrm{ adjustment }
 \end{align*}
 \begin{align*}
  \left\{ 
  \begin{array}{cc} \\ ( \textrm{ parts } \leq 3k-4) \end{array}
  \begin{array}{cc} & \mathbf{3k-1} \\ \mathbf{3k-2} & \end{array}
  \begin{array}{cc} \\ 3k+1  \end{array}
  \begin{array}{cc} \\ ( \textrm{ parts } \geq 3k+4) \end{array}
  \right\} 
 \end{align*}
 Again, the adjustment does not alter the weight of the partition.  
 It only resolves the violation of the difference condition by moving 
 the temporarily problematic 1-cluster three times forward, 
 and the temporarily problematic 2-cluster one time backward 
 as in \eqref{move_kr1_fwd2_m3} or \eqref{move_kr1_fwd2_m4} as if there are no obstacles.  
 The terminal partition satisfies the difference conditions if the initial one does.  
 Recall that we assume the initial partitions always satisfy the respective difference conditions.  
 
 There are four more cases.  We omit the intermediate steps, 
 since they are completely analogous to the above case.  
 \begin{align*}
  \left\{ 
  \begin{array}{cc} \\ ( \textrm{ parts } \leq 3k-7) \end{array}
  \begin{array}{cc} \\ 3k-5  \end{array}
  \begin{array}{cc} \\ 3k-3  \end{array}
  \begin{array}{cc} \mathbf{3k} \\ \mathbf{3k} \end{array}
  \begin{array}{cc} \\ ( \textrm{ parts } \geq 3k+3) \end{array} 
  \right\} 
 \end{align*}
 \begin{align*}
  \big\downarrow \textrm{ one backward move on the displayed 2-cluster, followed by two adjustments }
 \end{align*}
 \begin{align*}
  \left\{ 
  \begin{array}{cc} \\ ( \textrm{ parts } \leq 3k-7) \end{array}
  \begin{array}{cc} & \mathbf{3k-4} \\ \mathbf{3k-5} & \end{array}
  \begin{array}{cc} \\ 3k-2  \end{array}
  \begin{array}{cc} \\ 3k  \end{array}
  \begin{array}{cc} \\ ( \textrm{ parts } \geq 3k+3) \end{array}
  \right\} 
  \begin{array}{cc} \\ ,  \end{array}
 \end{align*}

 \begin{align*}
  \left\{ 
  \begin{array}{cc} \\ ( \textrm{ parts } \leq 3k-7) \end{array}
  \begin{array}{cc} \\ 3k-5  \end{array}
  \begin{array}{cc} \\ 3k-3  \end{array}
  \begin{array}{cc} \\ 3k-1  \end{array}
  \begin{array}{cc} & \mathbf{3k+2} \\ \mathbf{3k+1} & \end{array}
  \begin{array}{cc} \\ ( \textrm{ parts } \geq 3k+4) \end{array} 
  \right\} 
 \end{align*}
 \begin{align*}
  \big\downarrow \textrm{ one backward move on the displayed 2-cluster, followed by three adjustments }
 \end{align*}
 \begin{align*}
  \left\{ 
  \begin{array}{cc} \\ ( \textrm{ parts } \leq 3k-7) \end{array}
  \begin{array}{cc} & \mathbf{3k-4} \\ \mathbf{3k-5} & \end{array}
  \begin{array}{cc} \\ 3k-2  \end{array}
  \begin{array}{cc} \\ 3k  \end{array}
  \begin{array}{cc} \\ 3k+2  \end{array}
  \begin{array}{cc} \\ ( \textrm{ parts } \geq 3k+4) \end{array}
  \right\} 
  \begin{array}{cc} \\ ,  \end{array}
 \end{align*}

 \begin{align*}
  \left\{ 
  \begin{array}{cc} \\ ( \textrm{ parts } \leq 3k-6) \end{array}
  \begin{array}{cc} \\ 3k - 3  \end{array}
  \begin{array}{cc} \mathbf{3k} \\ \mathbf{3k} \end{array}
  \begin{array}{cc} \\ ( \textrm{ parts } \geq 3k+3) \end{array} 
  \right\} 
 \end{align*}
 \begin{align*}
  \big\downarrow \textrm{ one backward move on the displayed 2-cluster, followed by an adjustment }
 \end{align*}
 \begin{align*}
  \left\{ 
  \begin{array}{cc} \\ ( \textrm{ parts } \leq 3k-6) \end{array}
  \begin{array}{cc} \mathbf{3k-3} \\ \mathbf{3k-3} \end{array}
  \begin{array}{cc} \\ 3k  \end{array}
  \begin{array}{cc} \\ ( \textrm{ parts } \geq 3k+3) \end{array}
  \right\} 
  \begin{array}{cc} \\ ,  \end{array}
 \end{align*}

 \begin{align*}
  \left\{ 
  \begin{array}{cc} \\ ( \textrm{ parts } \leq 3k-6) \end{array}
  \begin{array}{cc} \\ 3k-3  \end{array}
  \begin{array}{cc} \\ 3k-1  \end{array}
  \begin{array}{cc} & \mathbf{3k+2} \\ \mathbf{3k+1} & \end{array}
  \begin{array}{cc} \\ ( \textrm{ parts } \geq 3k+4) \end{array} 
  \right\} 
 \end{align*}
 \begin{align*}
  \big\downarrow \textrm{ one backward move on the displayed 2-cluster, followed by two adjustments }
 \end{align*}
 \begin{align*}
  \left\{ 
  \begin{array}{cc} \\ ( \textrm{ parts } \leq 3k-6) \end{array}
  \begin{array}{cc} \mathbf{3k-3} \\ \mathbf{3k-3} \end{array}
  \begin{array}{cc} \\ 3k  \end{array}
  \begin{array}{cc} \\ 3k+2  \end{array}
  \begin{array}{cc} \\ ( \textrm{ parts } \geq 3k+4) \end{array}
  \right\} 
  \begin{array}{cc} \\ .  \end{array}
 \end{align*}
 The above cases exhaust all possibilities.  
 One can verify that the 2-cluster succeeding the displayed one may be moved at least once backward 
 after the described backward move.  
 Once the smallest 2-cluster is stowed as $\begin{array}{cc} & 2 \\ 1 & \end{array}$, 
 we continue with the next smallest 2-cluster.  
 We move it backward as many times as possible and place it as $\begin{array}{cc} & 5 \\ 4 & \end{array}$, 
 recording the number of moves as $\frac{1}{3}\eta_2$.  
 Then, continue with the next smallest 2-cluster, etc., obtaining $\eta$.    
 The above discussion ensures that $\eta_1$ $\leq \eta_2$ $\leq \cdots$ $\leq \eta_{n_2}$.  
 
 The careful reader will have noticed that the respective cases for the backward moves 
 and the forward moves on the 2-clusters have swapped initial and terminal configurations.  
 The forward and backward moves are inverses of each other in this sense.  
 
 Once the 2-clusters are lined up as in \eqref{intermPtnkr_1} and we have $\eta$, 
 we subtract $\mu_1$ from the smallest 1-cluster to make it $3n_2+1$, 
 $\mu_2$ from the next smallest to make it $3n_2 + 3$, etc.  
 This way, we will have constructed $\mu$.  
 Because the successive 1-clusters are at least two apart by Gordon marking, 
 $\mu_1$ $\leq \mu_2$ $\leq \cdots$ $\leq \mu_{n_1}$
 Subtracting $\mu_i$ from the $i$th smallest 1-cluster is nothing but performing 
 $\mu_i$ backward moves on it.  
 The forward and backward moves on the 1-clusters are obviously inverses of each other.  
 
 The remaining partition is \eqref{baseptnKR1}, namely the base partition $\beta$.  
 
 This justifies \eqref{eqGFtr1}, therefore concludes the proof.  
\end{proof}

As in other similar proofs, 
one can make the forward and backward moves on the 1- or 2-clusters exact opposites of each other, 
together with the temporary rulebreaking in the middle.  
However, we find the descriptions in the proofs more appealing.  

% begin example 
  
  {\bf Example: } 
  Using the notation in the above proof, 
  we will work in the forward direction, 
  and construct the partition $\lambda$ 
  having $n_1 = 3$ 1-clusters, $n_2 = 2$ 2-clusters, 
  with $\mu = 0+1+1$, and $\eta = 3+6$.  
  We start with $\beta$ is in the form \eqref{baseptnKR1}.  
  \begin{align*}
    \beta = 
   \left\{
   \begin{array}{cc} & 2 \\ 1 & \end{array}
   \begin{array}{cc} & 5 \\ 4 & \end{array}
   \begin{array}{cc} \\ 7 \end{array}
   \begin{array}{cc} \\ 9 \end{array}
   \begin{array}{cc} \\ \mathbf{11} \end{array}
   \right\}
  \end{align*}
  Applying $\mu$ first, we obtain
  \begin{align*}
   \left\{
   \begin{array}{cc} & 2 \\ 1 & \end{array}
   \begin{array}{cc} & \mathbf{5} \\ \mathbf{4} & \end{array}
   \begin{array}{cc} \\ 7 \end{array}
   \begin{array}{cc} \\ 10 \end{array}
   \begin{array}{cc} \\ 12 \end{array}
   \right\}
   \begin{array}{cc} \\ . \end{array}
  \end{align*}
  Then, we continue with incorporating $\eta$, 
  first $\frac{1}{3}\times$ its largest part as forward moves on the largest 2-cluster.  
  \begin{align*}
   \big\downarrow \textrm{ the first forward move on the larger 2-cluster }
  \end{align*}
  \begin{align*}
   \left\{
   \begin{array}{cc} & 2 \\ 1 & \end{array} \right.
   \underbrace{
   \begin{array}{cc} \mathbf{6} \\ \mathbf{6} \end{array}
   \begin{array}{cc} \\ 7 \end{array}
   }_{!}
   \begin{array}{cc} \\ 10 \end{array}
   \left. \begin{array}{cc} \\ 12 \end{array}
   \right\}
  \end{align*}
  \begin{align*}
   \big\downarrow \textrm{ adjustment }
  \end{align*}
  \begin{align*}
   \left\{
   \begin{array}{cc} & 2 \\ 1 & \end{array} \right.
   \begin{array}{cc} \\ 4 \end{array}
   \begin{array}{cc} & \mathbf{8} \\ \mathbf{7} & \end{array}
   \begin{array}{cc} \\ 10 \end{array}
   \left. \begin{array}{cc} \\ 12 \end{array}
   \right\}
  \end{align*}
  \begin{align*}
   \big\downarrow \textrm{ one more forward move on the larger 2-cluster }
  \end{align*}
  \begin{align*}
   \left\{
   \begin{array}{cc} & 2 \\ 1 & \end{array} \right.
   \begin{array}{cc} \\ 4 \end{array}
   \underbrace{
   \begin{array}{cc} \mathbf{9} \\ \mathbf{9} \end{array}
   \begin{array}{cc} \\ 10 \end{array}
   }_{!}
   \left. \begin{array}{cc} \\ 12 \end{array}
   \right\}
  \end{align*}
  \begin{align*}
   \big\downarrow \textrm{ adjustment }
  \end{align*}
  \begin{align*}
   \left\{
   \begin{array}{cc} & 2 \\ 1 & \end{array} \right.
   \begin{array}{cc} \\ 4 \end{array}
   \begin{array}{cc} \\ 7 \end{array}
   \underbrace{
   \begin{array}{cc} & \mathbf{11} \\ \mathbf{10} & \end{array}
   \begin{array}{cc} \\ 12 \end{array}
   }_{!}
   \left. \begin{array}{cc} \\ \phantom{0} \end{array} \right\}
  \end{align*}
  \begin{align*}
   \big\downarrow \textrm{ adjustment }
  \end{align*}
  \begin{align*}
   \left\{
   \begin{array}{cc} & \mathbf{2} \\ \mathbf{1} & \end{array} \right.
   \begin{array}{cc} \\ 4 \end{array}
   \begin{array}{cc} \\ 7 \end{array}
   \begin{array}{cc} \\ 9 \end{array}
   \left. \begin{array}{cc} 12 \\ 12 \end{array}
   \right\}
  \end{align*}
  This finishes the $\frac{1}{3}\eta_2 = 2$ forward moves on the larger 2-cluster.  
  We continue with $\frac{1}{3}\eta_1 = 1$ forward move on the smaller 2-cluster.  
  \begin{align*}
   \left\{ \begin{array}{cc} \\ \phantom{0} \end{array} \right.
   \underbrace{
   \begin{array}{cc} \mathbf{3} \\ \mathbf{3} \end{array} 
   \begin{array}{cc} \\ 4 \end{array}
   }_{!}
   \begin{array}{cc} \\ 7 \end{array}
   \begin{array}{cc} \\ 9 \end{array}
   \left. \begin{array}{cc} 12 \\ 12 \end{array}
   \right\}
  \end{align*}
  \begin{align*}
    \big\downarrow \textrm{ adjustment }
  \end{align*}
  \begin{align*}
    \lambda = 
   \left\{ 
   \begin{array}{cc} \\ 1 \end{array}
   \begin{array}{cc} & \mathbf{5} \\ \mathbf{4} & \end{array} 
   \begin{array}{cc} \\ 7 \end{array}
   \begin{array}{cc} \\ 9 \end{array}
   \begin{array}{cc} 12 \\ 12 \end{array}
   \right\}
  \end{align*}
  As expected, 
  \begin{align*}
   \vert \beta \vert + \vert \mu \vert + \vert \eta \vert
   = 39 + 2 + 9 = 50
   = \vert \lambda \vert.  
  \end{align*}
  
% end example

\begin{theorem}[cf. Kanade-Russell conjecture $I_2$]
\label{thmKR2}
  For $n, m \in \mathbb{N}$, 
  let $kr_2(n, m)$ be the number of partitions of $n$ into $m$ parts 
  with smallest part at least two, and difference at least three at distance two 
  such that if two successive parts differ by at most one, 
  then their sum is divisible by three.  
  Then, 
  \begin{align}
  \label{eqGF_KR2}
    \sum_{n, m \geq 0} kr_2(n, m) q^n x^m 
    = \sum_{n_1, n_2 \geq 0} \frac{ q^{3n_2^2 + 3n_2 + n_1^2 + n_1 + 3n_1n_2} x^{2n_2 + n_1} }
      { (q; q)_{n_1} (q^3; q^3)_{n_2} }.  
  \end{align}
\end{theorem}

\begin{proof}
 The proof is completely analogous to the proof of Theorem \ref{thmKR1}, 
 except that we have to use two different base partitions $\beta$ 
 for the cases $n_1 = 0$ and $n_1 > 0$.  
 When $n_1 = 0$, the base partition is clearly
 \begin{align}
 \label{basePtnKR2No1}
  \left\{
    \begin{array}{cc} 3 \\ 3 \end{array}
    \begin{array}{cc} 6 \\ 6 \end{array}
    \begin{array}{cc} \\ \cdots \end{array}
    \begin{array}{cc} 3n_2 \\ 3n_2 \end{array}
  \right\}
    \begin{array}{cc} \\ , \end{array}
 \end{align}
 with weight $3n_2^2 + 3n_2$.  
 If, however, $n_1 > 0$, that is, there is at least one 1-cluster, 
 the seemingly obvious choice 
 \begin{align}
 \label{basePtnKR2No2wrong}
  \left\{
    \begin{array}{cc} 3 \\ 3 \end{array}
    \begin{array}{cc} 6 \\ 6 \end{array}
    \begin{array}{cc} \\ \cdots \end{array}
    \begin{array}{cc} 3n_2 \\ 3n_2 \end{array}
    \begin{array}{cc} \\ 3n_2 + 3 \end{array}
    \begin{array}{cc} \\ 3n_2 + 5 \end{array}
    \begin{array}{cc} \\ \cdots \end{array}
    \begin{array}{cc} \\ 3n_2 + 2n_1 + 1 \end{array}
  \right\}
 \end{align}
 does not have minimal weight.  
 Moreover, one can never obtain a partition counted by $kr_2(n, m)$ 
 containing the part 2 this way.  
 The correct base partition in this case is
 \begin{align}
 \label{basePtnKR2No2correct}
  \left\{
    \begin{array}{cc} \\ 2 \end{array}
    \begin{array}{cc} & 5 \\ 4 & \end{array}
    \begin{array}{cc} & 8 \\ 7 & \end{array}
    \begin{array}{cc} \\ \cdots \end{array}
    \begin{array}{cc} & 3n_2+2 \\ 3n_2+1 & \end{array}
    \begin{array}{cc} \\ 3n_2 + 4 \end{array}
    \begin{array}{cc} \\ 3n_2 + 6 \end{array}
    \begin{array}{cc} \\ \cdots \end{array}
    \begin{array}{cc} \\ 3n_2 + 2n_1 \end{array}
  \right\}
    \begin{array}{cc} \\ , \end{array}
 \end{align}
 for $n_1 > 0$.  
 One can check that \eqref{basePtnKR2No2correct} has smaller weight than \eqref{basePtnKR2No2wrong}, 
 and that any other lineup of 2- and 1-clusters results in a greater weight.  
 \eqref{basePtnKR2No2correct} has weight $3n_2^2 + 3n_2 + n_1^2 + n_1 + 3n_2n_1$, 
 the $n_1 = 0$ case of which yields the weight of \eqref{basePtnKR2No1}.  
 
 There is one more twist before we leave the rest of the proof to the reader.  
 We need to discuss how the smallest 1-cluster can move forward.  
 Recall that in the proof of Theorem \ref{thmKR1}, 
 in order for the smallest one cluster to move forward, 
 each of the other 1-clusters must have moved forward at least once.  
 It is the same here, 
 so we assume that all but the smallest 1-clusters, if any, 
 have moved in \eqref{basePtnKR2No2correct}.  
 This yields the configuration below.  
 \begin{align*}
  \left\{
    \begin{array}{cc} \\ \mathbf{2} \end{array}
    \begin{array}{cc} & 5 \\ 4 & \end{array}
    \begin{array}{cc} & 8 \\ 7 & \end{array}
    \begin{array}{cc} \\ \cdots \end{array}
    \begin{array}{cc} & 3n_2+2 \\ 3n_2+1 & \end{array}
    \begin{array}{cc} \\ 3n_2 + 5 \end{array}
    \begin{array}{cc} \\ 3n_2 + 7 \end{array}
    \begin{array}{cc} \\ \cdots \end{array}
    \begin{array}{cc} \\ 3n_2 + 2n_1 + 1 \end{array}
  \right\}
 \end{align*}
 Now we want to move the smallest 1-cluster forward once.  
 This will entail \emph{prestidigitation} of the smallest 1-cluster 
 through the 2-clusters  (please see section \ref{secConclusion} and the appendix).  
 \begin{align*}
  \left\{ \begin{array}{cc} \\ \phantom{0} \end{array} \right.
    \underbrace{
    \begin{array}{cc} \\ \mathbf{3} \end{array}
    \begin{array}{cc} & 5 \\ 4 & \end{array}
    }_{!}
    \begin{array}{cc} & 8 \\ 7 & \end{array}
    \begin{array}{cc} \\ \cdots \end{array}
    \begin{array}{cc} & 3n_2+2 \\ 3n_2+1 & \end{array}
    \begin{array}{cc} \\ 3n_2 + 5 \end{array}
    \begin{array}{cc} \\ 3n_2 + 7 \end{array}
    \begin{array}{cc} \\ \cdots \end{array}
    \left. \begin{array}{cc} \\ 3n_2 + 2n_1 + 1 \end{array}
  \right\}
 \end{align*}
 \begin{align*}
  \big\downarrow \textrm{ adjustment }
 \end{align*}
 \begin{align*}
  \left\{ \begin{array}{cc} \\ \phantom{0} \end{array} \right.
    \begin{array}{cc} 3 \\ 3 \end{array}
    \underbrace{
    \begin{array}{cc} \\ \mathbf{6} \end{array}
    \begin{array}{cc} & 8 \\ 7 & \end{array}
    }_{!}
    \begin{array}{cc} \\ \cdots \end{array}
    \begin{array}{cc} & 3n_2+2 \\ 3n_2+1 & \end{array}
    \begin{array}{cc} \\ 3n_2 + 5 \end{array}
    \begin{array}{cc} \\ 3n_2 + 7 \end{array}
    \begin{array}{cc} \\ \cdots \end{array}
    \left. \begin{array}{cc} \\ 3n_2 + 2n_1 + 1 \end{array}
  \right\}
 \end{align*}
 \begin{align*}
  \big\downarrow n_2 - 1 \textrm{ more adjustments in a similar fashion }
 \end{align*}
 \begin{align*}
  \left\{ 
    \begin{array}{cc} 3 \\ 3 \end{array}
    \begin{array}{cc} 6 \\ 6 \end{array}
    \begin{array}{cc} \\ \cdots \end{array}
    \begin{array}{cc} 3n_2 \\ 3n_2 \end{array}
    \begin{array}{cc} \\ \mathbf{3n_2 + 3} \end{array}
    \begin{array}{cc} \\ 3n_2 + 5 \end{array}
    \begin{array}{cc} \\ 3n_2 + 7 \end{array}
    \begin{array}{cc} \\ \cdots \end{array}
    \begin{array}{cc} \\ 3n_2 + 2n_1 + 1 \end{array}
  \right\}
    \begin{array}{cc} \\ , \end{array}
 \end{align*}
 incidentally arriving at \eqref{basePtnKR2No2wrong}, 
 the weight of which is exactly $n_1$ more than that of \eqref{basePtnKR2No2correct},  
 for this reason.  
 
 As in the proof of Theorem \ref{thmKR1}, after the backward moves on the 2-clusters 
 making the intermediate partition
 \begin{align*}
  \left\{ 
    \begin{array}{cc} 3 \\ 3 \end{array}
    \begin{array}{cc} 6 \\ 6 \end{array}
    \begin{array}{cc} \\ \cdots \end{array}
    \begin{array}{cc} 3n_2 \\ 3n_2 \end{array}
    \begin{array}{cc} \\ ( \textrm{ parts } \geq 3n_2+3, \textrm{ all 1-clusters } ) \end{array}
  \right\}
    \begin{array}{cc} \\ . \end{array}
 \end{align*}
 We first move the smallest 1-cluster so as to bring it back to $3n_2 + 3$, 
 recorging the number of moves as $\mu_1 - 1$.  
 Now the intermediate partition looks like 
 \begin{align*}
  \left\{ 
    \begin{array}{cc} 3 \\ 3 \end{array}
    \begin{array}{cc} 6 \\ 6 \end{array}
    \begin{array}{cc} \\ \cdots \end{array}
    \begin{array}{cc} 3n_2 \\ 3n_2 \end{array}
    \begin{array}{cc} \\ \mathbf{3n_2 + 3} \end{array}
    \begin{array}{cc} \\ ( \textrm{ parts } \geq 3n_2+5, \textrm{ all 1-clusters } ) \end{array}
  \right\}
    \begin{array}{cc} \\ . \end{array}
 \end{align*}
 The final backward move on the smallest 1-cluster will again entail 
 prestidigitation of the smallest 1-cluster through the 2-clusters.  
 \begin{align*}
  \left\{ \begin{array}{cc} \\ \phantom{0} \end{array} \right.
    \begin{array}{cc} 3 \\ 3 \end{array}
    \begin{array}{cc} 6 \\ 6 \end{array}
    \begin{array}{cc} \\ \cdots \end{array}
    \underbrace{
    \begin{array}{cc} 3n_2 \\ 3n_2 \end{array}
    \begin{array}{cc} \\ \mathbf{3n_2 + 2} \end{array}
    }_{!}
    \left. \begin{array}{cc} \\ ( \textrm{ parts } \geq 3n_2+5) \end{array}
  \right\}
 \end{align*}
 \begin{align*}
  \big\downarrow \textrm{ adjustment }
 \end{align*}
 \begin{align*}
  \left\{ \begin{array}{cc} \\ \phantom{0} \end{array} \right.
    \begin{array}{cc} 3 \\ 3 \end{array}
    \begin{array}{cc} 6 \\ 6 \end{array}
    \begin{array}{cc} \\ \cdots \end{array}
    \underbrace{
    \begin{array}{cc} 3n_2-3 \\ 3n_2-3 \end{array}
    \begin{array}{cc} \\ \mathbf{3n_2 - 1} \end{array}
    }_{!}
    \begin{array}{cc} & 3n_2+2 \\ 3n_2+1 & \end{array}
    \left. \begin{array}{cc} \\ ( \textrm{ parts } \geq 3n_2+5) \end{array}
  \right\}
 \end{align*}
 \begin{align*}
  \big\downarrow \textrm{ after } n_2-1 \textrm{ more adjustments of similar sort }
 \end{align*}
 \begin{align*}
  \left\{ \begin{array}{cc} \\ \phantom{0} \end{array} \right.
    \begin{array}{cc} \\ \mathbf{2} \end{array}
    \begin{array}{cc} & 5 \\ 4 & \end{array}
    \begin{array}{cc} & 8 \\ 7 & \end{array}
    \begin{array}{cc} \\ \cdots \end{array}
    \begin{array}{cc} & 3n_2+2 \\ 3n_2+1 & \end{array}
    \left. \begin{array}{cc} \\ ( \textrm{ parts } \geq 3n_2+5) \end{array}
  \right\}
    \begin{array}{cc} \\ . \end{array}
 \end{align*}
 As far as the lineup of the smallest 1-cluster and all the 2-clusters is concerned, 
 the initial and terminal partitions are swapped in the forward and the backward moves.  
 Also, notice that this extra move on the smallest 1-cluster opens room for 
 the larger 1-clusters to move backward at least once more.  
 The remaining parts of the proof are completely analogous 
 to those parts of the proof of Theorem \ref{thmKR1}.  
\end{proof}

\begin{theorem}[cf. Kanade-Russell conjecture $I_3$]
\label{thmKR3}
  For $n, m \in \mathbb{N}$, 
  let $kr_3(n, m)$ be the number of partitions of $n$ into $m$ parts 
  with smallest part at least three, and difference at least three at distance two 
  such that if two successive parts differ by at most one, 
  then their sum is divisible by three.  
  Then, 
  \begin{align}
  \label{eqGF_KR3}
    \sum_{n, m \geq 0} kr_3(n, m) q^n x^m 
    = \sum_{n_1, n_2 \geq 0} \frac{ q^{3n_2^2 + 3n_2 + n_1^2 + 2n_1 + 3n_1n_2} x^{2n_2 + n_1} }
      { (q; q)_{n_1} (q^3; q^3)_{n_2} }.  
  \end{align}
\end{theorem}

\begin{proof}
 The proof of Theorem \ref{thmKR1} applies \emph{mutatis mutandis}.  
 The only difference being the base partition $\beta$: 
 \begin{align*}
  \left\{
    \begin{array}{cc} 3 \\ 3 \end{array}
    \begin{array}{cc} 6 \\ 6 \end{array}
    \begin{array}{cc} \\ \cdots \end{array}
    \begin{array}{cc} 3n_2 \\ 3n_2 \end{array}
    \begin{array}{cc} \\ 3n_2 + 3 \end{array}
    \begin{array}{cc} \\ 3n_2 + 5 \end{array}
    \begin{array}{cc} \\ \cdots \end{array}
    \begin{array}{cc} \\ 3n_2 + 2n_1 + 1 \end{array}
  \right\}
    \begin{array}{cc} \\ . \end{array}
 \end{align*}
 It is \eqref{basePtnKR2No2wrong}, and has weight $3n_2^2 + 3n_2 + n_1^2 + 2n_1 + 3n_1n_2$.  
 This weight is minimal among all partitions having $n_2$ 2-clusters, $n_1$ 1-clusters, 
 and satisfying the difference conditions imposed by $kr_3(n, m)$.  
\end{proof}

\begin{theorem}[cf. Kanade-Russell conjecture $I_4$]
\label{thmKR4}
  For $n, m \in \mathbb{N}$, 
  let $kr_4(n, m)$ be the number of partitions of $n$ into $m$ parts 
  with smallest part at least two, and difference at least three at distance two 
  such that if two successive parts differ by at most one, 
  then their sum is $\equiv 2 \pmod{3}$.  
  Then, 
  \begin{align}
  \label{eqGF_KR4}
    \sum_{n, m \geq 0} kr_4(n, m) q^n x^m 
    = \sum_{n_1, n_2 \geq 0} \frac{ q^{3n_2^2 + 2n_2 + n_1^2 + n_1 + 3n_1n_2} x^{2n_2 + n_1} }
      { (q; q)_{n_1} (q^3; q^3)_{n_2} }.  
  \end{align}
\end{theorem}

\begin{proof}
 We observe that if we take a partition counted by $kr_1(n, m)$ and add 1 to all parts, 
 the smallest parts becomes at least two.  
 Also, the 2-clusters, the only pair of parts whose pairwise difference is at most one, 
 become
 \begin{align*}
  \left\{
    \begin{array}{cc} \\ (\textrm{ parts } \leq 3k-2) \end{array}
    \begin{array}{cc} 3k+1 \\ 3k+1 \end{array}
    \begin{array}{cc} \\ (\textrm{ parts } \geq 3k+4) \end{array}
  \right\}
 \end{align*}
 and 
 \begin{align*}
  \left\{
    \begin{array}{cc} \\ (\textrm{ parts } \leq 3k) \end{array}
    \begin{array}{cc} & 3k+3 \\ 3k+2 & \end{array}
    \begin{array}{cc} \\ (\textrm{ parts } \geq 3k+5) \end{array}
  \right\}
    \begin{array}{cc} \\ , \end{array}
 \end{align*}
 instead of 
 \begin{align*}
  \left\{
    \begin{array}{cc} \\ (\textrm{ parts } \leq 3k-3) \end{array}
    \begin{array}{cc} 3k \\ 3k \end{array}
    \begin{array}{cc} \\ (\textrm{ parts } \geq 3k+3) \end{array}
  \right\}
 \end{align*}
 and 
 \begin{align*}
  \left\{
    \begin{array}{cc} \\ (\textrm{ parts } \leq 3k-1) \end{array}
    \begin{array}{cc} & 3k+2 \\ 3k+1 & \end{array}
    \begin{array}{cc} \\ (\textrm{ parts } \geq 3k+4) \end{array}
  \right\}
    \begin{array}{cc} \\ , \end{array}
 \end{align*}
 respectively.  
 Therefore, the sum of parts of the displayed 2-clusters become $\equiv 2 \pmod{3}$, 
 conforming to the definition of $kr_4(n, m)$.  
 
 Conversely, a partition enumerated by $kr_4(n, m)$ 
 can only have 1- or 2-marked parts in its Gordong marking.  
 Therefore, such a partition can have $r$-clusters for $r = 1, 2$, 
 but not for $r \geq 3$.  
 Because the 2-clusters consist of a pair of parts with difference zero or one, 
 they can be 
 \begin{align*}
  \begin{array}{cc} 3k \\ 3k \end{array}
  \begin{array}{cc} \\ , \end{array}
  \begin{array}{cc} 3k+1 \\ 3k+1 \end{array}
  \begin{array}{cc} \\ , \end{array}
  \begin{array}{cc} 3k+2 \\ 3k+2 \end{array}
  \begin{array}{cc} \\ , \end{array}
  \begin{array}{cc} & 3k+1 \\ 3k & \end{array}
  \begin{array}{cc} \\ , \end{array}
  \begin{array}{cc} & 3k+2 \\ 3k+1 & \end{array}
  \begin{array}{cc} \\ , \textrm{ or } \end{array}
  \begin{array}{cc} & 3k+3 \\ 3k+2 & \end{array}
  \begin{array}{cc} \\ . \end{array}
 \end{align*}
 Only the second and the sixth ones have sums $\equiv 2 \pmod{3}$, 
 therefore only such 2-clusters can occur in the said partition.  
 Because all parts are at least two we will not lose any parts, 
 nor do we need to redo the Gordon marking when we subtract one from all parts.  
 This operation makes the partition satisfy the conditions of $kr_1(n, m)$.  
 Therefore, we have $kr_4(n, m) = kr_1(n+m, m)$, yielding the theorem.   
\end{proof}

We can now turn our attention to the missing cases of partitions 
defined similarly to $kr_1(n, m)$-$kr_4(n, m)$.  
It turns out that only two such cases needs justification like the proofs of 
Theorems \ref{thmKR1}-\ref{thmKR3}, 
and the remaining ones can be obtained via shifts as in the proof of Theorem \ref{thmKR4}. 
Although Kanade and Russell's machinery in \cite{KR-conj} does not give nice single infinite products, 
hence nice partition identities for these missing cases, 
it is possible to write generating functions for them 
such as the Andrews-Gordon identities \cite{Andrews-Gordon}.  

\begin{theorem}
\label{thmKR3-1}
  For $n, m \in \mathbb{N}$, 
  let $kr_{3-1}(n, m)$ be the number of partitions of $n$ into $m$ parts 
  with smallest part at least two, and difference at least three at distance two 
  such that if two successive parts differ by at most one, 
  then their sum is $\equiv 2 \pmod{3}$.  
  Then, 
  \begin{align*}
%   \label{eqGF_KR3-1}
    \sum_{n, m \geq 0} kr_{3-1}(n, m) q^n x^m 
    & = \sum_{n_1, n_2 \geq 1} \frac{ q^{3n_2^2 + 6n_2 + n_1^2 + 3n_1 + 3n_1n_2-1} x^{2n_2 + n_1} }
      { (q; q)_{n_1} (q^3; q^3)_{n_2} } \\ 
    & + \sum_{n_2 \geq 0} \frac{ q^{3n_2^2 + 6n_2} x^{2n_2} }
      { (q^3; q^3)_{n_2} } 
    + \sum_{n_1 \geq 1} \frac{ q^{n_1^2 + 2n_1} x^{n_1} }
      { (q; q)_{n_1} }.  
  \end{align*}
\end{theorem}

\begin{proof}
 The idea of the proof is a direct extension of the proof of Theorem \ref{thmKR2} 
 based on the proof of Theorem \ref{thmKR1}.  
 The necessity of separate sums is in fact the necessity of different types of base partitions $\beta$ 
 for various constellations of the 2- and 1-clusters.  
 Observe that the ranges of the three sums 
 ($n_1, n_2 \geq 1$; $n_1 = 0$, $n_2 \geq 0$; $n_1 \geq 1$, $n_2 = 0$) 
 form a set partition of the expected natural range $n_1, n_2 \geq 0$.  
 Recall that $n_r$ is the number of the $r$-clusters of the partition at hand for $r = 1, 2$.  
 
 The base partition for the case $n_1, n_2 \geq 1$ is 
 \begin{align*}
  \left\{
    \begin{array}{cc} \\ 3 \end{array}
    \begin{array}{cc} 6 \\ 6 \end{array}
    \begin{array}{cc} 9 \\ 9 \end{array}
    \begin{array}{cc} \\ \cdots \end{array}
    \begin{array}{cc} 3n_3 + 3 \\ 3n_3 + 3 \end{array}
    \begin{array}{cc} \\ 3n_3 + 6 \end{array}
    \begin{array}{cc} \\ 3n_3 + 8 \end{array}
    \begin{array}{cc} \\ \cdots \end{array}
    \begin{array}{cc} \\ 3n_3 + 2n_1 + 4 \end{array}
  \right\}
    \begin{array}{cc} \\ , \end{array}
 \end{align*}
 with weight $3n_2^2 + 6n_2 + n_1^2 + 3n_1 + 3n_1n_2-1$.  
 Clearly, there are no 1-clusters greater than the 2-clusters if $n_1 = 1$.  
 
 When $n_1 = 0$ and $n_2 \geq 0$, the base partition $\beta$ is 
 \begin{align*}
  \left\{
    \begin{array}{cc} & 5 \\ 4 & \end{array}
    \begin{array}{cc} & 8 \\ 7 & \end{array}
    \begin{array}{cc} \\ \cdots \end{array}
    \begin{array}{cc} & 3n_3 + 2 \\ 3n_3 + 1 & \end{array}
  \right\}
    \begin{array}{cc} \\ , \end{array}
 \end{align*}
 with weight $3n_2^2 + 6n_2$.  It is the empty partition if $n_2 = 0$.  
 
 And finally, if $n_2 = 0$ and $n_1 \geq 1$, the base partition $\beta$ is 
 \begin{align*}
  \left\{ 
    \begin{array}{c} 3 \end{array}
    \begin{array}{c} 5 \end{array}
    \begin{array}{c} \cdots \end{array}
    \begin{array}{c} 2n_1 + 1\end{array}
  \right\} 
    \begin{array}{c} , \end{array}
 \end{align*}
 with weight $n_1^2 + 2n_1$.  
 We do not want to double count the empty partition here, hence $n_1 \geq 1$.  
 
 Without much difficulty, one can verify that the above $\beta$s are partitions with minimal weight 
 having specified number of 1- and 2-clusters ($n_1$ and $n_2$, respectively), 
 while satisfying the difference conditions set forth by $kr_{3-1}(n, m)$.  
\end{proof}

One can play with the $\pmod{3}$ condition on sums, 
and adjust the lower limit for the smallest part to populate the list.  
Theorems \ref{thmKR1}-\ref{thmKR3-1} are exclusive to obtain the respective series 
as generating functions by means of shifts on parts.  
We present two more examples.  

\begin{theorem}
\label{thmMS1-4}
  For $n, m \in \mathbb{N}$, let us define the partition enumerants below.  
  
  $kr^b_1(n, m)$ is the number of partitions of $n$ into $m$ parts 
  with difference at least three at distance two 
  such that if two successive parts differ by at most one, 
  then their sum is $\equiv 1 \pmod{3}$.  
  
  $kr^b_{4-2}(n, m)$ is the number of partitions of $n$ into $m$ parts 
  with at most one occurrence of the part 1, 
  and difference at least three at distance two 
  such that if two successive parts differ by at most one, 
  then their sum is $\equiv 2 \pmod{3}$.  
  
  Then, 
  \begin{align*}
%   \label{eqGF_KRb1}
    \sum_{n, m \geq 0} kr^b_1(n, m) q^n x^m 
    & = \sum_{n_1, n_2 \geq 0} \frac{ q^{3n_2^2 + n_2 + n_1^2 + 3n_1n_2} x^{2n_2 + n_1} }
      { (q; q)_{n_1} (q^3; q^3)_{n_2} }, 
  \end{align*}
  and 
  \begin{align*}
%   \label{eqGF_KR4-2}
    \sum_{n, m \geq 0} kr_{4-2}(n, m) q^n x^m 
    & = \sum_{n_1, n_2 \geq 1} \frac{ q^{3n_2^2 + 2n_2 + n_1^2 + n_1 + 3n_1n_2-1} x^{2n_2 + n_1} }
      { (q; q)_{n_1} (q^3; q^3)_{n_2} } \\ 
    & + \sum_{n_2 \geq 0} \frac{ q^{3n_2^2 + 2n_2} x^{2n_2} }
      { (q^3; q^3)_{n_2} } 
    + \sum_{n_1 \geq 1} \frac{ q^{n_1^2} x^{n_1} }
      { (q; q)_{n_1} }.  
  \end{align*}
\end{theorem}

\begin{proof}
 It suffices to see that $kr^b_1(n+m, m) = kr_2(n, m)$, 
 and that $kr_{4-2}(n+2m, m) = kr_{3-1}(n, m)$.  
 Then the results become corollaries of Theorems \ref{thmKR2} and \ref{thmKR3-1}, respectively.  
\end{proof}

We conclude this section with one last example.  

\begin{theorem}
\label{thmKRb1-1}
  For $n, m \in \mathbb{N}$, 
  let $kr^b_{1-1}(n, m)$ be the number of partition of $n$ into $m$ parts 
  with at most one occurrence of the part 2, 
  and with difference at least three at distance two 
  such that if two successive parts differ by at most one, 
  then their sum is $\equiv 1 \pmod{3}$.  
  Then, 
  \begin{align*}
%   \label{eqGF_KRb1-1}
    \sum_{n, m \geq 0} kr^b_{1-1}(n, m) q^n x^m 
    & = \sum_{\substack{n_1 \geq 0 \\ n_2 \geq 1} }
	  \frac{ q^{3n_2^2 + 4n_2 + n_1^2 + n_1 + 3n_1n_2} x^{2n_2 + n_1} }
	    { (q; q)_{n_1} (q^3; q^3)_{n_2} } \\
    & + \sum_{\substack{n_1 \geq 0 \\ n_2 \geq 1}} 
	  \frac{ q^{3n_2^2 + 4n_2 + (n_1+1)^2 + 3n_1n_2} x^{2n_2 + n_1+1} }
	    { (q; q)_{n_1} (q^3; q^3)_{n_2} } 
    + \sum_{n_1 \geq 0} \frac{ q^{n_1^2} x^{n_1} }
      { (q; q)_{n_1} }.  
  \end{align*}
\end{theorem}

The enumerant $kr^b_{1-1}(n, m)$ is brought to our attention by Alexander Berkovich.  
It is unusual in the sense that the number of occurrences 
is not restricted for the smallest admissible part, but for a larger one.  
We include it here to demonstrate the fact that the method may treat 
virtually all possible extra conditions on the first so many parts 
on top of the general difference conditions.  

\begin{proof}
 The proof is reminiscent of that of Theorem \ref{thmKR3-1}.  
 We need base partitions $\beta$ for several cases.  
 Below, $\lambda$ is a partition enumerated by $kr^b_{1-1}(n, m)$, 
 and $n_r$ is the number of $r$-clusters for $r = 1, 2$.  
 \begin{enumerate}[{\bf (i)}]
  \item $\lambda$ has no 2-clusters, i.e. $n_2 = 0$, 
  \item $\lambda$ has at least one 2-cluster, but no 1's,   
  \item $\lambda$ has at least one 2-cluster, and a 1.  
 \end{enumerate}
 
 In case {\bf (i)}, the base partition $\beta$ obviously is 
 \begin{align*}
  \left\{ 
    \begin{array}{c} 1 \end{array}
    \begin{array}{c} 3 \end{array}
    \begin{array}{c} \cdots \end{array}
    \begin{array}{c} 2n_1 - 1\end{array}
  \right\}
    \begin{array}{c} , \end{array}
 \end{align*}
 with weight $n_1^2$.  
 
 In case {\bf (ii)}, the base partitions $\beta$ are 
 \begin{align*}
  \left\{
    \begin{array}{cc} & 4 \\ 3 & \end{array}
    \begin{array}{cc} & 7 \\ 6 & \end{array}
    \begin{array}{cc} \\ \cdots \end{array}
    \begin{array}{cc} & 3n_2+1 \\ 3n_2 & \end{array}
  \right\}
 \end{align*}
 when $n_1 = 0$, 
 \begin{align}
 \label{baseptnKRb1-1_2}
  \left\{
    \begin{array}{cc} \\ 2 \end{array}
    \begin{array}{cc} 5 \\ 5 \end{array}
    \begin{array}{cc} 8 \\ 8 \end{array}
    \begin{array}{cc} \\ \cdots \end{array}
    \begin{array}{cc} 3n_2+2 \\ 3n_2+2 \end{array}
  \right\}
 \end{align}
 when $n_1 = 1$, 
 \begin{align}
 \label{baseptnKRb1-1_3}
  \left\{
    \begin{array}{cc} \\ 2 \end{array}
    \begin{array}{cc} \\ 4 \end{array}
    \begin{array}{cc} & 7 \\ 6 & \end{array}
    \begin{array}{cc} & 10 \\ 9 & \end{array}
    \begin{array}{cc} \\ \cdots \end{array}
    \begin{array}{cc} & 3n_2+4 \\ 3n_2+3 & \end{array}
    \begin{array}{cc} \\ 3n_2+6 \end{array}
    \begin{array}{cc} \\ 3n_2+8 \end{array}
    \begin{array}{cc} \\ \cdots \end{array}
    \begin{array}{cc} \\ 3n_2+2n_1 \end{array}
  \right\}
 \end{align}
  when $n_1 \geq 2$.  
  The weights of all three partitions above are $3n_2^2 + 4n_2 + n_1^2 + n_1 + 3n_2n_1$.  
  In \eqref{baseptnKRb1-1_2}, the initial forward move on the smallest 1-cluster, 
  and in \eqref{baseptnKRb1-1_3}, the initial forward moves on the two smallest 1-clusters 
  involve prestidigitating the said 1-clusters through the 2-clusters, if any.  
  
  In case {\bf (iii)}, the base partition is 
 \begin{align}
 \label{baseptnKRb1-1_4}
  \left\{
    \begin{array}{cc}  \\ 1 \end{array}
    \begin{array}{cc} & 4 \\ 3 & \end{array}
    \begin{array}{cc} & 7 \\ 6 & \end{array}
    \begin{array}{cc} \\ \cdots \end{array}
    \begin{array}{cc} & 3n_2+1 \\ 3n_2 & \end{array}
    \begin{array}{cc}  \\ 3n_2+3 \end{array}
    \begin{array}{cc}  \\ 3n_2+5 \end{array}
    \begin{array}{cc}  \\ \cdots \end{array}
    \begin{array}{cc}  \\ 3n_2+2n_1+1 \end{array}
  \right\}
    \begin{array}{cc}  \\ . \end{array}
 \end{align}
 Here, we leave the part 1 where it is, 
 and set $n_1 = $ the number of 1-clusters except the part 1.  
 In other words, we do not perform any forward moves on the part 1.  
\end{proof}

{\bf Remark: } One can also argue that $kr^b_{1-1}(n, m) = A(n, m) + B(n, m)$, 
where $A(n, m)$ is the set of partitions enumerated by $kr^b_{1-1}(n, m)$ which contain a 1, 
and $B(n, m)$ is those which do not.  
Then, one can establish $A(n, m) = kr_1(n+2m, m)$ by deleting 1 from the said partitions, 
and $B(n+m, m) = kr^b_{1-3}(n, m)$ to obtain
\begin{align*}
%   \label{eqGF_KRb1-1-alt}
  & \sum_{n, m \geq 0} kr^b_{1-1}(n, m) q^n x^m 
  = \sum_{n_1, n_2 \geq 0}
	\frac{ q^{3n_2^2 + 4n_2 + (n_1+1)^2 + n_1 + 3n_1n_2} x^{2n_2 + n_1+1} }
	  { (q; q)_{n_1} (q^3; q^3)_{n_2} } \\
  & + \sum_{n_1, n_2 \geq 1} 
	\frac{ q^{3n_2^2 + 4n_2 + n_1^2 + 3n_1n_2-1} x^{2n_2 + n_1} }
	  { (q; q)_{n_1} (q^3; q^3)_{n_2} } 
  + \sum_{n_2 \geq 0} \frac{ q^{3n_2^2+4n_2} x^{2n_2} }
    { (q^3; q^3)_{n_2} } 
  + \sum_{n_1 \geq 1} \frac{ q^{n_1^2} x^{n_1} }
    { (q; q)_{n_1} }.  
\end{align*}
It is a simple matter to show the equivalence of the above identity to the 
combination of multiple series in Theorem \ref{thmMS1-4}, 
once one knows the combinatorics behind.  

Yet a third way to obtain another alternative is to exclude 
the partitions counted by $kr^b_1(n, m)$ which have the 2-cluster 
$\begin{array}{c} 2 \\ 2 \end{array}$ using $kr_{3-1}(n, m)$.  
However, we do not favor inclusion-exclusion in this note.  

% begin example 
  
  {\bf Example: } 
  Following the notation in the section so far, 
  we will decode the partition $\lambda$ enumerated by $kr^b_{1-1}(62, 7)$ below 
  into $(\beta, \mu, \eta)$.  
  \begin{align*}
   \left\{
   \begin{array}{cc} \\ 1 \end{array} \right. 
   \begin{array}{cc} & \mathbf{7} \\ \mathbf{6} & \end{array}
   \begin{array}{cc} \\ 9 \end{array}
   \begin{array}{cc} \\ 11 \end{array}
   \left. \begin{array}{cc} 14 \\ 14 \end{array}
   \right\}
  \end{align*}
  Obviously, we are in the case {\bf (iii)} of the above proof.  
  $\lambda$ has $n_2 = 2$ 2-clusters, $n_1 = 2$ 1-clusters, and a 1.  
  We stow the smaller 2-cluster first, and record $\eta_1$ 
  as three times the performed number of moves.  
  \begin{align*}
    \big\downarrow \textrm{ one backward move on the smaller 2-cluster }
  \end{align*}
  \begin{align*}
   \left\{
   \begin{array}{cc} \\ 1 \end{array} \right. 
   \begin{array}{cc} \mathbf{5} \\ \mathbf{5} \end{array}
   \begin{array}{cc} \\ 9 \end{array}
   \begin{array}{cc} \\ 11 \end{array}
   \left. \begin{array}{cc} 14 \\ 14 \end{array}
   \right\}
  \end{align*}
  \begin{align*}
    \big\downarrow \textrm{ one more backward move on the smaller 2-cluster }
  \end{align*}
  \begin{align*}
   \left\{
   \begin{array}{cc} \\ 1 \end{array} \right. 
   \begin{array}{cc} & 4 \\ 3 & \end{array}
   \begin{array}{cc} \\ 9 \end{array}
   \begin{array}{cc} \\ 11 \end{array}
   \left. \begin{array}{cc} \mathbf{14} \\ \mathbf{14} \end{array}
   \right\}  
  \end{align*}
  At this point, we have $\eta_1 = 6$.  
  \begin{align*}
    \big\downarrow \textrm{ one backward move on the larger 2-cluster }
  \end{align*}
  \begin{align*}
   \left\{
   \begin{array}{cc} \\ 1 \end{array} \right. 
   \begin{array}{cc} & 4 \\ 3 & \end{array}
   \begin{array}{cc} \\ 9 \end{array}
   \underbrace{
   \begin{array}{cc} \\ 11 \end{array}
   \begin{array}{cc} & \mathbf{13} \\ \mathbf{12} & \end{array}
   }_{!} 
   \left. \begin{array}{cc} \\ \phantom{0} \end{array} \right\}
  \end{align*}
  \begin{align*}
    \big\downarrow \textrm{ adjustment }
  \end{align*}
  \begin{align*}
   \left\{
   \begin{array}{cc} \\ 1 \end{array} \right. 
   \begin{array}{cc} & 4 \\ 3 & \end{array}
   \underbrace{
   \begin{array}{cc} \\ 9 \end{array}
   \begin{array}{cc} \mathbf{11} \\ \mathbf{11} \end{array}
   }_{!} 
   \begin{array}{cc} \\ 14 \end{array}
   \left. \begin{array}{cc} \\ \phantom{0} \end{array} \right\}
  \end{align*}
  \begin{align*}
    \big\downarrow \textrm{ adjustment }
  \end{align*}
  \begin{align*}
   \left\{
   \begin{array}{cc} \\ 1 \end{array} \right. 
   \begin{array}{cc} & 4 \\ 3 & \end{array}
   \begin{array}{cc} & \mathbf{10} \\ \mathbf{9} & \end{array}
   \begin{array}{cc} \\ 12 \end{array}
   \left. \begin{array}{cc} \\ 14 \end{array}
   \right\}
  \end{align*}
  \begin{align*}
    \big\downarrow \textrm{ two more backward moves on the larger 2-cluster }
  \end{align*}
  \begin{align*}
   \left\{
   \begin{array}{cc} \\ 1 \end{array} \right. 
   \begin{array}{cc} & 4 \\ 3 & \end{array}
   \begin{array}{cc} & 7 \\ 6 & \end{array}
   \begin{array}{cc} \\ \mathbf{12} \end{array}
   \left. \begin{array}{cc} \\ 14 \end{array}
   \right\}
  \end{align*}
  Now we have $\eta = 6+9$.  
  Decoding the backward moves on the 1-clusters is easier.  
  It is obvious that $\mu = 3+3$, 
  and once we perform that many backward moves on the respective 1-clusters, 
  we arrive at \eqref{baseptnKRb1-1_4}.  
  \begin{align*}
   \left\{
   \begin{array}{cc} \\ 1 \end{array} \right. 
   \begin{array}{cc} & 4 \\ 3 & \end{array}
   \begin{array}{cc} & 7 \\ 6 & \end{array}
   \begin{array}{cc} \\ 9 \end{array}
   \left. \begin{array}{cc} \\ 11 \end{array}
   \right\}
  \end{align*} 
  The sum of weights also check.  
  \begin{align*}
   \vert \lambda \vert 
   = 62 = 41 + 6 + 15
   = \vert \beta \vert + \vert \mu \vert + \vert \eta \vert 
  \end{align*}
  
% end example 

\section{Kanade and Russell's Conjectures 5-6 and Some Missing Cases}
\label{secKR5-6}

\begin{theorem}[cf. Kanade-Russell Conjecture $I_5$]
\label{thmKR5}
  For $m, n \in \mathbb{N}$, 
  let $kr_5(m,n)$ be the number of partitions of $n$ into $m$ parts, 
  with at most one occurrence of the part 1, 
  and difference at least three at distance three 
  such that is parts at distance two differ by at most 1, 
  then their sum, together with the intermediate part, is $\equiv 1 \pmod{3}$.  
  Then, 
  \begin{align}
  \nonumber
  & \sum_{m,n \geq 0} kr_5(n,m) q^n x^m \\
  \label{eqGF_KR5}
  = & \sum_{n_1, n_2, n_3 \geq 0} \frac{ q^{ ( 9n_3^2 + 5n_3 )/2 + 2n_2^2 + n_2 + n_1^2 + 6n_3n_2 + 3n_3n_1 + 2n_2n_1 } 
      (-q; q^2)_{n_2} x^{3n_3 + 2n_2 + n_1} }
    { (q; q)_{n_1} (q^2; q^2)_{n_2} (q^3; q^3)_{n_3} }.  
  \end{align}
\end{theorem}

\begin{proof}
 Throughout the proof, $n_r$ will denote the number of $r$-clusters for $r = 1, 2, 3$.  
 $\lambda$ will denote a partition enumerated by $kr_5(n,m)$.  
 We will follow the idea of proof in Theorem \ref{thmKR1}, but there are more intricacies.  
 Construction of the base partition is a major part.  
 
 The base partition when $n_1 > 0$ is 
 \begin{align}
 \nonumber
  & \left\{ \begin{array}{cc} & \\ & 2 \\ 1 & \end{array}
  \begin{array}{cc} & \\ & 4 \\ 3 & \end{array}
  \begin{array}{cc} \\ \\ \cdots \end{array} 
  \begin{array}{cc} & \\ & 2 n_2 \\ 2n_2-1 & \end{array}
  \begin{array}{cc} \\ \\ 2n_2+1 \end{array}
  \begin{array}{cc} & 2n_2+4 \\ 2n_2+3 & \\ 2n_2+3 & \end{array}
  \begin{array}{cc} & 2n_2+7 \\ 2n_2+6 & \\ 2n_2+6 & \end{array} \right. \\[10pt]
 \label{baseptnKR5-n1positive}
  \begin{array}{cc} \\ \\ \cdots \end{array} 
  & \begin{array}{cc} & 2n_2+3n_3+1 \\ 2n_2+3n_3 & \\ 2n_2+3n_3 & \end{array}  
  \begin{array}{cc} \\ \\ 2n_2+3n_3+3 \end{array}
  \begin{array}{cc} \\ \\ 2n_2+3n_3+5 \end{array} \\[10pt]
  \nonumber
  \begin{array}{cc} \\ \\ \cdots \end{array} 
  & \left. \begin{array}{cc} \\ \\ 2n_2+3n_3+2n_1-1 \end{array} \right\}, 
 \end{align}
 and when $n_1 = 0$ it is
 \begin{align}
 \nonumber
  & \left\{ \begin{array}{cc} & \\ & 2 \\ 1 & \end{array}
  \begin{array}{cc} & \\ & 4 \\ 3 & \end{array}
  \begin{array}{cc} \\ \\ \cdots \end{array} 
  \begin{array}{cc} & \\ & 2 n_2 \\ 2n_2-1 & \end{array}
  \begin{array}{cc} & 2n_2+3 \\ 2n_2+2 & \\ 2n_2+2 & \end{array}
  \begin{array}{cc} & 2n_2+6 \\ 2n_2+5 & \\ 2n_2+5 & \end{array} \right. \\[10pt]
 \label{baseptnKR5-n1zero}
  \begin{array}{cc} \\ \\ \cdots \end{array} 
  & \begin{array}{cc} & 2n_2+3n_3 \\ 2n_2+3n_3-1 & \\ 2n_2+3n_3-1 & \end{array}  
  \begin{array}{cc} \\ \\ 2n_2+3n_3+2 \end{array}
  \begin{array}{cc} \\ \\ 2n_2+3n_3+4 \end{array} \\[10pt]
  \nonumber
  \begin{array}{cc} \\ \\ \cdots \end{array} 
  & \left. \begin{array}{cc} \\ \\ 2n_2+3n_3+2n_1-2 \end{array} \right\}.  
 \end{align}
 The weight of both of them is $( 9n_3^2 + 5n_3 )/2 + 2n_2^2 + n_2 + n_1^2 + 6n_3n_2 + 3n_3n_1 + 2n_2n_1$.  
 
 We have to argue that this is indeed the partition counted by $kr_5(n,m)$ 
 having $n_r$ $r$-clusters for $r = 1,2,3$ and minimal weight.  
 
 If $\lambda$ has a 3-marked part $k$, then there is a 2-marked part $k$ or $k-1$, 
 and a 1-marked part $k$ or $k-1$.  There can be no other parts equal to $k$ or $k-1$ 
 because of the difference at least three at distance three condition.  
 For the same reason, the succeding smallest part can be at least $k+2$, 
 and the preceding smallest part can be at most $k-2$.  
 Among the three possibilities for the 3-clusters, 
 \begin{align*}
  \begin{array}{cc} & k \\ k-1 & \\ k-1 & \end{array}
  \begin{array}{cc}  \\ \\ , \end{array} \qquad
  \begin{array}{cc} & k \\ & k \\ k-1 & \end{array}
  \begin{array}{cc}  \\ \\ \textrm{ and } \end{array} \qquad
  \begin{array}{cc} k \\ k \\ k \end{array}
  \begin{array}{cc}  \\ \\ , \end{array}
 \end{align*}
 which all have difference at most 1 at distance two, 
 the only one satisfying the sum condition, 
 i.e. the sum of the parts, together with the middle part $\equiv 1 \pmod{3}$ is 
 \begin{align*}
  \left\{
  \begin{array}{cc}  \\ \\ ( \textrm{ parts } \leq k-3) \end{array}    
  \begin{array}{cc} & k \\ k-1 & \\ k-1 & \end{array} 
  \begin{array}{cc}  \\ \\ ( \textrm{ parts } \geq k+2) \end{array}
  \begin{array}{cc}  \\ \\ . \end{array}
  \right\}
 \end{align*}
 Therefore, all 3-clusters are of this form.  The preceding cluster can be at most 
 $ % \begin{align*}
  \begin{array}{cc} & k-3 \\ k-4 & \\ k-4 & \end{array} 
%   \begin{array}{cc}  \\ \\ , \end{array}
 $, % \end{align*}
 and the succeeding cluster can be at least
 $ % \begin{align*}
  \begin{array}{cc} & k+3 \\ k+2 & \\ k+2 & \end{array} 
%   \begin{array}{cc}  \\ \\ . \end{array}
 $.  % \end{align*}
 Also, a 3-cluster in $\lambda$ can be
 $ % \begin{align*}
  \begin{array}{cc} & 3 \\ 2 & \\ 2 & \end{array} 
%   \begin{array}{cc}  \\ \\ , \end{array}
 $, % \end{align*}
 but not 
 $ % \begin{align*}
  \begin{array}{cc} & 2 \\ 1 & \\ 1 & \end{array} 
%   \begin{array}{cc}  \\ \\ , \end{array}
 $, % \end{align*}
 because at most one occurrence of the part 1 is allowed.  
 This shows that, if a base partition consists of 3-clusters only, it will be 
 \begin{align*}
  \left\{
  \begin{array}{cc} & 3 \\ 2 & \\ 2 & \end{array} 
  \begin{array}{cc} & 6 \\ 5 & \\ 5 & \end{array} 
  \begin{array}{cc}  \\ \\ \cdots \end{array}
  \begin{array}{cc} & 3n_3 \\ 3n_3-1 & \\ 3n_3-1 & \end{array} 
  \begin{array}{cc}  \\ \\ . \end{array}
  \right\}
 \end{align*}
 
 For a moment, suppose that there are no 3-clusters in $\lambda$.  
 Equivalently, there are no 3-marked parts.  
 The 2-clusters will look like
 $ % \begin{align*}
  \begin{array}{cc} & k \\ k-1 & \end{array}
 $ or 
%   \begin{array}{cc}  \\ \textrm{ or } \end{array} \qquad
 $ \begin{array}{cc} k \\ k \end{array}
%   \begin{array}{cc}  \\ . \end{array}
 $. % \end{align*}
 Two successive 2-clusters may look like
 \begin{align*}
  \left\{ \begin{array}{cc}  \\ \cdots \end{array}
  \begin{array}{cc} & k \\ k-1 & \end{array}
  \begin{array}{cc} & k+2 \\ k+1 & \end{array}
  \begin{array}{cc}  \\ \cdots \end{array} \right\} \quad
  \begin{array}{cc}  \\ \textrm{ or } \end{array} \quad
  \left\{ \begin{array}{cc}  \\ \cdots \end{array}
  \begin{array}{cc} k \\ k \end{array}
  \begin{array}{cc} & k+3 \\ k+2 & \end{array}
  \begin{array}{cc}  \\ \cdots \end{array} \right\}
  \begin{array}{cc}  \\ , \end{array}
 \end{align*}
 but not
 \begin{align*}
  \left\{ \begin{array}{cc}  \\ \cdots \end{array}
  \begin{array}{cc} k \\ k \end{array}
  \begin{array}{cc} k+2 \\ k+2 \end{array}
  \begin{array}{cc}  \\ \cdots \end{array} \right\}
  \begin{array}{cc}  \\ . \end{array}
 \end{align*}
 In the last instance, the difference at least three at distance three condition is violated.  
 
 1-clusters preceding or succeeding a 2-cluster may look like
 \begin{align*}
  \left\{ \begin{array}{cc}  \\ \cdots \end{array}
  \begin{array}{cc} \\ k-5 \end{array}
  \begin{array}{cc} \\ k-3 \end{array}
  \begin{array}{cc} & k \\ k-1 & \end{array}
  \begin{array}{cc} \\ k+1 \end{array}
  \begin{array}{cc} \\ k+3 \end{array}
  \begin{array}{cc}  \\ \cdots \end{array} \right\}
 \end{align*}
 or 
 \begin{align*}
  \left\{ \begin{array}{cc}  \\ \cdots \end{array}
  \begin{array}{cc} \\ k-4 \end{array}
  \begin{array}{cc} \\ k-2 \end{array}
  \begin{array}{cc} k \\ k \end{array}
  \begin{array}{cc} \\ k+2 \end{array}
  \begin{array}{cc} \\ k+4 \end{array}
  \begin{array}{cc}  \\ \cdots \end{array} \right\}
  \begin{array}{cc}  \\ . \end{array}
 \end{align*}
 Recall that if 1-clusters have pairwise difference 1, they become 2-clusters.  
 Or an instance such as 
 \begin{align*}
  \left\{ \begin{array}{cc}  \\ \cdots \end{array}
  \begin{array}{cc} \\ k-2 \end{array}
  \begin{array}{cc} & k \\ k - 1 & \end{array}
  \begin{array}{cc}  \\ \cdots \end{array} \right\}
 \end{align*}
 requires redefinition of the Gordon marking, hence the clusters as 
 \begin{align*}
  \left\{ \begin{array}{cc}  \\ \cdots \end{array}
  \begin{array}{cc} & k-1 \\ k-2 & \end{array}
  \begin{array}{cc} \\ k \end{array}
  \begin{array}{cc}  \\ \cdots \end{array} \right\}
  \begin{array}{cc}  \\ , \end{array}
 \end{align*}
 or even create a 3-cluster.  
 
 Therefore, a base partition consisting only of 1- and 2-clusters looks like
 \begin{align*}
  \left\{ \begin{array}{cc} & 2 \\ 1 & \end{array}
  \begin{array}{cc} & 4 \\ 3 & \end{array}
  \begin{array}{cc}  \\ \cdots \end{array}
  \begin{array}{cc} & 2n_2 \\ 2n_2-1 & \end{array}
  \begin{array}{cc} \\ 2n_2+1 \end{array}
  \begin{array}{cc} \\ 2n_2+3 \end{array}
  \begin{array}{cc}  \\ \cdots \end{array}
  \begin{array}{cc} \\ 2n_2+2n_1-1 \end{array} \right\}
  \begin{array}{cc}  \\ . \end{array}
 \end{align*}
 Having 2-clusters greater than 1-clusters will only increase the weight.  
 One way to see this is that the 1-marked parts can be 
 $1, 3, \ldots, 2k-1$ for the least weight.  
 The introduction of the 2-marked parts will form 2-clusters.  
 $2, 4, \ldots, 2l$ is the least addendum to the weight.  
 We recall once again that a second occurrence of 1 is not allowed.  
 This covers the cases $n_1 = 0$ or $n_2 = 0$ as well.  
 
 The remaining cases are the coexistence of 3-clusters, and 1- and 2-clusters.  
 We will examine the cases $n_1 = 0$, $n_2, n_3 > 0$, 
 and $n_1, n_3 > 0$, $n_2 \geq 0$ separately, 
 for reasons that will become clear in the course.  
 
 It is clear that each cluster should have as small parts as possible in a base partition to ensure minimum weight.  
 Therefore, we will only focus on the relative placement of the clusters.  
 The na\"{i}ve guess is to place 3-clusters first, followed by 2-clusters, and then the 1-clusters.  
 For example, 
 \begin{align*}
  \left\{ \begin{array}{cc} & 3 \\ 2 & \\ 2 & \end{array}
  \begin{array}{cc} & 6 \\ 5 & \\ 5 & \end{array}
  \begin{array}{cc} \\ 8 \\ 8 \end{array}
  \begin{array}{cc} & \\ & 11 \\ 10 & \end{array}
  \begin{array}{cc} \\ \\ 12 \end{array}
  \begin{array}{cc} \\ \\ 14 \end{array} \right\}
 \end{align*} 
 has weight 86.  However, 
 \begin{align*}
  \left\{   \begin{array}{cc} \\ & 2 \\ 1 & \end{array}
  \begin{array}{cc} & \\ & 4 \\ 3 & \end{array} 
  \begin{array}{cc} & 7 \\ 6 & \\ 6 & \end{array}
  \begin{array}{cc} & 10 \\ 9 & \\ 9 & \end{array}
  \begin{array}{cc} \\ \\ 12 \end{array}
  \begin{array}{cc} \\ \\ 14 \end{array} \right\}
 \end{align*} 
 has weight 83, while 
 \begin{align*}
  \left\{   \begin{array}{cc} \\ & 2 \\ 1 & \end{array}
  \begin{array}{cc} & \\ & 4 \\ 3 & \end{array} 
  \begin{array}{cc} \\ \\ 5 \end{array}
  \begin{array}{cc} & 8 \\ 7 & \\ 7 & \end{array}
  \begin{array}{cc} & 11 \\ 10 & \\ 10 & \end{array}
  \begin{array}{cc} \\ \\ 12 \end{array} \right\}
 \end{align*} 
 has weight 80.  Having been experienced, one tries 
 \begin{align*}
  \left\{   \begin{array}{cc} \\ & 2 \\ 1 & \end{array}
  \begin{array}{cc} & \\ & 4 \\ 3 & \end{array} 
  \begin{array}{cc} \\ \\ 5 \end{array}
  \begin{array}{cc} \\ \\ 7 \end{array} 
  \begin{array}{cc} & 10 \\ 9 & \\ 9 & \end{array}
  \begin{array}{cc} & 13 \\ 12 & \\ 12 & \end{array}\right\}
  \begin{array}{cc} \\ \\ , \end{array} 
 \end{align*} 
 but the weight becomes 87.  
 The na\"{i}ve guess has another problem,   
 we will come back to it during the implementation of the forward moves.  
 
 The general case is similarly treated.  
 One should keep in mind that the 2-clusters should precede the 1-clusters in the base partition as discussed above, 
 so the relative places of the 3-clusters are to be decided.  
 One can also verify that placing 1- or 2-clusters between two 3-clusters increases the weight.  
 In summary, depending on the existence of 1-clusters, 
 the base partition will be \eqref{baseptnKR5-n1zero} or \eqref{baseptnKR5-n1positive}.  
 
 Next, we argue that any $\lambda$ enumerated by $kr_5(n,m)$ 
 having $n_r$ $r$-clusters for $r = 1, 2, 3$ corresponds to a quadruple 
 $(\beta, \mu, \eta, \nu)$ such that 
 \begin{itemize}
  \item $\beta$ is one of the base partitions \eqref{baseptnKR5-n1zero} or \eqref{baseptnKR5-n1positive}, 
    depending on $n_1 = 0$ or $n_1 > 0$, respectively,  
    
  \item $\mu$ is a partition with $n_1$ parts (counting zeros), 
  
  \item $\eta$ is a partition with $n_2$ parts (counting zeros) where no odd part repeats, 
  
  \item $\nu$ is a partition into multiples of three with $n_3$ parts (counting zeros), 
  
  \item $\vert \lambda \vert = \vert \beta \vert + \vert \mu \vert + \vert \eta \vert + \vert \nu \vert$.  
 \end{itemize}
 If, say, $\mu$ has less than $n_1$ positive parts, 
 we simply write $\mu_1 = \mu_2 = \cdots = \mu_s = 0$.  
 That is, the first so many parts of $\mu$ are declared zero.  
 Recall that we agreed to write the smaller parts first in a partition.  
 If $\mu$ is the empty partition, then all parts of it are zero.  
 $\eta$ and $\nu$ are treated likewise.  This will give us 
 \begin{align}
 \nonumber 
  & \sum_{m,n \geq 0} kr_5(n,m) q^n x^m 
  = \sum_{n_1, n_2, n_3 \geq 0} q^{\vert \beta \vert} x^{l(\beta)} \sum_{\beta, \mu, \eta, \nu} 
    q^{\vert \mu \vert + \vert \eta \vert + \vert \nu \vert}  \\
 \label{eqGFkr5intermediate}
  = & \sum_{n_1, n_2, n_3 \geq 0} 
    \underbrace{ q^{ ( 9n_3^2 + 5n_3 )/2 + 2n_2^2 + n_2 + n_1^2 + 6n_3n_2 + 3n_3n_1 + 2n_2n_1 } 
      x^{3n_3 + 2n_2 + n_1}  }_{\textrm{generating } \beta} \cdots \\
 \nonumber 
  \times & \underbrace{ \frac{ 1 }{ (q; q)_{n_1} } }_{\textrm{generating } \mu }
      \underbrace{ \frac{ (-q; q^2)_{n_2} }{ (q^2; q^2)_{n_2} } }_{\textrm{generating } \eta }
      \underbrace{ \frac{ 1 }{ (q^3; q^3)_{n_3} } }_{\textrm{generating } \nu }, 
 \end{align}
 proving the theorem.  
 We used Proposition \ref{propDistinctOdds} in the generation of $\eta$.  
 
 Given a quadruple $(\beta, \mu, \eta, \nu)$ as described above, 
 we will obtain $\lambda$ in a series of forward moves.  
 \begin{enumerate}[(a)]
  \item The $i$th largest 1-cluster in $\beta$ is moved forward the $i$th largest part of $\mu$ times 
    for $i = 1, 2, \ldots, n_1$, in this order.  
  
  \item The $i$th largest 2-cluster in the obtained intermediate partition 
    is moved forward the $i$th largest part of $\eta$ times for $i = 1, 2, \ldots, n_2$, in this order.  
  
  \item The $i$th largest 3-cluster in the obtained intermediate partition 
    is moved forward $\frac{1}{3}\times$(the $i$th largest part of $\nu$) 
    times for $i = 1, 2, \ldots, n_3$, in this order.  
 \end{enumerate}
 Conversely, given $\lambda$, we will obtain the quadruple $(\beta, \mu, \eta, \nu)$ 
 by performing backward moves on the 3-, 2-, and 1-, clusters in the exact reverse order.  
 Finally, we will argue that the forward moves and the backward moves on the $r$-clusters 
 are inverses of each other for $r = 1,2,3$, 
 and that the moves honor the difference conditions defining $kr_5(n,m)$. 
 
 The forward and backward moves on the 3-clusters are not exactly 
 forward and backward moves of the 3rd kind in the sense of Definitions \ref{defFwdMove}-\ref{defBackwdMove}.  
 However, the forward and backward moves on the 2-clusters are 
 forward or backward moves of the 2nd kind, with one exception.  
 The exception is described in due course.  
 
 We start with the forward moves.  When $\beta$ has at least one 1-cluster, i.e. $n_1 > 0$, 
 the smallest 1-cluster is smaller than the 3-clusters
 For $i = 1, 2, \ldots, n_1 - 1$, we simply add the $i$th largest part of $\mu$ to the $i$th largest 1-cluster.  
 This only increases the pairwise difference of the 1-clusters, so the difference conditions are retained.  
 If $\mu_1 > 0$, observe that the $(n_1 - 1)$th 1-cluster, if it exits, is moved forward $\mu_2$ times.  
 Therefore, it is now equal to $2n_2+3n_3+3+\mu_2 \geq 2n_2+3n_3+3+\mu_1$.  
 The first forward move on the smallest 1-cluster $2n_2+1$ entails a prestidigitation through 
 the 3-clusters as described below.  
 \begin{align*}
  & \left\{  \begin{array}{cc} \\ \\ \cdots \end{array}
  \begin{array}{cc} & \\ & 2n_2 \\ 2n_2-1 & \end{array} 
  \begin{array}{cc} \\ \\ \mathbf{2n_2+1} \end{array}
  \begin{array}{cc} & 2n_2+4 \\ 2n_2+3 & \\ 2n_2+3 & \end{array}
  \begin{array}{cc} & 2n_2+7 \\ 2n_2+6 & \\ 2n_2+6 & \end{array}
  \begin{array}{cc} \\ \\ \cdots \end{array} \right. \\[10pt]
  & \left. \begin{array}{cc} & 2n_2+3n_3+1 \\ 2n_2+3n_3 & \\ 2n_2+3n_3 & \end{array} 
  \begin{array}{cc} \\ \\ ( \textrm{ 1-clusters } \geq 2n_2+3n_3+3+\mu_1 ) \end{array}
  \right\}
 \end{align*}
 \begin{align*}
  \big\downarrow & \textrm{ 1 forward move on the 1-cluster } 2n_2 + 1
 \end{align*}
 \begin{align*}
  & \left\{  \begin{array}{cc} \\ \\ \cdots \end{array} \right.
  \begin{array}{cc} & \\ & 2n_2 \\ 2n_2-1 & \end{array} 
  \underbrace{\begin{array}{cc} \\ \\ \mathbf{2n_2+2} \end{array}
  \begin{array}{cc} & 2n_2+4 \\ 2n_2+3 & \\ 2n_2+3 & \end{array}}_{!}
  \begin{array}{cc} & 2n_2+7 \\ 2n_2+6 & \\ 2n_2+6 & \end{array}
  \begin{array}{cc} \\ \\ \cdots \end{array} \\[10pt]
  & \left. \begin{array}{cc} & 2n_2+3n_3+1 \\ 2n_2+3n_3 & \\ 2n_2+3n_3 & \end{array} 
  \begin{array}{cc} \\ \\ ( \textrm{ 1-clusters } \geq 2n_2+3n_3+3+\mu_1 ) \end{array}
  \right\} \textrm{(temporarily)}
 \end{align*}
 Here, the ! symbol signifies the violation of the difference condition at the indicated place.  
 As usual, we highlight the cluster(s) that is (are) being moved.  
 \begin{align*}
  \big\downarrow & \textrm{ adjustment }
 \end{align*}
 \begin{align*}
  & \left\{  \begin{array}{cc} \\ \\ \cdots \end{array} \right.
  \begin{array}{cc} & \\ & 2n_2 \\ 2n_2-1 & \end{array} 
  \begin{array}{cc} & 2n_2+3 \\ 2n_2+2 & \\ 2n_2+2 & \end{array}
  \underbrace{
  \begin{array}{cc} \\ \\ \mathbf{2n_2+5} \end{array}
  \begin{array}{cc} & 2n_2+7 \\ 2n_2+6 & \\ 2n_2+6 & \end{array}
  }_{!}
  \begin{array}{cc} \\ \\ \cdots \end{array} \\[10pt]
  & \left. \begin{array}{cc} & 2n_2+3n_3+1 \\ 2n_2+3n_3 & \\ 2n_2+3n_3 & \end{array} 
  \begin{array}{cc} \\ \\ ( \textrm{ 1-clusters } \geq 2n_2+3n_3+3+\mu_1 ) \end{array}
  \right\} \textrm{(temporarily)}
 \end{align*} 
 \begin{align*}
  \big\downarrow & \textrm{ after a total of } n_3 \textrm{ similar adjustments }
 \end{align*}
 \begin{align*}
  & \left\{  \begin{array}{cc} \\ \\ \cdots \end{array} \right.
  \begin{array}{cc} & \\ & 2n_2 \\ 2n_2-1 & \end{array} 
  \begin{array}{cc} & 2n_2+3 \\ 2n_2+2 & \\ 2n_2+2 & \end{array}
  \begin{array}{cc} & 2n_2+6 \\ 2n_2+5 & \\ 2n_2+5 & \end{array}
  \begin{array}{cc} \\ \\ \cdots \end{array} \\[10pt]
  & \left. 
  \begin{array}{cc} & 2n_2+3n_3 \\ 2n_2+3n_3-1 & \\ 2n_2+3n_3-1 & \end{array} 
  \begin{array}{cc} \\ \\ \mathbf{2n_2+3n_3+2} \end{array}
  \begin{array}{cc} \\ \\ ( \textrm{ 1-clusters } \geq 2n_2+3n_3+3+\mu_1 ) \end{array}
  \right\}
 \end{align*} 
 
 Notice that the adjustments do not alter the weight.  
 When the 1-cluster encounters a 3-cluster, temporarily violating the difference condition, 
 they switch places like in a \emph{puss-in-the-corner} game.  
 Three is added to the 1-cluster, and each part in the 3-cluster is decreased by one, 
 therefore preserving the total weight.  
 The process is repeated if there is another 3-cluster ahead.  
 
 We still need to add $\mu_1-1$ to the 1-cluster $2n_2 + 3n_2+ 2$, making it $2n_2+3n_3+\mu_1+1$, 
 respecting the difference condition in the configuration
 \begin{align*}
  & \left\{  \begin{array}{cc} \\ \\ \cdots \end{array} \right.
  \begin{array}{cc} & \\ & 2n_2 \\ 2n_2-1 & \end{array} 
  \begin{array}{cc} & 2n_2+3 \\ 2n_2+2 & \\ 2n_2+2 & \end{array}
  \begin{array}{cc} & 2n_2+6 \\ 2n_2+5 & \\ 2n_2+5 & \end{array}
  \begin{array}{cc} \\ \\ \cdots \end{array} \\[10pt]
  & \left. 
  \begin{array}{cc} & 2n_2+3n_3 \\ 2n_2+3n_3-1 & \\ 2n_2+3n_3-1 & \end{array} 
  \begin{array}{cc} \\ \\ ( \textrm{ 1-clusters } \geq 2n_2+3n_3+1+\mu_1 ) \end{array}
  \right\}
 \end{align*} 
 for $\mu_1 > 0$.  
 In case $\mu_1 = 0$, i.e. $\mu$ has less than $n_1$ positive parts, 
 The smallest 1-cluster stays in its original place at this stage.  
 
 Next, the forward moves on the 2-clusters are implemented.  
 The $i$th largest 2-cluster is moved the $i$th largest part of $\eta$ times forward.  
 For each positive part of $\eta$, we will prestidigitate the 2-clusters 
 through the 3-clusters as follows.  
 \begin{align*}
  & \left\{  \begin{array}{cc} \\ \\ \cdots \end{array} \right.
  \begin{array}{cc} & \\ & 2n_2-2 \\ 2n_2-3 & \end{array} 
  \begin{array}{cc} & \\ & \mathbf{2n_2} \\ \mathbf{2n_2-1} & \end{array} 
  \begin{array}{cc} & 2n_2+3 \\ 2n_2+2 & \\ 2n_2+2 & \end{array}
  \begin{array}{cc} & 2n_2+6 \\ 2n_2+5 & \\ 2n_2+5 & \end{array}
  \begin{array}{cc} \\ \\ \cdots \end{array} \\[10pt]
  & \left. 
  \begin{array}{cc} & 2n_2+3n_3 \\ 2n_2+3n_3-1 & \\ 2n_2+3n_3-1 & \end{array} 
  \begin{array}{cc} \\ \\ ( \textrm{ parts } \geq 2n_2+3n_3+2 ) \end{array}
  \right\}
 \end{align*} 
 \begin{align*}
  \big\downarrow & \textrm{ 1 forward move on the 2-cluster } 
    \begin{array}{cc} & 2n_2 \\ 2n_2-1 & \end{array}
 \end{align*}
 \begin{align*}
  & \left\{  \begin{array}{cc} \\ \\ \cdots \end{array} \right.
  \begin{array}{cc} & \\ & 2n_2-2 \\ 2n_2-3 & \end{array} 
  \underbrace{
  \begin{array}{cc} \\ \mathbf{2n_2} \\ \mathbf{2n_2} \end{array} 
  \begin{array}{cc} & 2n_2+3 \\ 2n_2+2 & \\ 2n_2+2 & \end{array}
  }_{!}
  \begin{array}{cc} & 2n_2+6 \\ 2n_2+5 & \\ 2n_2+5 & \end{array}
  \begin{array}{cc} \\ \\ \cdots \end{array} \\[10pt]
  & \left. 
  \begin{array}{cc} & 2n_2+3n_3 \\ 2n_2+3n_3-1 & \\ 2n_2+3n_3-1 & \end{array} 
  \begin{array}{cc} \\ \\ ( \textrm{ parts } \geq 2n_2+3n_3+2 ) \end{array}
  \right\}
  \textrm{ (temporarily) }
 \end{align*} 
 \begin{align*}
  \big\downarrow & \textrm{ adjustment } 
 \end{align*}
 \begin{align*}
  & \left\{  \begin{array}{cc} \\ \\ \cdots \end{array} \right.
  \begin{array}{cc} & \\ & 2n_2-2 \\ 2n_2-3 & \end{array} 
  \begin{array}{cc} & 2n_2+1 \\ 2n_2 & \\ 2n_2 & \end{array}
  \underbrace{
  \begin{array}{cc} \\ \mathbf{2n_2+3} \\ \mathbf{2n_2+3} \end{array} 
  \begin{array}{cc} & 2n_2+6 \\ 2n_2+5 & \\ 2n_2+5 & \end{array}
  }_{!}
  \begin{array}{cc} \\ \\ \cdots \end{array} \\[10pt]
  & \left. 
  \begin{array}{cc} & 2n_2+3n_3 \\ 2n_2+3n_3-1 & \\ 2n_2+3n_3-1 & \end{array} 
  \begin{array}{cc} \\ \\ ( \textrm{ parts } \geq 2n_2+3n_3+2 ) \end{array}
  \right\}
  \textrm{ (temporarily) }
 \end{align*} 
 \begin{align*}
  \big\downarrow & \textrm{ after } n_3 - 1 \textrm{ adjustments of the same kind } 
 \end{align*}
 \begin{align*}
  & \left\{  \begin{array}{cc} \\ \\ \cdots \end{array} \right.
  \begin{array}{cc} & \\ & 2n_2-2 \\ 2n_2-3 & \end{array} 
  \begin{array}{cc} & 2n_2+1 \\ 2n_2 & \\ 2n_2 & \end{array}
  \begin{array}{cc} & 2n_2+4 \\ 2n_2+3 & \\ 2n_2+3 & \end{array}
  \begin{array}{cc} \\ \\ \cdots \end{array} \\[10pt]
  & \left. 
  \begin{array}{cc} & 2n_2+3n_3-2 \\ 2n_2+3n_3-3 & \\ 2n_2+3n_3-3 & \end{array} 
  \begin{array}{cc} \\ \mathbf{2n_2+3n_3} \\ \mathbf{2n_2+3n_3} \end{array} 
  \begin{array}{cc} \\ \\ ( \textrm{ parts } \geq 2n_2+3n_3+2 ) \end{array}
  \right\}
 \end{align*} 
 At this point, the parts $\geq 2n_2+3n_3+2$ are all 1-clusters, so the difference conditions are met.  
 The initial move on each of the so many largest 2-clusters for each nonzero part of $\eta$ are these 
 prestidigitation of the 2-clusters through the 3-clusters.  
 After this initial move, the remaining moves are performed 
 as in the construction of the series side of Andrews-Gordon identities \cite{K-GordonMarking}.  
 
 There is one more condition on the collective forward moves on the 2-clusters.  
 $\eta$ cannot have repeated odd parts.  
 In other words, two successive 2-clusters cannot be moved the same odd number of times forward.  
 Let's see why this violates the difference condition.  
 
 Assume, on the contrary, that each of the two consecutive 2-clusters are to be moved $2r+1$ times forward.  
 After the initial prestidigitation through the 3-clusters, the 2-clusters will be
 \begin{align*}
  & \left\{  \begin{array}{cc} \\ \\ \cdots \end{array} \right.
  \begin{array}{cc} & k-2 \\ k-3 & \\ k-3 & \end{array} 
  \underbrace{
  \begin{array}{cc} \\ \mathbf{k} \\ \mathbf{k} \end{array} 
  \begin{array}{cc} \\ \mathbf{k+2} \\ \mathbf{k+2} \end{array} 
  }_{!}
  \left. \begin{array}{cc} \\ \\ ( \textrm{ parts } \geq k+4, \textrm{ all 1- or 2-clusters } ) \end{array}
  \right\}
  \begin{array}{cc} \\ \\ . \end{array} 
 \end{align*} 
 Then, the 2-clusters violating the difference at least three at distance three condition 
 will be double moved forward $r$ times each, each pair of double moves retaining the violation as 
 \begin{align*}
  \left\{  \begin{array}{cc} \\ \cdots \end{array} \right.
  \underbrace{
  \begin{array}{cc} \mathbf{k} \\ \mathbf{k} \end{array} 
  \begin{array}{cc} \mathbf{k+2} \\ \mathbf{k+2} \end{array} 
  }_{!}
  \left. \begin{array}{cc} \\ \cdots \end{array}
  \right\}
  \quad \longrightarrow \quad
  \left\{  \begin{array}{cc} \\ \cdots \end{array} \right.
  \underbrace{
  \begin{array}{cc} \mathbf{k+1} \\ \mathbf{k+1} \end{array} 
  \begin{array}{cc} \mathbf{k+3} \\ \mathbf{k+3} \end{array} 
  }_{!}
  \left. \begin{array}{cc} \\ \cdots \end{array}
  \right\}
  \begin{array}{cc} \\ , \end{array} 
 \end{align*} 
 or 
 \begin{align*}
  \left\{  \begin{array}{cc} \\ \cdots \end{array} \right.
  \underbrace{
  \begin{array}{cc} \mathbf{k} \\ \mathbf{k} \end{array} 
  \begin{array}{cc} \mathbf{k+2} \\ \mathbf{k+2} \end{array} 
  }_{!}
  \begin{array}{cc}  \\ k+4 \end{array} 
  \left. \begin{array}{cc} \\ \cdots \end{array}
  \right\}
  \quad \longrightarrow \quad
  \left\{  \begin{array}{cc} \\ \cdots \end{array} \right. 
  \begin{array}{cc} \\ k \end{array} 
  \underbrace{
  \begin{array}{cc} \mathbf{k+2} \\ \mathbf{k+2} \end{array} 
  \begin{array}{cc} \mathbf{k+4} \\ \mathbf{k+4} \end{array} 
  }_{!}
  \left. \begin{array}{cc} \\ \cdots \end{array}
  \right\}
  \begin{array}{cc} \\ . \end{array} 
 \end{align*} 
 In the latter possibility, the 2-clusters encountered a 1-cluster on the way.  
 
 However, the same even number of forward moves will leave the clusters as
 \begin{align*}
  \left\{  \begin{array}{cc} \\ \cdots \end{array}
  \begin{array}{cc} & \mathbf{k+1} \\ \mathbf{k} &  \end{array} 
  \begin{array}{cc} & \mathbf{k+3} \\ \mathbf{k+2} & \end{array} 
  \begin{array}{cc} \\ \cdots \end{array}
  \right\}
  \begin{array}{cc} \\ , \end{array}
 \end{align*} 
 conforming to the difference condition.  
 Or, one extra move on the larger cluster will yield
 \begin{align*}
  \left\{  \begin{array}{cc} \\ \cdots \end{array}
  \begin{array}{cc} \mathbf{k} \\ \mathbf{k} \end{array} 
  \begin{array}{cc} & \mathbf{k+3} \\ \mathbf{k+2} & \end{array} 
  \begin{array}{cc} \\ \cdots \end{array}
  \right\}
  \begin{array}{cc} \\ , \end{array}
 \end{align*} 
 again honoring the difference condition.  
 
 Thus, after the implementation of $\mu$ and $\eta$ as forward moves on the 1- and 2-clusters, 
 the intermediate partition looks like 
 \begin{align*}
  & \left\{  
  \begin{array}{cc} \\ & 2 \\ 1 &  \end{array} \right.
  \begin{array}{cc} \\ & 4 \\ 3 &  \end{array} 
  \begin{array}{cc} \\ \\ \cdots \end{array} 
  \begin{array}{cc} & \\ & 2s_2 \\ 2s_2-1 & \end{array} 
  \begin{array}{cc} & 2s_2+3 \\ 2s_2+2 & \\ 2s_2+2 & \end{array}
  \begin{array}{cc} & 2s_2+6 \\ 2s_2+5 & \\ 2s_2+5 & \end{array}
  \begin{array}{cc} \\ \\ \cdots \end{array} \\[10pt]
  & \left. 
  \begin{array}{cc} & \mathbf{2s_2+3n_3} \\ \mathbf{2s_2+3n_3-1} & \\ \mathbf{2s_2+3n_3-1} & \end{array} 
  \begin{array}{cc} \\ \\ ( \textrm{ parts } \geq 2s_2+3n_3+2, \textrm{ all 1- or 2-clusters} ) \end{array}
  \right\}
  \begin{array}{cc} \\ \\ , \end{array} 
 \end{align*} 
 for $s_2 \geq 0$, or 
 \begin{align*}
  & \left\{  
  \begin{array}{cc} \\ & 2 \\ 1 &  \end{array} \right.
  \begin{array}{cc} \\ & 4 \\ 3 &  \end{array} 
  \begin{array}{cc} \\ \\ \cdots \end{array} 
  \begin{array}{cc} & \\ & 2s_2 \\ 2s_2-1 & \end{array} 
  \begin{array}{cc} \\ \\ 2s_2+1 \end{array} 
  \begin{array}{cc} & 2s_2+4 \\ 2s_2+3 & \\ 2s_2+3 & \end{array}
  \begin{array}{cc} & 2s_2+7 \\ 2s_2+6 & \\ 2s_2+6 & \end{array}
  \begin{array}{cc} \\ \\ \cdots \end{array} \\[10pt]
  & \left. 
  \begin{array}{cc} & \mathbf{2s_2+3n_3+1} \\ \mathbf{2s_2+3n_3} & \\ \mathbf{2s_2+3n_3} & \end{array} 
  \begin{array}{cc} \\ \\ ( \textrm{ parts } \geq 2s_2+3n_3+3, \textrm{ all 1- or 2-clusters} ) \end{array}
  \right\}
  \begin{array}{cc} \\ \\ , \end{array} 
 \end{align*} 
 again, for $s_2 \geq 0$.  
 Both of the above satisfy the difference conditions.  
 The former possibly has a sediment, i.e. unmoved 2-clusters if $s_2 > 0$. 
 The latter has a sediment consisting of a 1-cluster, and if $s_2>0$, some 2-clusters as well.  
 The presence of unmoved 1- or 2-clusters, namely, sediments, 
 indicate that $\mu$ or $\eta$, respectively, have some zeros.  
 
 It remains to move the $i$th largest 3-cluster $\frac{1}{3}\times$(the $i$th largest part of $\nu$) 
 times forward.  
 Recall that $\nu$ consists of multiples of three.  
 The forward moves on the 3-clusters can be visualized in the following exclusive cases, 
 each adding three to the weight of the partition.  
 In each case, we assume that the initial configuration satisfies the necessary difference conditions.  
 \begin{align*}
  \left\{  \begin{array}{cc} \\ \\ ( \textrm{ parts } \leq k-2) \end{array} \right.
  \begin{array}{cc} & \mathbf{k+1} \\ \mathbf{k} & \\ \mathbf{k} & \end{array} 
  \left. \begin{array}{cc} \\ \\ (\textrm{ parts } \geq k+4) \end{array}
  \right\}
 \end{align*} 
 \begin{align*}
  \big\downarrow \textrm{ 1 forward move on the displayed 3-cluster}
 \end{align*} 
 \begin{align*}
  \left\{  \begin{array}{cc} \\ \\ ( \textrm{ parts } \leq k-2) \end{array} \right. 
  \begin{array}{cc} & \mathbf{k+2} \\ \mathbf{k+1} & \\ \mathbf{k+1} & \end{array} 
  \left. \begin{array}{cc} \\ \\ (\textrm{ parts } \geq k+4) \end{array}
  \right\}
 \end{align*} 
 Above, the part $k-2$ cannot repeat if it occurs, 
 since we assumed that the initial configuration satisfies the difference conditions.  
 $k+4$ may occur up to twice, but not thrice.  
 \begin{align*}
  \left\{  \begin{array}{cc} \\ \\ ( \textrm{ parts } \leq k-2) \end{array} \right.
  \begin{array}{cc} & \mathbf{k+1} \\ \mathbf{k} & \\ \mathbf{k} & \end{array} 
  \begin{array}{cc} \\ \\ k+3 \end{array} 
  \left. \begin{array}{cc} \\ \\ (\textrm{ parts } \geq k+5) \end{array}
  \right\}
 \end{align*} 
 \begin{align*}
  \big\downarrow \textrm{ 1 forward move on the displayed 3-cluster}
 \end{align*} 
 \begin{align*}
  \left\{  \begin{array}{cc} \\ \\ ( \textrm{ parts } \leq k-2) \end{array} \right. 
  \underbrace{
  \begin{array}{cc} & \mathbf{k+2} \\ \mathbf{k+1} & \\ \mathbf{k+1} & \end{array} 
  \begin{array}{cc} \\ \\ k+3 \end{array} 
  }_{!}
  \left. \begin{array}{cc} \\ \\ (\textrm{ parts } \geq k+5) \end{array}
  \right\} 
  \textrm{(temporarily)}
 \end{align*} 
 \begin{align*}
  \big\downarrow \textrm{ adjustment }
 \end{align*} 
 \begin{align*}
  \left\{  \begin{array}{cc} \\ \\ ( \textrm{ parts } \leq k-2) \end{array} \right. 
  \begin{array}{cc} \\ \\ k \end{array} 
  \begin{array}{cc} & \mathbf{k+3} \\ \mathbf{k+2} & \\ \mathbf{k+2} & \end{array} 
  \left. \begin{array}{cc} \\ \\ (\textrm{ parts } \geq k+5) \end{array}
  \right\} 
 \end{align*} 
 Above, again, the part $k-2$ can occur only once.  
 $k+5$ may occur twice, but not thrice.  
 
 \begin{align*}
  & \left\{  \begin{array}{cc} \\ \\ ( \textrm{ parts } \leq k-2) \end{array} \right.
  \begin{array}{cc} & \mathbf{k+1} \\ \mathbf{k} & \\ \mathbf{k} & \end{array} 
  \begin{array}{cc} \\ & k+4 \\ k+3 & \end{array} 
  \begin{array}{cc} \\ & k+6 \\ k+5 & \end{array} \\[10pt]
  & \begin{array}{cc} \\ \\ \cdots \end{array} 
  \begin{array}{cc} \\ & k+2s+2 \\ k+2s+1 & \end{array} 
  \left. \begin{array}{cc} \\ \\ (\textrm{ parts } \geq k+2s+4) \end{array}
  \right\}
 \end{align*} 
 \begin{align*}
  \big\downarrow \textrm{ 1 forward move on the displayed 3-cluster}
 \end{align*} 
 \begin{align*}
  & \left\{  \begin{array}{cc} \\ \\ ( \textrm{ parts } \leq k-2) \end{array} \right.
  \underbrace{
  \begin{array}{cc} & \mathbf{k+2} \\ \mathbf{k+1} & \\ \mathbf{k+1} & \end{array} 
  \begin{array}{cc} \\ & k+4 \\ k+3 & \end{array} 
  }_{!}
  \begin{array}{cc} \\ & k+6 \\ k+5 & \end{array} \\[10pt]
  & \begin{array}{cc} \\ \\ \cdots \end{array} 
  \begin{array}{cc} \\ & k+2s+2 \\ k+2s+1 & \end{array} 
  \left. \begin{array}{cc} \\ \\ (\textrm{ parts } \geq k+2s+4) \end{array}
  \right\} \textrm{ (temporarily) }
 \end{align*} 
 \begin{align*}
  \big\downarrow \textrm{ adjustment}
 \end{align*} 
 \begin{align*}
  & \left\{  \begin{array}{cc} \\ \\ ( \textrm{ parts } \leq k-2) \end{array} \right.
  \begin{array}{cc} \\ & k+1 \\ k & \end{array} 
  \underbrace{
  \begin{array}{cc} & \mathbf{k+4} \\ \mathbf{k+3} & \\ \mathbf{k+3} & \end{array} 
  \begin{array}{cc} \\ & k+6 \\ k+5 & \end{array} 
  }_{!}
  \begin{array}{cc} \\ \\ \cdots \end{array} \\[10pt]
  & \begin{array}{cc} \\ & k+2s+2 \\ k+2s+1 & \end{array} 
  \left. \begin{array}{cc} \\ \\ (\textrm{ parts } \geq k+2s+4) \end{array}
  \right\} \textrm{ (temporarily) }
 \end{align*} 
 \begin{align*}
  \big\downarrow \textrm{ after } s-1 \textrm{ similar adjustments}
 \end{align*} 
 \begin{align*}
  & \left\{  \begin{array}{cc} \\ \\ ( \textrm{ parts } \leq k-2) \end{array} \right.
  \begin{array}{cc} \\ & k+1 \\ k & \end{array} 
  \begin{array}{cc} \\ & k+3 \\ k+2 & \end{array} 
  \begin{array}{cc} \\ \\ \cdots \end{array} \\[10pt]
  & \begin{array}{cc} \\ & k+2s-1 \\ k+2s-2 & \end{array} 
  \begin{array}{cc} & \mathbf{k+2s+2} \\ \mathbf{k+2s+1} & \\ \mathbf{k+2s+1} & \end{array} 
  \left. \begin{array}{cc} \\ \\ (\textrm{ parts } \geq k+2s+4) \end{array}
  \right\}
 \end{align*} 
 for $s \geq 1$.  
 Again, if $k-2$ occurs in the above configuration, it cannot repeat.  
 $k+2s-4$ may repeat up to twice.  
 The adjustments do not alter the weight.  
 The adjustments are switching places of the 3- and 2-clusters when they are too close together.  
 There are three other cases summarized below.  
 They are very similar to the ones already explained, so we omit the details.  
 
 \begin{align*}
  & \left\{  \begin{array}{cc} \\ \\ ( \textrm{ parts } \leq k-2) \end{array} \right.
  \begin{array}{cc} & \mathbf{k+1} \\ \mathbf{k} & \\ \mathbf{k} & \end{array} 
  \begin{array}{cc} \\ & k+4 \\ k+3 & \end{array} 
  \begin{array}{cc} \\ & k+6 \\ k+5 & \end{array} \\[10pt]
  & \begin{array}{cc} \\ \\ \cdots \end{array} 
  \begin{array}{cc} \\ & k+2s+2 \\ k+2s+1 & \end{array} 
  \begin{array}{cc} \\ \\ k+2s+3 \end{array} 
  \left. \begin{array}{cc} \\ \\ (\textrm{ parts } \geq k+2s+5) \end{array}
  \right\}
 \end{align*} 
 \begin{align*}
  \big\downarrow \textrm{ 1 forward move on the displayed 3-cluster, followed by adjustments}
 \end{align*} 
 \begin{align*}
  & \left\{  \begin{array}{cc} \\ \\ ( \textrm{ parts } \leq k-2) \end{array} \right.
  \begin{array}{cc} \\ & k+1 \\ k & \end{array} 
  \begin{array}{cc} \\ & k+3 \\ k+2 & \end{array} 
  \begin{array}{cc} \\ \\ \cdots \end{array} \\[10pt]
  & \begin{array}{cc} \\ & k+2s-1 \\ k+2s-2 & \end{array} 
  \begin{array}{cc} \\ \\ k+2s \end{array} 
  \begin{array}{cc} & \mathbf{k+2s+3} \\ \mathbf{k+2s+2} & \\ \mathbf{k+2s+2} & \end{array} 
  \left. \begin{array}{cc} \\ \\ (\textrm{ parts } \geq k+2s+5) \end{array}
  \right\}
 \end{align*} 
 for $s \geq 0$.  
 
 \begin{align*}
  & \left\{  \begin{array}{cc} \\ \\ ( \textrm{ parts } \leq k-2) \end{array} \right.
  \begin{array}{cc} & \mathbf{k+1} \\ \mathbf{k} & \\ \mathbf{k} & \end{array} 
  \begin{array}{cc} \\ k+3 \\ k+3 \end{array} 
  \begin{array}{cc} \\ & k+6 \\ k+5 & \end{array}
  \begin{array}{cc} \\ & k+8 \\ k+7 & \end{array} \\[10pt]
  & \begin{array}{cc} \\ \\ \cdots \end{array} 
  \begin{array}{cc} \\ & k+2s+2 \\ k+2s+1 & \end{array} 
  \left. \begin{array}{cc} \\ \\ (\textrm{ parts } \geq k+2s+4) \end{array}
  \right\}
 \end{align*} 
 \begin{align*}
  \big\downarrow \textrm{ 1 forward move on the displayed 3-cluster, followed by adjustments}
 \end{align*} 
 \begin{align*}
  & \left\{  \begin{array}{cc} \\ \\ ( \textrm{ parts } \leq k-2) \end{array} \right.
  \begin{array}{cc} \\ k \\ k \end{array} 
  \begin{array}{cc} \\ & k+3 \\ k+2 & \end{array} 
  \begin{array}{cc} \\ & k+5 \\ k+4 & \end{array} 
  \begin{array}{cc} \\ \\ \cdots \end{array} \\[10pt]
  & \begin{array}{cc} \\ & k+2s-1 \\ k+2s-2 & \end{array} 
  \begin{array}{cc} & \mathbf{k+2s+2} \\ \mathbf{k+2s+1} & \\ \mathbf{k+2s+1} & \end{array} 
  \left. \begin{array}{cc} \\ \\ (\textrm{ parts } \geq k+2s+4) \end{array}
  \right\}
 \end{align*} 
 for $s \geq 1$, the case $s=1$ giving an empty streak after the smallest displayed 2-cluster.  
 
 \begin{align*}
  & \left\{  \begin{array}{cc} \\ \\ ( \textrm{ parts } \leq k-2) \end{array} \right.
  \begin{array}{cc} & \mathbf{k+1} \\ \mathbf{k} & \\ \mathbf{k} & \end{array} 
  \begin{array}{cc} \\ k+3 \\ k+3 \end{array} 
  \begin{array}{cc} \\ & k+6 \\ k+5 & \end{array}
  \begin{array}{cc} \\ & k+8 \\ k+7 & \end{array} \\[10pt]
  & \begin{array}{cc} \\ \\ \cdots \end{array} 
  \begin{array}{cc} \\ & k+2s+2 \\ k+2s+1 & \end{array} 
  \begin{array}{cc} \\ \\ k+2s+3 \end{array} 
  \left. \begin{array}{cc} \\ \\ (\textrm{ parts } \geq k+2s+5) \end{array}
  \right\}
 \end{align*} 
 \begin{align*}
  \big\downarrow \textrm{ 1 forward move on the displayed 3-cluster, followed by adjustments}
 \end{align*} 
 \begin{align*}
  & \left\{  \begin{array}{cc} \\ \\ ( \textrm{ parts } \leq k-2) \end{array} \right.
  \begin{array}{cc} \\ k \\ k \end{array} 
  \begin{array}{cc} \\ & k+3 \\ k+2 & \end{array} 
  \begin{array}{cc} \\ & k+5 \\ k+4 & \end{array} 
  \begin{array}{cc} \\ \\ \cdots \end{array} \\[10pt]
  & \begin{array}{cc} \\ & k+2s-1 \\ k+2s-2 & \end{array} 
  \begin{array}{cc} \\ \\ k+2s\end{array} 
  \begin{array}{cc} & \mathbf{k+2s+3} \\ \mathbf{k+2s+2} & \\ \mathbf{k+2s+2} & \end{array} 
  \left. \begin{array}{cc} \\ \\ (\textrm{ parts } \geq k+2s+5) \end{array}
  \right\}
 \end{align*} 
 for $s \geq 1$.  
 In the above three respective cases, $k+2s+4$ or $k+2s+5$ may repeat up to twice.  
 In none of the cases may $k-2$ repeat without violating the difference conditions in the initial configuration.  
 
 It is routine to check that in all of the above forward moves on the 3-cluster, 
 the preceding cluster, if any, may also move forward at least once.  
 This concludes the construction of $\lambda$ enumerated by $kr_5(n, m)$, 
 given $(\beta, \mu, \eta, \nu)$.  
 
 The reverse part of the construction is the decomposition of $\lambda$ into the quadruple 
 $(\beta, \mu, \eta, \nu)$ as described above.  
 First, we determine the number or $r$-clusters $n_r$ for $r = 1, 2, 3$ in $\lambda$.  
 
 We will first move the smallest 3-cluster, if any, backward so many times, 
 and call the number of required moves $\frac{1}{3}\times \nu_1$, 
 where $\nu_1$ is the smallest part of $\nu$.  
 $\nu_1$ will clearly be a multiple of three.  
 Each backward move on this cluster will deduct three from the weight of $\lambda$, 
 and the same amount will be registered as the weight of $\nu$.  
 
 $\lambda$ may start with either of the following sediments.  
 \begin{align*}
  \left\{  \begin{array}{cc} & 2 \\ 1 & \end{array} 
  \begin{array}{cc} & 3 \\ 4 & \end{array} 
  \begin{array}{cc} \\ \cdots \end{array} 
  \begin{array}{cc} & 2s \\ 2s-1 & \end{array} 
  \begin{array}{cc} \\ (\textrm{ parts } \geq 2s+2) \end{array}
  \right\}
  \begin{array}{cc} \\ , \end{array}
 \end{align*} 
 or
 \begin{align*}
  \left\{  \begin{array}{cc} & 2 \\ 1 & \end{array} 
  \begin{array}{cc} & 3 \\ 4 & \end{array} 
  \begin{array}{cc} \\ \cdots \end{array} 
  \begin{array}{cc} & 2s \\ 2s-1 & \end{array} 
  \begin{array}{cc} \\ 2s+1 \end{array} 
  \begin{array}{cc} \\ (\textrm{ parts } \geq 2s+3) \end{array}
  \right\}
  \begin{array}{cc} \\ , \end{array}
 \end{align*} 
 for $s \geq 0$, the case $s = 0$ corresponding to having no 2-clusters in the sediments.  
 In the above two events, the backward moves on the smallest 3-cluster will stow it as
 \begin{align*}
  \left\{  \begin{array}{cc} \\ & 2 \\ 1 & \end{array} 
  \begin{array}{cc} \\ & 3 \\ 4 & \end{array} 
  \begin{array}{cc} \\ \\ \cdots \end{array} 
  \begin{array}{cc} \\ & 2s \\ 2s-1 & \end{array} 
  \begin{array}{cc} & \mathbf{2s+3} \\ \mathbf{2s+2} & \\ \mathbf{2s+2} & \end{array} 
  \begin{array}{cc} \\ \\ (\textrm{ parts } \geq 2s+5) \end{array}
  \right\}
  \begin{array}{cc} \\ \\ , \end{array}
 \end{align*} 
 or
 \begin{align*}
  \left\{  \begin{array}{cc} \\ & 2 \\ 1 & \end{array} 
  \begin{array}{cc} \\ & 3 \\ 4 & \end{array} 
  \begin{array}{cc} \\ \\ \cdots \end{array} 
  \begin{array}{cc} \\ & 2s \\ 2s-1 & \end{array} 
  \begin{array}{cc} \\ \\ 2s+1 \end{array} 
  \begin{array}{cc} & \mathbf{2s+4} \\ \mathbf{2s+3} & \\ \mathbf{2s+3} & \end{array} 
  \begin{array}{cc} \\ \\ (\textrm{ parts } \geq 2s+6) \end{array}
  \right\}
  \begin{array}{cc} \\ \\ , \end{array}
 \end{align*} 
 respectively.  
 If the smallest 3-cluster is already one of the displayed ones above, 
 we declare $\nu_1 = 0$.  
 
 Let's describe the backward moves and adjustments in the exclusive cases below.  
 Then, we will argue that the 3-cluster cannot go further back.  
 \begin{align*}
  \left\{  \begin{array}{cc} \\ \\ ( \textrm{ parts } \leq k-3) \end{array} \right. 
  \begin{array}{cc} & \mathbf{k+1} \\ \mathbf{k} & \\ \mathbf{k} & \end{array} 
  \left. \begin{array}{cc} \\ \\ (\textrm{ parts } \geq k+3) \end{array}
  \right\}
 \end{align*} 
 \begin{align*}
  \big\downarrow \textrm{ 1 backward move on the displayed 3-cluster}
 \end{align*} 
 \begin{align*}
  \left\{  \begin{array}{cc} \\ \\ ( \textrm{ parts } \leq k-3) \end{array} \right.
  \begin{array}{cc} & \mathbf{k} \\ \mathbf{k-1} & \\ \mathbf{k-1} & \end{array} 
  \left. \begin{array}{cc} \\ \\ (\textrm{ parts } \geq k+3) \end{array}
  \right\}
 \end{align*} 
 Above, $k-3$ will be assumed to not repeat, 
 so that the difference conditions are met in the terminal configuration.  
 However, $k-3$ may very well repeat without violating the difference conditions in the initial configuration.  
 That case will be treated below.  $k+3$ may repeat up to twice.  
 \begin{align*}
  \left\{  \begin{array}{cc} \\ \\ ( \textrm{ parts } \leq k-4) \end{array} \right. 
  \begin{array}{cc} \\ \\ k-2 \end{array} 
  \begin{array}{cc} & \mathbf{k+1} \\ \mathbf{k} & \\ \mathbf{k} & \end{array} 
  \left. \begin{array}{cc} \\ \\ (\textrm{ parts } \geq k+3) \end{array}
  \right\} 
 \end{align*} 
 \begin{align*}
  \big\downarrow \textrm{ 1 forward move on the displayed 3-cluster}
 \end{align*} 
 \begin{align*}
  \left\{  \begin{array}{cc} \\ \\ ( \textrm{ parts } \leq k-4) \end{array} \right. 
  \underbrace{
  \begin{array}{cc} \\ \\ k-2 \end{array} 
  \begin{array}{cc} & \mathbf{k} \\ \mathbf{k-1} & \\ \mathbf{k-1} & \end{array} 
  }_{!}
  \left. \begin{array}{cc} \\ \\ (\textrm{ parts } \geq k+3) \end{array}
  \right\} 
  \textrm{(temporarily)}
 \end{align*} 
 \begin{align*}
  \big\downarrow \textrm{ adjustment }
 \end{align*} 
 \begin{align*}
  \left\{  \begin{array}{cc} \\ \\ ( \textrm{ parts } \leq k-4) \end{array} \right.
  \begin{array}{cc} & \mathbf{k-1} \\ \mathbf{k-2} & \\ \mathbf{k-2} & \end{array} 
  \begin{array}{cc} \\ \\ k+1 \end{array} 
  \left. \begin{array}{cc} \\ \\ (\textrm{ parts } \geq k+3) \end{array}
  \right\}
 \end{align*} 
 Observe that the adjustment does not change the weight of the partition.  
 Again, we assume that $k-4$ is not repeated, 
 so that the difference condition is not violated in the terminal configuration.  
 The case of repeating $k-4$'s will be treated below.  
 $k+3$ may repeat up to twice, but not thrice.  
 
 \begin{align*}
  & \left\{  \begin{array}{cc} \\ \\ ( \textrm{ parts } \leq k-2s-3) \end{array} \right.
  \begin{array}{cc} \\ & k-2s \\ k-2s-1 & \end{array} 
  \begin{array}{cc} \\ & k-2s+2 \\ k-2s+1 & \end{array} 
  \begin{array}{cc} \\ \\ \cdots \end{array} \\[10pt]
  & \begin{array}{cc} \\ & k-2 \\ k-3 & \end{array} 
  \begin{array}{cc} & \mathbf{k+1} \\ \mathbf{k} & \\ \mathbf{k} & \end{array} 
  \left. \begin{array}{cc} \\ \\ (\textrm{ parts } \geq k+3) \end{array}
  \right\}
 \end{align*} 
 \begin{align*}
  \big\downarrow \textrm{ 1 backward move on the displayed 3-cluster}
 \end{align*} 
 \begin{align*}
  & \left\{  \begin{array}{cc} \\ \\ ( \textrm{ parts } \leq k-2s-3) \end{array} \right.
  \begin{array}{cc} \\ & k-2s \\ k-2s-1 & \end{array} 
  \begin{array}{cc} \\ & k-2s+2 \\ k-2s+1 & \end{array} 
  \begin{array}{cc} \\ \\ \cdots \end{array} \\[10pt]
  & \underbrace{
  \begin{array}{cc} \\ & k-2 \\ k-3 & \end{array} 
  \begin{array}{cc} & \mathbf{k} \\ \mathbf{k-1} & \\ \mathbf{k-1} & \end{array} 
  }_{!}
  \left. \begin{array}{cc} \\ \\ (\textrm{ parts } \geq k+3) \end{array}
  \right\} \textrm{ (temporarily) }
 \end{align*} 
 \begin{align*}
  \big\downarrow \textrm{ adjustment}
 \end{align*} 
 \begin{align*}
  & \left\{  \begin{array}{cc} \\ \\ ( \textrm{ parts } \leq k-2s-3) \end{array} \right.
  \begin{array}{cc} \\ & k-2s \\ k-2s-1 & \end{array} 
  \begin{array}{cc} \\ & k-2s+2 \\ k-2s+1 & \end{array} 
  \begin{array}{cc} \\ \\ \cdots \end{array} \\[10pt]
  & \underbrace{
  \begin{array}{cc} \\ & k-4 \\ k-5 & \end{array} 
  \begin{array}{cc} & \mathbf{k-2} \\ \mathbf{k-3} & \\ \mathbf{k-3} & \end{array} 
  }_{!}
  \begin{array}{cc} \\ & k+1 \\ k & \end{array} 
  \left. \begin{array}{cc} \\ \\ (\textrm{ parts } \geq k+3) \end{array}
  \right\} \textrm{ (temporarily) }
 \end{align*} 
 \begin{align*}
  \big\downarrow \textrm{ after } s-1 \textrm{ similar adjustments}
 \end{align*} 
 \begin{align*}
  & \left\{  \begin{array}{cc} \\ \\ ( \textrm{ parts } \leq k-2s-3) \end{array} \right.
  \begin{array}{cc} & \mathbf{k-2s} \\ \mathbf{k-2s-1} & \\ \mathbf{k-2s-1} & \end{array} 
  \begin{array}{cc} \\ & k-2s+3 \\ k-2s+2 & \end{array} \\[10pt]
  & \begin{array}{cc} \\ & k-2s+5 \\ k+2s+4 & \end{array} 
  \begin{array}{cc} \\ \\ \cdots \end{array} 
  \begin{array}{cc} \\ & k+1 \\ k & \end{array} 
  \left. \begin{array}{cc} \\ \\ (\textrm{ parts } \geq k+3) \end{array}
  \right\}
 \end{align*} 
 for $s \geq 1$.  
 Here, again, we will assume that $k-2s-3$ does not repeat, 
 so that the terminal configuration conforms to the difference conditions set forth by $kr_5(n, m)$.  
 $k+3$ may repeat up to twice.  
 As before, the adjustments do not alter the weight.  
 The three cases below are very similar to the last one.  
 They cover the cases of repeated smaller parts as well.  
 We leave the details to the reader.  
 
 \begin{align*}
  & \left\{  \begin{array}{cc} \\ \\ ( \textrm{ parts } \leq k-2s-4) \end{array} \right.
  \begin{array}{cc} \\ & k-2s-1 \\ k-2s-2 & \end{array} 
  \begin{array}{cc} \\ & k-2s+1 \\ k-2s & \end{array} 
  \begin{array}{cc} \\ \\ \cdots \end{array} \\[10pt]
  & \begin{array}{cc} \\ & k-3 \\ k-4 & \end{array} 
  \begin{array}{cc} \\ \\ k-2 \end{array} 
  \begin{array}{cc} & \mathbf{k+1} \\ \mathbf{k} & \\ \mathbf{k} & \end{array} 
  \left. \begin{array}{cc} \\ \\ (\textrm{ parts } \geq k+3) \end{array}
  \right\}
 \end{align*} 
 \begin{align*}
  \big\downarrow \textrm{ 1 backward move on the displayed 3-cluster, followed by adjustments}
 \end{align*} 
 \begin{align*}
  & \left\{  \begin{array}{cc} \\ \\ ( \textrm{ parts } \leq k-2s-4) \end{array} \right.
  \begin{array}{cc} & \mathbf{k-2s-1} \\ \mathbf{k-2s-2} & \\ \mathbf{k-2s-2} & \end{array} 
  \begin{array}{cc} \\ & k-2s+2 \\ k-2s+1 & \end{array} \\[10pt]
  & \begin{array}{cc} \\ & k-2s+4 \\ k-2s+3 & \end{array} 
  \begin{array}{cc} \\ \\ \cdots \end{array} 
  \begin{array}{cc} \\ & k \\ k-1 & \end{array} 
  \begin{array}{cc} \\ \\ k+1 \end{array} 
  \left. \begin{array}{cc} \\ \\ (\textrm{ parts } \geq k+3) \end{array}
  \right\}
 \end{align*} 
 for $s \geq 1$.  
 
 \begin{align*}
  & \left\{  \begin{array}{cc} \\ \\ ( \textrm{ parts } \leq k-2s-4) \end{array} \right.
  \begin{array}{cc} \\ k-2s-1 \\ k-2s-1 \end{array} 
  \begin{array}{cc} \\ & k-2s+2 \\ k-2s+1 & \end{array} \\[10pt]
  & \begin{array}{cc} \\ & k-2s+4 \\ k-2s+3 & \end{array} 
  \begin{array}{cc} \\ \\ \cdots \end{array} 
  \begin{array}{cc} \\ & k-3\\ k-2 & \end{array} 
  \begin{array}{cc} & \mathbf{k+1} \\ \mathbf{k} & \\ \mathbf{k} & \end{array} 
  \left. \begin{array}{cc} \\ \\ (\textrm{ parts } \geq k+3) \end{array}
  \right\}
 \end{align*} 
 \begin{align*}
  \big\downarrow \textrm{ 1 backward move on the displayed 3-cluster, followed by adjustments}
 \end{align*} 
 \begin{align*}
  & \left\{  \begin{array}{cc} \\ \\ ( \textrm{ parts } \leq k-2s-4) \end{array} \right.
  \begin{array}{cc} & \mathbf{k-2s} \\ \mathbf{k-2s-1} & \\ \mathbf{k-2s-1} & \end{array} 
  \begin{array}{cc} \\ k-2s+2 \\ k-2s+2 \end{array} 
  \begin{array}{cc} \\ & k-2s+5 \\ k-2s+4 & \end{array} \\[10pt]
  & \begin{array}{cc} \\ & k-2s+7 \\ k-2s+6 & \end{array}
  \begin{array}{cc} \\ \\ \cdots \end{array} 
  \begin{array}{cc} \\ & k+1 \\ k & \end{array} 
  \left. \begin{array}{cc} \\ \\ (\textrm{ parts } \geq k+3) \end{array}
  \right\}
 \end{align*} 
 for $s \geq 1$, the case $s=1$ giving an empty streak after the smallest displayed 2-cluster.  
 
 \begin{align*}
  & \left\{  \begin{array}{cc} \\ \\ ( \textrm{ parts } \leq k-2s-4) \end{array} \right.
  \begin{array}{cc} \\ k-2s-2 \\ k-2s-2 \end{array} 
  \begin{array}{cc} \\ & k-2s+1 \\ k-2s & \end{array} 
  \begin{array}{cc} \\ & k-2s+3 \\ k-2s+2 & \end{array} 
  \begin{array}{cc} \\ \\ \cdots \end{array} \\[10pt]
  & \begin{array}{cc} \\ & k-3 \\ k-4 & \end{array} 
  \begin{array}{cc} \\ \\ k-2\end{array} 
  \begin{array}{cc} & \mathbf{k+1} \\ \mathbf{k} & \\ \mathbf{k} & \end{array} 
  \left. \begin{array}{cc} \\ \\ (\textrm{ parts } \geq k+3) \end{array}
  \right\}
 \end{align*} 
 \begin{align*}
  \big\downarrow \textrm{ 1 backward move on the displayed 3-cluster, followed by adjustments}
 \end{align*} 
 \begin{align*}
  & \left\{  \begin{array}{cc} \\ \\ ( \textrm{ parts } \leq k-2s-4) \end{array} \right.
  \begin{array}{cc} & \mathbf{k-2s-1} \\ \mathbf{k-2s-2} & \\ \mathbf{k-2s-2} & \end{array} 
  \begin{array}{cc} \\ k-2s+1 \\ k-2s+1 \end{array} 
  \begin{array}{cc} \\ & k-2s+4 \\ k-2s+3 & \end{array} \\[10pt]
  & \begin{array}{cc} \\ & k-2s+6 \\ k-2s+5 & \end{array}
  \begin{array}{cc} \\ \\ \cdots \end{array} 
  \begin{array}{cc} \\ & k \\ k-1 & \end{array} 
  \begin{array}{cc} \\ \\ k+1 \end{array} 
  \left. \begin{array}{cc} \\ \\ (\textrm{ parts } \geq k+3) \end{array}
  \right\}
 \end{align*} 
 for $s \geq 1$.  
 Above, $k+3$ may repeat twice, but not thrice.  
 In none of the respective three cases above, do $k-2s-4$ or $k-2s-3$ repeat, if they occur.  
 Notice that the omitted cases of repetition are taken care of by the last two cases.  
 
 Again, it is routine to verify that one backward move on a 3-cluster 
 allows at least one move on the succeding 3-cluster.  
 
 Once we complete the backward moves on the smallest 3-cluster, 
 we repeat the same process for the next smallest, 
 and move it backward as far as it can go, 
 recording the number of moves as $\frac{1}{3} \times \nu_2$, $\frac{1}{3} \times \nu_3$, \ldots, 
 $\frac{1}{3} \times \nu_{n_3}$.  
 This will give us the partition $\nu$ with $n_3$ parts (counting zeros) into multiples of three.  
 The intermediate partition looks like
 \begin{align}
 \nonumber
  & \left\{  \begin{array}{cc} \\ & 2 \\ 1 & \end{array} \right.
  \begin{array}{cc} \\ & 4 \\ 3 & \end{array} 
  \begin{array}{cc} \\ \\ \cdots \end{array} 
  \begin{array}{cc} \\ & 2s \\ 2s-1 & \end{array} 
  \begin{array}{cc} & 2s+3 \\ 2s+2 & \\ 2s+2 & \end{array} 
  \begin{array}{cc} & 2s+6 \\ 2s+5 & \\ 2s+5 & \end{array} 
  \begin{array}{cc} \\ \\ \cdots \end{array}  \\[10pt]
 \label{intermPtnkr_5-back-Conf1}
  & \begin{array}{cc} & 2s+3n_3 \\ 2s+3n_3-1 & \\ 2s+3n_3-1 & \end{array} 
  \left. \begin{array}{cc} \\ \\ (\textrm{ parts } \geq 2s+3n_3+2, \textrm{ all 1- or 2-clusters }) \end{array}
  \right\}
 \end{align}
 for $s \geq 0$, $s = 0$ being the case of no 2-clusters smaller than the 3-clusters, or
 \begin{align}
 \nonumber
  & \left\{  \begin{array}{cc} \\ & 2 \\ 1 & \end{array} \right.
  \begin{array}{cc} \\ & 4 \\ 3 & \end{array} 
  \begin{array}{cc} \\ \\ \cdots \end{array} 
  \begin{array}{cc} \\ & 2s \\ 2s-1 & \end{array} 
  \begin{array}{cc} \\ \\ 2s+1 \end{array} 
  \begin{array}{cc} & 2s+4 \\ 2s+3 & \\ 2s+3 & \end{array} 
  \begin{array}{cc} & 2s+7 \\ 2s+6 & \\ 2s+6 & \end{array} 
  \begin{array}{cc} \\ \\ \cdots \end{array} \\[10pt]
 \label{intermPtnkr_5-back-Conf2}
  & \begin{array}{cc} & 2s+3n_3+1 \\ 2s+3n_3 & \\ 2s+3n_3 & \end{array} 
  \left. \begin{array}{cc} \\ \\ (\textrm{ parts } \geq 2s+3n_3+3, \textrm{ all 1- or 2-clusters }) \end{array}
  \right\}
 \end{align}
 for $s \geq 0$.  
 If one or more 3-clusters were in the indicated places, 
 we would have set $\eta_1 = 0$, $\eta_2 = 0$, \ldots, as many as necessary.  
 
 Notice that the cases for the backward moves on the 3-clusters are inverses of the cases 
 for the forward moves on the 3-clusters, in their respective order, 
 after necessary shifts of all parts.  
 The rulebreaking in the middle temporary cases are slightly different; 
 however, the initial cases become the terminal cases, and vice-versa.  
 We find the given descriptions more intuitive.  
 
 For a moment, suppose we wanted to move the smallest 3-cluster backward one more time, 
 and do some adjustments so as to retain the difference conditions imposed by $kr_5(n, m)$, 
 in the intermediate partition \eqref{intermPtnkr_5-back-Conf1}.  
 \begin{align*}
  \left\{  \begin{array}{cc} \\ & 2 \\ 1 & \end{array} \right.
  \begin{array}{cc} \\ & 4 \\ 3 & \end{array} 
  \begin{array}{cc} \\ \\ \cdots \end{array} 
  \begin{array}{cc} \\ & 2s \\ 2s-1 & \end{array} 
  \begin{array}{cc} & \mathbf{2s+3} \\ \mathbf{2s+2} & \\ \mathbf{2s+2} & \end{array} 
  \left. \begin{array}{cc} \\ \\ (\textrm{ parts } \geq 2s+5) \end{array}
  \right\}
 \end{align*}
 \begin{align*}
  \big\downarrow \textrm{ 1 backward move on the displayed 3-cluster, followed by adjustments}
 \end{align*} 
 \begin{align*}
  \left\{  \begin{array}{cc} & \mathbf{2} \\ \mathbf{1} & \\ \mathbf{1} & \end{array} 
  \begin{array}{cc} \\ & 5 \\ 4 & \end{array} \right.
  \begin{array}{cc} \\ & 7 \\ 6 & \end{array} 
  \begin{array}{cc} \\ \\ \cdots \end{array} 
  \begin{array}{cc} \\ & 2s \\ 2s-1 & \end{array} 
  \left. \begin{array}{cc} \\ \\ (\textrm{ parts } \geq 2s+5) \end{array}
  \right\}
 \end{align*}
 This creates two occurrences of 1's, 
 which is forbidden by the conditions of $kr_5(n, m)$, 
 and shows us that the 3-clusters are indeed as small as they can be.  
 
 Now, in either \eqref{intermPtnkr_5-back-Conf1} or \eqref{intermPtnkr_5-back-Conf2}, 
 we continue with implementing the backward moves on the 2-clusters.  
 In either configuration, if $s > 0$, we set $\eta_1$ $= \eta_2$ $= \cdots$ $= \eta_s$ $= 0$.  
 This because the smallest $s$ 2-clusters are already minimal.  
 They cannot be moved further back.  
 We then move the $(s+1)$th smallest 2-cluster using the backward moves of the 2nd kind (ref), bringing it to
 \begin{align*}
  & \left\{  \begin{array}{cc} \\ & 2 \\ 1 & \end{array} \right.
  \begin{array}{cc} \\ & 4 \\ 3 & \end{array} 
  \begin{array}{cc} \\ \\ \cdots \end{array} 
  \begin{array}{cc} \\ & 2s \\ 2s-1 & \end{array} 
  \begin{array}{cc} & 2s+3 \\ 2s+2 & \\ 2s+2 & \end{array} 
  \begin{array}{cc} & 2s+6 \\ 2s+5 & \\ 2s+5 & \end{array} 
  \begin{array}{cc} \\ \\ \cdots \end{array}  \\[10pt]
  & \begin{array}{cc} & 2s+3n_3 \\ 2s+3n_3-1 & \\ 2s+3n_3-1 & \end{array} 
  \begin{array}{cc} \\ \mathbf{2s+3n_3+2} \\ \mathbf{2s+3n_3+2} \end{array} 
  \left. \begin{array}{cc} \\ \\ (\textrm{ parts } \geq 2s+3n_3+4, \textrm{ all 1- or 2-clusters }) \end{array}
  \right\} 
  \begin{array}{cc} \\ \\ , \end{array} 
 \end{align*}
 or
 \begin{align*}
  & \left\{  \begin{array}{cc} \\ & 2 \\ 1 & \end{array} \right.
  \begin{array}{cc} \\ & 4 \\ 3 & \end{array} 
  \begin{array}{cc} \\ \\ \cdots \end{array} 
  \begin{array}{cc} \\ & 2s \\ 2s-1 & \end{array} 
  \begin{array}{cc} \\ \\ 2s+1 \end{array} 
  \begin{array}{cc} & 2s+4 \\ 2s+3 & \\ 2s+3 & \end{array} 
  \begin{array}{cc} & 2s+7 \\ 2s+6 & \\ 2s+6 & \end{array} 
  \begin{array}{cc} \\ \\ \cdots \end{array}  \\[10pt]
  & \begin{array}{cc} & 2s+3n_3+1 \\ 2s+3n_3 & \\ 2s+3n_3 & \end{array} 
  \begin{array}{cc} \\ \mathbf{2s+3n_3+3} \\ \mathbf{2s+3n_3+3} \end{array} 
  \left. \begin{array}{cc} \\ \\ (\textrm{ parts } \geq 2s+3n_3+5, \textrm{ all 1- or 2-clusters }) \end{array}
  \right\} 
  \begin{array}{cc} \\ \\ . \end{array} 
 \end{align*}
 We record the number of required moves as $\eta_{s+1}-1$.  
 If $n_3 > 0$, the final backward move involves prestidigitating the 2-cluster 
 through the 3-clusters as follows. 
 After one more backward move of the 2nd kind on the $(s+1)$th smallest 2-cluster, 
 say in the former configuration, 
 \begin{align*}
  & \left\{  \begin{array}{cc} \\ & 2 \\ 1 & \end{array} \right.
  \begin{array}{cc} \\ & 4 \\ 3 & \end{array} 
  \begin{array}{cc} \\ \\ \cdots \end{array} 
  \begin{array}{cc} \\ & 2s \\ 2s-1 & \end{array} 
  \begin{array}{cc} & 2s+3 \\ 2s+2 & \\ 2s+2 & \end{array} 
  \begin{array}{cc} & 2s+6 \\ 2s+5 & \\ 2s+5 & \end{array} 
  \begin{array}{cc} \\ \\ \cdots \end{array}  \\[10pt]
  & \underbrace{
  \begin{array}{cc} & 2s+3n_3 \\ 2s+3n_3-1 & \\ 2s+3n_3-1 & \end{array} 
  \begin{array}{cc} \\ & \mathbf{2s+3n_3+2} \\ \mathbf{2s+3n_3+1} & \end{array} 
  }_{!} \\[10pt]
  & \left. \begin{array}{cc} \\ \\ (\textrm{ parts } \geq 2s+3n_3+4, \textrm{ all 1- or 2-clusters }) \end{array}
  \right\}  \textrm{ (temporarily) }
 \end{align*}
 \begin{align*}
  \big\downarrow \textrm{ adjustment }
 \end{align*} 
 \begin{align*}
  & \left\{  \begin{array}{cc} \\ & 2 \\ 1 & \end{array} \right.
  \begin{array}{cc} \\ & 4 \\ 3 & \end{array} 
  \begin{array}{cc} \\ \\ \cdots \end{array} 
  \begin{array}{cc} \\ & 2s \\ 2s-1 & \end{array} 
  \begin{array}{cc} & 2s+3 \\ 2s+2 & \\ 2s+2 & \end{array} 
  \begin{array}{cc} & 2s+6 \\ 2s+5 & \\ 2s+5 & \end{array} 
  \begin{array}{cc} \\ \\ \cdots \end{array}  \\[10pt]
  & \underbrace{
  \begin{array}{cc} & 2s+3n_3-3 \\ 2s+3n_3-4 & \\ 2s+3n_3-4 & \end{array} 
  \begin{array}{cc} \\ & \mathbf{2s+3n_3-1} \\ \mathbf{2s+3n_3-2} & \end{array} 
  }_{!} \\[10pt]
  & \begin{array}{cc} & 2s+3n_3+2 \\ 2s+3n_3+1 & \\ 2s+3n_3+1 & \end{array}
  \left. \begin{array}{cc} \\ \\ (\textrm{ parts } \geq 2s+3n_3+4) \end{array}
  \right\} \textrm{ (temporarily) }
 \end{align*}
 \begin{align*}
  \big\downarrow \textrm{ after } n_3 -1 \textrm{ similar adjustments }
 \end{align*} 
 \begin{align*}
  & \left\{  \begin{array}{cc} \\ & 2 \\ 1 & \end{array} \right.
  \begin{array}{cc} \\ & 4 \\ 3 & \end{array} 
  \begin{array}{cc} \\ \\ \cdots \end{array} 
  \begin{array}{cc} \\ & \mathbf{2s+2} \\ \mathbf{2s+1} & \end{array} 
  \begin{array}{cc} & 2s+5 \\ 2s+4 & \\ 2s+4 & \end{array} 
  \begin{array}{cc} & 2s+8 \\ 2s+7 & \\ 2s+7 & \end{array} 
  \begin{array}{cc} \\ \\ \cdots \end{array}  \\[10pt]
  & \begin{array}{cc} & 2s+3n_3+2 \\ 2s+3n_3+1 & \\ 2s+3n_3+1 & \end{array} 
  \left. \begin{array}{cc} \\ \\ (\textrm{ parts } \geq 2s+3n_3+4, \textrm{ all 1- or 2-clusters }) \end{array}
  \right\} 
  \begin{array}{cc} \\ \\ , \end{array} 
 \end{align*}
 and the $(s+1)$st 2-cluster is stowed in its proper place.  
 This determines $\eta_{s+1}$, which is positive.  
 The second case is almost the same except that the Gordon marking has to be updated after the final adjustment.  
 We repeat the process, and record $\eta_{s+2}$, $\eta_{s+3}$, \ldots, $\eta_{n_2}$.  
 We note that the total weight of $\lambda$ and $\eta$ remain constant, 
 because any drop in the weight of $\lambda$ is registered in $\eta$ in the same amount, 
 thanks to the definition of the backward move of the 2nd kind, 
 namely, Definition \ref{defBackwdMove}.  
 
 At this point, we should justify the fact that $\eta$ cannot have repeated odd parts.  
 Initially, and after any moves followed by a streak of adjustments, 
 $\lambda$ has satisfied the difference conditions given by $kr_5(n ,m)$.  
 Also, the moves on the 2- and 3-clusters are performed in the exact reverse order.  
 As we showed in the forward moves on the 2-clusters, 
 any repeated odd part in $\eta$ will result in a violation of the said difference conditions.  
 Moreover, the violation precisely occurs when $\eta$ has repeated odd parts.  
 Thus, $\eta$ as constructed above cannot have repeated odd parts.  
 
 So far, the intermediate partition looks like
 \begin{align}
 \nonumber
  & \left\{  \begin{array}{cc} \\ & 2 \\ 1 & \end{array} \right.
  \begin{array}{cc} \\ & 4 \\ 3 & \end{array} 
  \begin{array}{cc} \\ \\ \cdots \end{array} 
  \begin{array}{cc} \\ & 2n_2 \\ 2n_2-1 & \end{array} 
  \begin{array}{cc} & 2n_2+3 \\ 2n_2+2 & \\ 2n_2+2 & \end{array} 
  \begin{array}{cc} & 2n_2+6 \\ 2n_2+5 & \\ 2n_2+5 & \end{array} 
  \begin{array}{cc} \\ \\ \cdots \end{array}  \\[10pt]
 \label{intermPtnkr_5-back-Conf3}
  & \begin{array}{cc} & 2n_2+3n_3 \\ 2n_2+3n_3-1 & \\ 2n_2+3n_3-1 & \end{array} 
  \left. \begin{array}{cc} \\ \\ (\textrm{ parts } \geq 2n_2+3n_3+2, \textrm{ all 1-clusters }) \end{array}
  \right\}
  \begin{array}{cc} \\ \\ , \end{array} 
 \end{align}
 or
 \begin{align}
 \nonumber
  & \left\{  \begin{array}{cc} \\ & 2 \\ 1 & \end{array} \right.
  \begin{array}{cc} \\ & 4 \\ 3 & \end{array} 
  \begin{array}{cc} \\ \\ \cdots \end{array} 
  \begin{array}{cc} \\ & 2n_2 \\ 2n_2-1 & \end{array} 
  \begin{array}{cc} \\ \\ 2n_2+1 \end{array} 
  \begin{array}{cc} & 2n_2+4 \\ 2n_2+3 & \\ 2n_2+3 & \end{array} 
  \begin{array}{cc} & 2n_2+7 \\ 2n_2+6 & \\ 2n_2+6 & \end{array} \\[10pt]
 \label{intermPtnkr_5-back-Conf4}
  & \begin{array}{cc} \\ \\ \cdots \end{array}
  \begin{array}{cc} & 2n_2+3n_3+1 \\ 2n_2+3n_3 & \\ 2n_2+3n_3 & \end{array} 
  \left. \begin{array}{cc} \\ \\ (\textrm{ parts } \geq 2n_2+3n_3+3, \textrm{ all 1-clusters }) \end{array}
  \right\}
  \begin{array}{cc} \\ \\ , \end{array} 
 \end{align}
 where $n_2$, $n_3$, or both, are possibly zero.  
 
 In \eqref{intermPtnkr_5-back-Conf4}, we simply start by setting $\mu_1 = 0$, 
 because the smallest 1-cluster is already as small as it can be.  
 It cannot be moved further back without vanishing or messing up the Gordon marking; 
 therefore changing at least one of $n_1$, $n_2$ or $n_3$.  
 
 In \eqref{intermPtnkr_5-back-Conf3}, we first subtract the necessary amount from the smallest 1-cluster, 
 and record the necessary number of moves as $\mu_1-1$.  
 The partition becomes 
 \begin{align*}
  & \left\{  \begin{array}{cc} \\ & 2 \\ 1 & \end{array} \right.
  \begin{array}{cc} \\ & 4 \\ 3 & \end{array} 
  \begin{array}{cc} \\ \\ \cdots \end{array} 
  \begin{array}{cc} \\ & 2n_2 \\ 2n_2-1 & \end{array} 
  \begin{array}{cc} & 2n_2+3 \\ 2n_2+2 & \\ 2n_2+2 & \end{array} 
  \begin{array}{cc} & 2n_2+6 \\ 2n_2+5 & \\ 2n_2+5 & \end{array} 
  \begin{array}{cc} \\ \\ \cdots \end{array}  \\[10pt]
  & \begin{array}{cc} & 2n_2+3n_3 \\ 2n_2+3n_3-1 & \\ 2n_2+3n_3-1 & \end{array} 
  \begin{array}{cc} \\ \\ \mathbf{2n_2+3n_3+2} \end{array} 
  \left. \begin{array}{cc} \\ \\ (\textrm{ parts } \geq 2n_2+3n_3+4, \textrm{ all 1-clusters }) \end{array}
  \right\}
  \begin{array}{cc} \\ \\ . \end{array} 
 \end{align*}
 We then perform one more deduction on the smallest 1-cluster, 
 followed by prestidigitating that 1-cluster through the 3-clusters, 
 hence obtaining $\mu_1$.  
 \begin{align*}
  & \left\{  \begin{array}{cc} \\ & 2 \\ 1 & \end{array} \right.
  \begin{array}{cc} \\ & 4 \\ 3 & \end{array} 
  \begin{array}{cc} \\ \\ \cdots \end{array} 
  \begin{array}{cc} \\ & 2n_2 \\ 2n_2-1 & \end{array} 
  \begin{array}{cc} & 2n_2+3 \\ 2n_2+2 & \\ 2n_2+2 & \end{array} 
  \begin{array}{cc} & 2n_2+6 \\ 2n_2+5 & \\ 2n_2+5 & \end{array} 
  \begin{array}{cc} \\ \\ \cdots \end{array}  \\[10pt]
  & \underbrace{
  \begin{array}{cc} & 2n_2+3n_3 \\ 2n_2+3n_3-1 & \\ 2n_2+3n_3-1 & \end{array} 
  \begin{array}{cc} \\ \\ \mathbf{2n_2+3n_3+1} \end{array} 
  }_{!} 
  \left. \begin{array}{cc} \\ \\ (\textrm{ parts } \geq 2n_2+3n_3+4, \textrm{ all 1-clusters }) \end{array}
  \right\}
 \end{align*}
 \begin{align*}
  \big\downarrow \textrm{ adjustment }
 \end{align*} 
 \begin{align*}
  & \left\{  \begin{array}{cc} \\ & 2 \\ 1 & \end{array} \right.
  \begin{array}{cc} \\ & 4 \\ 3 & \end{array} 
  \begin{array}{cc} \\ \\ \cdots \end{array} 
  \begin{array}{cc} \\ & 2n_2 \\ 2n_2-1 & \end{array} 
  \begin{array}{cc} & 2n_2+3 \\ 2n_2+2 & \\ 2n_2+2 & \end{array} 
  \begin{array}{cc} & 2n_2+6 \\ 2n_2+5 & \\ 2n_2+5 & \end{array} 
  \begin{array}{cc} \\ \\ \cdots \end{array}  \\[10pt]
  & \underbrace{
  \begin{array}{cc} & 2n_2+3n_3-3 \\ 2n_2+3n_3-4 & \\ 2n_2+3n_3-4 & \end{array} 
  \begin{array}{cc} \\ \\ \mathbf{2n_2+3n_3-2} \end{array} 
  }_{!} 
  \begin{array}{cc} & 2n_2+3n_3+1 \\ 2n_2+3n_3 & \\ 2n_2+3n_3 & \end{array} \\[10pt]
  & \left. \begin{array}{cc} \\ \\ (\textrm{ parts } \geq 2n_2+3n_3+4, \textrm{ all 1-clusters }) \end{array}
  \right\}
 \end{align*}
 \begin{align*}
  \big\downarrow \textrm{ after } n_3-1 \textrm{ similar adjustments }
 \end{align*} 
 \begin{align*}
  & \left\{  \begin{array}{cc} \\ & 2 \\ 1 & \end{array} \right.
  \begin{array}{cc} \\ & 4 \\ 3 & \end{array} 
  \begin{array}{cc} \\ \\ \cdots \end{array} 
  \begin{array}{cc} \\ & 2n_2 \\ 2n_2-1 & \end{array} 
  \begin{array}{cc} \\ \\ \mathbf{2n_2+1} \end{array} 
  \begin{array}{cc} & 2n_2+4 \\ 2n_2+3 & \\ 2n_2+3 & \end{array} 
  \begin{array}{cc} & 2n_2+7 \\ 2n_2+6 & \\ 2n_2+6 & \end{array} 
  \begin{array}{cc} \\ \\ \cdots \end{array}  \\[10pt]
  & \begin{array}{cc} & 2n_2+3n_3+1 \\ 2n_2+3n_3 & \\ 2n_2+3n_3 & \end{array} 
  \left. \begin{array}{cc} \\ \\ (\textrm{ parts } \geq 2n_2+3n_3+4 \textrm{ all 1-clusters }) \end{array}
  \right\}
  \begin{array}{cc} \\ \\ , \end{array} 
 \end{align*} 
 arriving at \eqref{intermPtnkr_5-back-Conf4} with $\mu_1 > 0$.  
 
 We continue with subtracting $\mu_i$ from the $i$th smallest 1-cluster for $i = 2, 3, \ldots, n_1$ 
 in the given order, to obtain the base partition $\beta$ as \eqref{baseptnKR5-n1positive}.  
 Because the pairwise difference of 1-clusters are at least two, 
 we immediately get $\mu_2 \leq \mu_3 \leq \cdots \leq \mu_{n_1}$.  
 To see that $\mu_1 \leq \mu_2$, 
 simply notice that without the final backward move involving 
 the prestidigitation of the smallest 1-cluster
 through the 3-clusters, we would have $\mu_1 - 1 \leq \mu_2 - 1 \leq \cdots \leq \mu_{n_1} - 1$.  
 If there were no 1-clusters, we would have stopped at \eqref{intermPtnkr_5-back-Conf3}, 
 which incidentally would have been the base partition $\beta$, 
 and declare $\mu$ the empty partition.  
 This yields the quadruple $(\beta, \mu, \eta, \nu)$ we have been looking for, 
 given $\lambda$ counted by $kr_5(n, m)$, and concludes the proof.  
\end{proof}

% begin example 
  {\bf Example: } 
  Following the notation in the proof of Theorem \ref{thmKR5}, 
  let's take the base partition $\beta$ having 
  $n_1 = 3$ 1-clusters, $n_2 = 2$ 2-clusters, and $n_3 = 2$ 3-clusters.  
  Assume that $\mu = 1+1+1$, $\eta = 0+5$, and $\nu = 3+9$.  
  \begin{align*}
   \beta = 
   \left\{ 
   \begin{array}{cc} & \\ & 2 \\ 1 & \end{array}
   \begin{array}{cc} & \\ & 4 \\ 3 & \end{array}
   \begin{array}{cc} \\ \\ 5 \end{array}
   \begin{array}{cc} & 8 \\ 7 & \\ 7 & \end{array}
   \begin{array}{cc} & 11 \\ 10 & \\ 10 & \end{array}
   \begin{array}{cc} \\ \\ 13 \end{array}
   \begin{array}{cc} \\ \\ \mathbf{15} \end{array}
   \right\}
  \end{align*}
  The weight of $\beta$ is 96.  
  
  We first incorporate $\mu_3$ and $\mu_2$ on the two largest 1-clusters, 
  which are simple additions.  
  \begin{align*}
   \left\{ 
   \begin{array}{cc} & \\ & 2 \\ 1 & \end{array}
   \begin{array}{cc} & \\ & 4 \\ 3 & \end{array}
   \begin{array}{cc} \\ \\ \mathbf{5} \end{array}
   \begin{array}{cc} & 8 \\ 7 & \\ 7 & \end{array}
   \begin{array}{cc} & 11 \\ 10 & \\ 10 & \end{array}
   \begin{array}{cc} \\ \\ 14 \end{array}
   \begin{array}{cc} \\ \\ 16 \end{array}
   \right\}
  \end{align*}
  We then perform the $\mu_1 = 1$ forward move on the smallest 1-cluster, 
  and watch it being prestidigitated through the 3-clusters.  
  \begin{align*}
   \left\{ 
   \begin{array}{cc} & \\ & 2 \\ 1 & \end{array} \right.
   \begin{array}{cc} & \\ & 4 \\ 3 & \end{array}
   \underbrace{
   \begin{array}{cc} \\ \\ \mathbf{6} \end{array}
   \begin{array}{cc} & 8 \\ 7 & \\ 7 & \end{array}
   }_{!}
   \begin{array}{cc} & 11 \\ 10 & \\ 10 & \end{array}
   \begin{array}{cc} \\ \\ 14 \end{array}
   \left. \begin{array}{cc} \\ \\ 16 \end{array}
   \right\}
  \end{align*}
  \begin{align*}
   \big\downarrow \textrm{ adjustment }
  \end{align*}
  \begin{align*}
   \left\{ 
   \begin{array}{cc} & \\ & 2 \\ 1 & \end{array} \right.
   \begin{array}{cc} & \\ & 4 \\ 3 & \end{array}
   \begin{array}{cc} & 7 \\ 6 & \\ 6 & \end{array}
   \underbrace{
   \begin{array}{cc} \\ \\ \mathbf{9} \end{array}
   \begin{array}{cc} & 11 \\ 10 & \\ 10 & \end{array}
   }_{!}
   \begin{array}{cc} \\ \\ 14 \end{array}
   \left. \begin{array}{cc} \\ \\ 16 \end{array}
   \right\}
  \end{align*}
  \begin{align*}
   \big\downarrow \textrm{ adjustment }
  \end{align*}
  \begin{align*}
   \left\{ 
   \begin{array}{cc} & \\ & 2 \\ 1 & \end{array} \right.
   \begin{array}{cc} & \\ & 4 \\ 3 & \end{array}
   \begin{array}{cc} & 7 \\ 6 & \\ 6 & \end{array}
   \begin{array}{cc} & 10 \\ 9 & \\ 9 & \end{array}
   \begin{array}{cc} \\ \\ \mathbf{12} \end{array}
   \begin{array}{cc} \\ \\ 14 \end{array}
   \left. \begin{array}{cc} \\ \\ 16 \end{array}
   \right\}
  \end{align*}
  This completes the incorporation of $\mu$ as forward moves on the 1-clusters.  
  Next, we turn to $\eta = 0 + 5$.  
  The larger 2-cluster will be moved 5 times forward.  
  The first of those moves will involve prestidigitation through the 3-clusters.  
  The smaller 2-cluster will stay put, thanks to $\eta_1$ being zero.  
  \begin{align*}
   \big\downarrow \textrm{ the first move forward on the larger 2-cluster }
  \end{align*}
  \begin{align*}
   \left\{ 
   \begin{array}{cc} & \\ & 2 \\ 1 & \end{array} \right.
   \underbrace{
   \begin{array}{cc} \\ \mathbf{4} \\ \mathbf{4} \end{array}
   \begin{array}{cc} & 7 \\ 6 & \\ 6 & \end{array}
   }_{!}
   \begin{array}{cc} & 10 \\ 9 & \\ 9 & \end{array}
   \begin{array}{cc} \\ \\ 12 \end{array}
   \begin{array}{cc} \\ \\ 14 \end{array}
   \left. \begin{array}{cc} \\ \\ 16 \end{array}
   \right\}
  \end{align*}
  \begin{align*}
   \big\downarrow \textrm{ adjustment }
  \end{align*}
  \begin{align*}
   \left\{ 
   \begin{array}{cc} & \\ & 2 \\ 1 & \end{array} \right.
   \begin{array}{cc} & 5 \\ 4 & \\ 4 & \end{array}
   \underbrace{
   \begin{array}{cc} \\ \mathbf{7} \\ \mathbf{7} \end{array}
   \begin{array}{cc} & 10 \\ 9 & \\ 9 & \end{array}
   }_{!}
   \begin{array}{cc} \\ \\ 12 \end{array}
   \begin{array}{cc} \\ \\ 14 \end{array}
   \left. \begin{array}{cc} \\ \\ 16 \end{array}
   \right\}
  \end{align*}
  \begin{align*}
   \big\downarrow \textrm{ adjustment }
  \end{align*}
  \begin{align*}
   \left\{ 
   \begin{array}{cc} & \\ & 2 \\ 1 & \end{array} \right.
   \begin{array}{cc} & 5 \\ 4 & \\ 4 & \end{array}
   \begin{array}{cc} & 8 \\ 7 & \\ 7 & \end{array}
   \begin{array}{cc} \\ \mathbf{10} \\ \mathbf{10} \end{array}
   \begin{array}{cc} \\ \\ 12 \end{array}
   \begin{array}{cc} \\ \\ 14 \end{array}
   \left. \begin{array}{cc} \\ \\ 16 \end{array}
   \right\}
  \end{align*}
  \begin{align*}
   \big\downarrow \textrm{ four more moves on the larger 2-cluster }
  \end{align*}
  \begin{align*}
   \left\{ 
   \begin{array}{cc} & \\ & 2 \\ 1 & \end{array} \right.
   \begin{array}{cc} & 5 \\ 4 & \\ 4 & \end{array}
   \begin{array}{cc} & \mathbf{8} \\ \mathbf{7} & \\ \mathbf{7} & \end{array}
   \begin{array}{cc} \\ \\ 10 \end{array}
   \begin{array}{cc} \\ \\ 12 \end{array}
   \begin{array}{cc} \\ 14 \\ 14 \end{array}
   \left. \begin{array}{cc} \\ \\ 16 \end{array}
   \right\}
  \end{align*}
  Finally, we use $\nu = 3 + 9$ 
  to move the larger 3-cluster $\frac{1}{3}\nu_2 = 3$ times forward, 
  and then the smaller 3-cluster $\frac{1}{3}\nu_1 = 1$ times forward.  
  \begin{align*}
   \big\downarrow \textrm{ the first forward move on the larger 3-cluster }
  \end{align*}
  \begin{align*}
   \left\{ 
   \begin{array}{cc} & \\ & 2 \\ 1 & \end{array} \right.
   \begin{array}{cc} & 5 \\ 4 & \\ 4 & \end{array}
   \underbrace{
   \begin{array}{cc} & \mathbf{9} \\ \mathbf{8} & \\ \mathbf{8} & \end{array}
   \begin{array}{cc} \\ \\ 10 \end{array}
   }_{!}
   \begin{array}{cc} \\ \\ 12 \end{array}
   \begin{array}{cc} \\ 14 \\ 14 \end{array}
   \left. \begin{array}{cc} \\ \\ 16 \end{array}
   \right\}
  \end{align*}
  \begin{align*}
   \big\downarrow \textrm{ adjustment }
  \end{align*}
  \begin{align*}
   \left\{ 
   \begin{array}{cc} & \\ & 2 \\ 1 & \end{array} \right.
   \begin{array}{cc} & 5 \\ 4 & \\ 4 & \end{array}
   \begin{array}{cc} \\ \\ 7 \end{array}
   \begin{array}{cc} & \mathbf{10} \\ \mathbf{9} & \\ \mathbf{9} & \end{array}
   \begin{array}{cc} \\ \\ 12 \end{array}
   \begin{array}{cc} \\ 14 \\ 14 \end{array}
   \left. \begin{array}{cc} \\ \\ 16 \end{array}
   \right\}
  \end{align*}
  \begin{align*}
   \big\downarrow \textrm{ the second forward move on the larger 3-cluster }
  \end{align*}
  \begin{align*}
   \left\{ 
   \begin{array}{cc} & \\ & 2 \\ 1 & \end{array} \right.
   \begin{array}{cc} & 5 \\ 4 & \\ 4 & \end{array}
   \begin{array}{cc} \\ \\ 7 \end{array}
   \underbrace{
   \begin{array}{cc} & \mathbf{11} \\ \mathbf{10} & \\ \mathbf{10} & \end{array}
   \begin{array}{cc} \\ \\ 12 \end{array}
   }_{!}
   \begin{array}{cc} \\ 14 \\ 14 \end{array}
   \left. \begin{array}{cc} \\ \\ 16 \end{array}
   \right\}
  \end{align*}
  \begin{align*}
   \big\downarrow \textrm{ adjustment }
  \end{align*}
  \begin{align*}
   \left\{ 
   \begin{array}{cc} & \\ & 2 \\ 1 & \end{array} \right.
   \begin{array}{cc} & 5 \\ 4 & \\ 4 & \end{array}
   \begin{array}{cc} \\ \\ 7 \end{array}
   \begin{array}{cc} \\ \\ 9 \end{array}
   \begin{array}{cc} & \mathbf{12} \\ \mathbf{11} & \\ \mathbf{11} & \end{array}
   \begin{array}{cc} \\ 14 \\ 14 \end{array}
   \left. \begin{array}{cc} \\ \\ 16 \end{array}
   \right\}
  \end{align*}
  \begin{align*}
   \big\downarrow \textrm{ the third, and the last, forward move on the larger 3-cluster }
  \end{align*}
  \begin{align*}
   \left\{ 
   \begin{array}{cc} & \\ & 2 \\ 1 & \end{array} \right.
   \begin{array}{cc} & 5 \\ 4 & \\ 4 & \end{array}
   \begin{array}{cc} \\ \\ 7 \end{array}
   \begin{array}{cc} \\ \\ 9 \end{array}
   \underbrace{
   \begin{array}{cc} & \mathbf{13} \\ \mathbf{12} & \\ \mathbf{12} & \end{array}
   \begin{array}{cc} \\ 14 \\ 14 \end{array}
   }_{!}
   \left. \begin{array}{cc} \\ \\ 16 \end{array}
   \right\}
  \end{align*}
  \begin{align*}
   \big\downarrow \textrm{ adjustment }
  \end{align*}
  \begin{align*}
   \left\{ 
   \begin{array}{cc} & \\ & 2 \\ 1 & \end{array} \right.
   \begin{array}{cc} & 5 \\ 4 & \\ 4 & \end{array}
   \begin{array}{cc} \\ \\ 7 \end{array}
   \begin{array}{cc} \\ \\ 9 \end{array}
   \begin{array}{cc} \\ 11 \\ 11 \end{array}
   \underbrace{
   \begin{array}{cc} & \mathbf{15} \\ \mathbf{14} & \\ \mathbf{14} & \end{array}
   \begin{array}{cc} \\ \\ 16 \end{array}
   }_{!}
   \left. \begin{array}{cc} \\ \\ \end{array}
   \right\}
  \end{align*}
  \begin{align*}
   \big\downarrow \textrm{ adjustment }
  \end{align*}
  \begin{align*}
   \left\{ 
   \begin{array}{cc} & \\ & 2 \\ 1 & \end{array} \right.
   \begin{array}{cc} & \mathbf{5} \\ \mathbf{4} & \\ \mathbf{4} & \end{array}
   \begin{array}{cc} \\ \\ 7 \end{array}
   \begin{array}{cc} \\ \\ 9 \end{array}
   \begin{array}{cc} \\ 11 \\ 11 \end{array}
   \begin{array}{cc} \\ \\ 13 \end{array}
   \left. \begin{array}{cc} & 16 \\ 15 & \\ 15 & \end{array}
   \right\}
  \end{align*}
  \begin{align*}
   \big\downarrow \textrm{ one forward move on the smaller 3-cluster }
  \end{align*}
  \begin{align*}
   \left\{ 
   \begin{array}{cc} & \\ & 2 \\ 1 & \end{array} \right.
   \underbrace{
   \begin{array}{cc} & \mathbf{6} \\ \mathbf{5} & \\ \mathbf{5} & \end{array}
   \begin{array}{cc} \\ \\ 7 \end{array}
   }_{!}
   \begin{array}{cc} \\ \\ 9 \end{array}
   \begin{array}{cc} \\ 11 \\ 11 \end{array}
   \begin{array}{cc} \\ \\ 13 \end{array}
   \left. \begin{array}{cc} & 16 \\ 15 & \\ 15 & \end{array}
   \right\}
  \end{align*}
  \begin{align*}
   \big\downarrow \textrm{ adjustment }
  \end{align*}
  \begin{align*}
    \lambda = 
   \left\{ 
   \begin{array}{cc} & \\ & 2 \\ 1 & \end{array} \right.
   \begin{array}{cc} \\ \\ 4 \end{array}
   \begin{array}{cc} & \mathbf{7} \\ \mathbf{6} & \\ \mathbf{6} & \end{array}
   \begin{array}{cc} \\ \\ 9 \end{array}
   \begin{array}{cc} \\ 11 \\ 11 \end{array}
   \begin{array}{cc} \\ \\ 13 \end{array}
   \left. \begin{array}{cc} & 16 \\ 15 & \\ 15 & \end{array}
   \right\}
  \end{align*}
  The weight of $\lambda$, as expected is 116.  
  $\lambda$ has the sediment $\begin{array}{cc} & 2 \\ 1 & \end{array}$, 
  for the sole unmoved 2-cluster.  
  \begin{align*}
   \vert \lambda \vert = 116 = 96 + 3 + 5 + 12 
   = \vert \beta \vert + \vert \mu \vert + \vert \eta \vert + \vert \nu \vert
  \end{align*}

% end example

\begin{theorem}[cf. Kanade-Russell conjecture $I_6$]
\label{thmKR6}
  For $n, m \in \mathbb{N}$, let $kr_6(n, m)$ be the number of partitions of $n$ into $m$ parts
  with smallest part at least 2, at most one appearance of the part 2, 
  and difference at least three at distance three such that 
  if parts at distance two differ by at most one, 
  then their sum, together with the intermediate part, 
  is $\equiv 2 \pmod{3}$.  
  Then, 
  \begin{align}
  \nonumber
  & \sum_{m,n \geq 0} kr_6(n,m) q^n x^m \\
  \label{eqGF_KR6}
  = & \sum_{n_1, n_2, n_3 \geq 0} \frac{ 
      q^{ ( 9n_3^2 + 7n_3 )/2 + 2n_2^2 + 3n_2 + n_1^2 + n_1 + 6n_3n_2 + 3n_3n_1 + 2n_2n_1 } 
      (-q; q^2)_{n_2} x^{3n_3 + 2n_2 + n_1} }
    { (q; q)_{n_1} (q^2; q^2)_{n_2} (q^3; q^3)_{n_3} }.  
  \end{align}
\end{theorem}

\begin{proof}
 The proof is a simpler version of the proof of Theorem \ref{thmKR5}.  
 There is only one type of base partition $\beta$.  
 \begin{align*}
  & \left\{ \begin{array}{cc} & 3 \\ & 3 \\ 2 & \end{array} \right.
  \begin{array}{cc} & 6 \\ & 6 \\ 5 & \end{array}
  \begin{array}{cc} \\ \\ \cdots \end{array}
  \begin{array}{cc} & 3n_3 \\ & 3n_3 \\ 3n_3-1 & \end{array}
  \begin{array}{cc} \\ & 3n_3 + 3 \\ 3n_3 + 2 &  \end{array}
  \begin{array}{cc} \\ & 3n_3 + 5 \\ 3n_3 + 4 &  \end{array}
  \begin{array}{cc} \\ \\ \cdots \end{array} \\[10pt]
  & \begin{array}{cc} \\ & 3n_3 + 2n_2 + 1 \\ 3n_3 + 2n_2 &  \end{array}
  \begin{array}{cc} \\ \\ 3n_3 + 2n_2 + 2 \end{array}
  \begin{array}{cc} \\ \\ 3n_3 + 2n_2 + 4 \end{array}
  \begin{array}{cc} \\ \\ \cdots \end{array}
  \left. \begin{array}{cc} \\ \\ 3n_3 + 2n_2 + 2n_1 \end{array} \right\}
 \end{align*} 
 This partition has the minimum weight among all numerated by $kr_6(n, m)$ ,
 having $n_r$ $r$-clusters for $r = 1, 2, 3$.  
 Here, any $n_r$ may be zero.  
 Clearly, the only possible 3-clusters are 
 \begin{align*}
  \left\{ \begin{array}{cc} \\ \\ ( \textrm{ parts } \leq k - 2 ) \end{array}
  \begin{array}{cc} & k+1 \\ & k+1 \\ k & \end{array}
  \begin{array}{cc} \\ \\ ( \textrm{ parts } \geq k + 3 ) \end{array} \right\}
  \begin{array}{cc} \\ \\ . \end{array}
 \end{align*} 
 The rest of the proof is the  same as that of Theorem \ref{thmKR5}.  
 One does not even need to prestidigitate the 1- or 2- clusters through the 3-clusters.  
\end{proof}

We now write generating functions for some similarly described enumerants, 
which are not listed in \cite{KR-conj} 
because they did not yield nice infinite products, hence partition identities.  
In their proofs, we indicate the extra details only.  

\begin{theorem}
\label{thmKRc1-2}
  For $n, m \in \mathbb{N}$, let $kr^c_{1-2}(n, m)$ be the number of partitions of $n$ into $m$ parts
  with difference at least three at distance three such that 
  if parts at distance two differ by at most one, 
  then their sum, together with the intermediate part, 
  is $\equiv 1 \pmod{3}$.  
  Then, 
  \begin{align}
  \nonumber
  & \sum_{m,n \geq 0} kr^c_{1-2}(n,m) q^n x^m \\
  \label{eqGF-KRc1-2}
  = & \sum_{\substack{n_1, n_3 \geq 0 \\ n_2 > 0}} \frac{ 
      q^{ ( 9n_3^2 - n_3 )/2 + 2n_2^2 + n_2 + n_1^2 + 6n_3n_2 + 3n_3n_1 + 2n_2n_1 - 1 } 
      (1 + q)(-q; q^2)_{n_2 - 1} \; x^{3n_3 + 2n_2 + n_1} }
    { (q; q)_{n_1} (q^2; q^2)_{n_2} (q^3; q^3)_{n_3} } \\
  \nonumber
  + & \sum_{n_1, n_3 \geq 0} \frac{ q^{ ( 9n_3^2 - n_3 )/2 + n_1^2 + 3n_3n_1 } x^{3n_3 + n_1} }
    { (q; q)_{n_1} (q^3; q^3)_{n_3} } \\
  \label{eqGF-KRc1-2-2nd}
  = & \sum_{n_1, n_2, n_3 \geq 0} \frac{ 
      q^{ ( 9n_3^2 - n_3 )/2 + 2n_2^2 + n_2 + n_1^2 + 6n_3n_2 + 3n_3n_1 + 2n_2n_1} 
      (-1/q; q^2)_{n_2} \; x^{3n_3 + 2n_2 + n_1} }
    { (q; q)_{n_1} (q^2; q^2)_{n_2} (q^3; q^3)_{n_3} }.  
  \end{align}
\end{theorem}

{\bf Remark: } Notice that no $\lambda$ enumerated by $kr^c_{1-2}(n, m)$ can have three occurrences of 1.  

\begin{proof}
 We will show \eqref{eqGF-KRc1-2} only.  
 \eqref{eqGF-KRc1-2-2nd} follows by standard algebraic manipulations.  

 The proof is similar to the proof of Theorem \ref{thmKR5}.  
 Two separate series are for two separate base partitions 
 for the cases $n_1, n_3 \geq 0, n_2 > 0$, and $n_1, n_3 \geq 0, n_2 = 0$.  
 Here, again, $n_r$ is the number of $r$-clusters for $r = 1, 2, 3$ 
 of the partition at hand.  
 
 In case $n_2 > 0$, the base partition $\beta$ is 
 \begin{align}
 \nonumber
  & \left\{ \begin{array}{cc} & 2 \\ 1 & \\ 1 & \end{array} \right.
  \begin{array}{cc} & 5 \\ 4 & \\ 4 & \end{array}
  \begin{array}{cc} \\ \\ \cdots \end{array}
  \begin{array}{cc} & 3n_3-1 \\ 3n_3-2 & \\ 3n_3-2 & \end{array}
  \begin{array}{cc} \\ 3n_3 + 1 \\ 3n_3 + 1  \end{array}
  \begin{array}{cc} \\ & 3n_3 + 4 \\ 3n_3 + 3 &  \end{array}
  \begin{array}{cc} \\ & 3n_3 + 6 \\ 3n_3 + 5 &  \end{array}
  \begin{array}{cc} \\ \\ \cdots \end{array} \\[10pt]
 \label{baseptnKRc1-2-1}
  & \begin{array}{cc} \\ & 3n_3 + 2n_2 \\ 3n_3 + 2n_2-1 &  \end{array}
  \begin{array}{cc} \\ \\ 3n_3 + 2n_2 + 1 \end{array}
  \begin{array}{cc} \\ \\ 3n_3 + 2n_2 + 3 \end{array}
  \begin{array}{cc} \\ \\ \cdots \end{array}
  \left. \begin{array}{cc} \\ \\ 3n_3 + 2n_2 + 2n_1-1 \end{array} \right\}
  \begin{array}{cc} \\ \\ , \end{array}
 \end{align} 
 with weight $( 9n_3^2 - n_3 )/2 + 2n_2^2 + n_2 + n_1^2 + 6n_3n_2 + 3n_3n_1 + 2n_2n_1 - 1 $.  
 
 When $n_2 = 0$, the base partition is 
 \begin{align}
 \label{baseptnKRc1-2-2}
  & \left\{ \begin{array}{cc} & 2 \\ 1 & \\ 1 & \end{array} \right.
  \begin{array}{cc} & 5 \\ 4 & \\ 4 & \end{array}
  \begin{array}{cc} \\ \\ \cdots \end{array}
  \begin{array}{cc} & 3n_3-1 \\ 3n_3-2 & \\ 3n_3-2 & \end{array}
  \begin{array}{cc} \\ \\ 3n_3  + 1 \end{array}
  \begin{array}{cc} \\ \\ 3n_3  + 3 \end{array}
  \begin{array}{cc} \\ \\ \cdots \end{array}
  \left. \begin{array}{cc} \\ \\ 3n_3 + 2n_1-1 \end{array} \right\}
  \begin{array}{cc} \\ \\ , \end{array}
 \end{align} 
 with weight $( 9n_3^2 - n_3 )/2 + n_1^2 + 3n_3n_1$.  
 This is not the $n_2 = 0$ case of \eqref{baseptnKRc1-2-1}.  
 
 The novelty in \eqref{baseptnKRc1-2-1} is that the smallest 2-cluster 
 $\begin{array}{cc} 3n_3 + 1 \\ 3n_3 + 1  \end{array}$ 
 has an extra move forward.  
 If that extra move is made, then the 2-clusters in the resulting partition 
 can be treated as in the proof of Theorem \ref{thmKR5}.  
 Without this extra move, we only have $n_2 - 1$ 2-clusters to move forward.  
 
 In a partition $\lambda$ enumerated by $kr^c_{1-2}(n, m)$, 
 we check if there is a sediment of the form
 \begin{align*}
  & \left\{ \begin{array}{cc} & 2 \\ 1 & \\ 1 & \end{array} \right.
  \begin{array}{cc} & 5 \\ 4 & \\ 4 & \end{array}
  \begin{array}{cc} \\ \\ \cdots \end{array}
  \begin{array}{cc} & 3s-1 \\ 3s-2 & \\ 3s-2 & \end{array}
  \begin{array}{cc} \\ 3s + 1\\ 3s  + 1 \end{array}
  \left. \begin{array}{cc} \\ \\ ( \textrm{ parts } \geq 3s + 3) \end{array} \right\}
 \end{align*} 
 for $s \geq 0$ to tell the cases apart.  
 
 The partition accounting for the forward or backward moves on the 2-clusters is generated by 
 \begin{align*}
  \frac{ (-q; q^2)_{n_2-1} }{ (q^2; q^2)_{n_2-1} } 
  + q \frac{ (-q; q^2)_{n_2} }{ (q^2; q^2)_{n_2} } 
  = \frac{ (1 + q) (-q; q^2)_{n_2-1} }{ (q^2; q^2)_{n_2} } 
 \end{align*}
 for $n_2 \geq 1$.  
 The factor $q$ in the second term is for the extra move.  
 For $n_2 = 0$, it is simply 1, the empty partition.  
 
 The rest of the proof is the same as the proof of Theorem \ref{thmKR5}, 
 except that prestidigitating 1- or 2-clusters through the 3-clusters is not necessary.  
\end{proof}

% begin example 
  
  {\bf Example: } Following the notation of the proof of the above theorem, let
  \begin{align*}
    \lambda = 
   \left\{ 
   \begin{array}{cc} & \\ & 2 \\ 1 & \end{array} \right.
   \begin{array}{cc} \\ \\ 4 \end{array}
   \begin{array}{cc} & \mathbf{7} \\ \mathbf{6} & \\ \mathbf{6} & \end{array}
   \begin{array}{cc} \\ \\ 9 \end{array}
   \begin{array}{cc} \\ 11 \\ 11 \end{array}
   \begin{array}{cc} \\ \\ 13 \end{array}
   \left. \begin{array}{cc} & 16 \\ 15 & \\ 15 & \end{array}
   \right\}
   \begin{array}{cc} \\ \\ . \end{array}
  \end{align*}
  This is one of the partitions we encountered before.  
  We will examine it once more as a partition satisfying the conditions of $kr^c_{1-2}(116, 13)$.  
  $\lambda$ as such has no sediments, 
  therefore the initial forward move was applied to the smallest 2-cluster, 
  and $\eta$ has two parts.  
  
  We begin by decoding $\nu$ through the backward moves on the 3-clusters, 
  the smallest first.  
  \begin{align*}
   \big\downarrow \textrm{ one backward move on the smallest 3-cluster }
  \end{align*}
  \begin{align*}
   \left\{ 
   \begin{array}{cc} & \\ & 2 \\ 1 & \end{array} \right.
   \underbrace{
   \begin{array}{cc} \\ \\ 4 \end{array}
   \begin{array}{cc} & \mathbf{6} \\ \mathbf{5} & \\ \mathbf{5} & \end{array}
   }_{!}
   \begin{array}{cc} \\ \\ 9 \end{array}
   \begin{array}{cc} \\ 11 \\ 11 \end{array}
   \begin{array}{cc} \\ \\ 13 \end{array}
   \left. \begin{array}{cc} & 16 \\ 15 & \\ 15 & \end{array}
   \right\}
  \end{align*}
  \begin{align*}
   \big\downarrow \textrm{ adjustment }
  \end{align*}
  \begin{align*}
   \left\{ 
   \begin{array}{cc} & \\ & 2 \\ 1 & \end{array} \right.
   \begin{array}{cc} & \mathbf{5} \\ \mathbf{4} & \\ \mathbf{4} & \end{array}
   \begin{array}{cc} \\ \\ 7 \end{array}
   \begin{array}{cc} \\ \\ 9 \end{array}
   \begin{array}{cc} \\ 11 \\ 11 \end{array}
   \begin{array}{cc} \\ \\ 13 \end{array}
   \left. \begin{array}{cc} & 16 \\ 15 & \\ 15 & \end{array}
   \right\}
  \end{align*}
  \begin{align*}
   \big\downarrow \textrm{ one more backward move on the smallest 3-cluster }
  \end{align*}
  \begin{align*}
   \left\{ \begin{array}{cc} \\ \\ \end{array} \right.
   \underbrace{
   \begin{array}{cc} & \\ & 2 \\ 1 & \end{array} 
   \begin{array}{cc} & \mathbf{4} \\ \mathbf{3} & \\ \mathbf{3} & \end{array}
   }_{!}
   \begin{array}{cc} \\ \\ 7 \end{array}
   \begin{array}{cc} \\ \\ 9 \end{array}
   \begin{array}{cc} \\ 11 \\ 11 \end{array}
   \begin{array}{cc} \\ \\ 13 \end{array}
   \left. \begin{array}{cc} & 16 \\ 15 & \\ 15 & \end{array}
   \right\}
  \end{align*}
  \begin{align*}
   \big\downarrow \textrm{ adjustment }
  \end{align*}
  \begin{align*}
   \left\{ \begin{array}{cc} \\ \\ \end{array} \right.
   \begin{array}{cc} & 2 \\ 1 & \\ 1 & \end{array}
   \begin{array}{cc} & \\ & 5 \\ 4 & \end{array} 
   \begin{array}{cc} \\ \\ 7 \end{array}
   \begin{array}{cc} \\ \\ 9 \end{array}
   \begin{array}{cc} \\ 11 \\ 11 \end{array}
   \begin{array}{cc} \\ \\ 13 \end{array}
   \left. \begin{array}{cc} & \mathbf{16} \\ \mathbf{15} & \\ \mathbf{15} & \end{array}
   \right\}
  \end{align*}
  The smallest 3-cluster has been stowed after two backward moves on it, 
  thus, $\nu_1 = 3 \cdot 2 = 6$.  
  \begin{align*}
   \big\downarrow \textrm{ one backward move on the larger 3-cluster }
  \end{align*}
  \begin{align*}
   \left\{ 
   \begin{array}{cc} & 2 \\ 1 & \\ 1 & \end{array} \right. 
   \begin{array}{cc} & \\ & 5 \\ 4 & \end{array} 
   \begin{array}{cc} \\ \\ 7 \end{array}
   \begin{array}{cc} \\ \\ 9 \end{array}
   \begin{array}{cc} \\ 11 \\ 11 \end{array}
   \underbrace{
   \begin{array}{cc} \\ \\ 13 \end{array}
   \begin{array}{cc} & \mathbf{15} \\ \mathbf{14} & \\ \mathbf{14} & \end{array}
   }_{!}
   \left. \begin{array}{cc} \\ \\ \end{array} \right\}
  \end{align*}
  \begin{align*}
   \big\downarrow \textrm{ adjustment }
  \end{align*}
  \begin{align*}
   \left\{ 
   \begin{array}{cc} & 2 \\ 1 & \\ 1 & \end{array} \right. 
   \begin{array}{cc} & \\ & 5 \\ 4 & \end{array} 
   \begin{array}{cc} \\ \\ 7 \end{array}
   \begin{array}{cc} \\ \\ 9 \end{array}
   \underbrace{
   \begin{array}{cc} \\ 11 \\ 11 \end{array}
   \begin{array}{cc} & \mathbf{14} \\ \mathbf{13} & \\ \mathbf{13} & \end{array}
   }_{!}
   \left. \begin{array}{cc} \\ \\ 16 \end{array}
   \right\}
  \end{align*}
  \begin{align*}
   \big\downarrow \textrm{ adjustment }
  \end{align*}
  \begin{align*}
   \left\{ 
   \begin{array}{cc} & 2 \\ 1 & \\ 1 & \end{array} \right. 
   \begin{array}{cc} & \\ & 5 \\ 4 & \end{array} 
   \begin{array}{cc} \\ \\ 7 \end{array}
   \begin{array}{cc} \\ \\ 9 \end{array}
   \begin{array}{cc} & \mathbf{12} \\ \mathbf{11} & \\ \mathbf{11} & \end{array}
   \begin{array}{cc} \\ 14 \\ 14 \end{array}
   \left. \begin{array}{cc} \\ \\ 16 \end{array}
   \right\}
  \end{align*}
  \begin{align*}
   \big\downarrow \textrm{ one more backward move on the larger 3-cluster }
  \end{align*}
  \begin{align*}
   \left\{ 
   \begin{array}{cc} & 2 \\ 1 & \\ 1 & \end{array} \right. 
   \begin{array}{cc} & \\ & 5 \\ 4 & \end{array} 
   \begin{array}{cc} \\ \\ 7 \end{array}
   \underbrace{
   \begin{array}{cc} \\ \\ 9 \end{array}
   \begin{array}{cc} & \mathbf{11} \\ \mathbf{10} & \\ \mathbf{10} & \end{array}
   }_{!}
   \begin{array}{cc} \\ 14 \\ 14 \end{array}
   \left. \begin{array}{cc} \\ \\ 16 \end{array}
   \right\}
  \end{align*}
  \begin{align*}
   \big\downarrow \textrm{ adjustment }
  \end{align*}
  \begin{align*}
   \left\{ 
   \begin{array}{cc} & 2 \\ 1 & \\ 1 & \end{array} \right. 
   \begin{array}{cc} & \\ & 5 \\ 4 & \end{array} 
   \begin{array}{cc} \\ \\ 7 \end{array}
   \begin{array}{cc} & \mathbf{10} \\ \mathbf{9} & \\ \mathbf{9} & \end{array}
   \begin{array}{cc} \\ \\ 12 \end{array}
   \begin{array}{cc} \\ 14 \\ 14 \end{array}
   \left. \begin{array}{cc} \\ \\ 16 \end{array}
   \right\}
  \end{align*}
  \begin{align*}
   \big\downarrow \textrm{ one more backward move on the larger 3-cluster }
  \end{align*}
  \begin{align*}
   \left\{ 
   \begin{array}{cc} & 2 \\ 1 & \\ 1 & \end{array} \right. 
   \begin{array}{cc} & \\ & 5 \\ 4 & \end{array} 
   \underbrace{
   \begin{array}{cc} \\ \\ 7 \end{array}
   \begin{array}{cc} & \mathbf{9} \\ \mathbf{8} & \\ \mathbf{8} & \end{array}
   }_{!}
   \begin{array}{cc} \\ \\ 12 \end{array}
   \begin{array}{cc} \\ 14 \\ 14 \end{array}
   \left. \begin{array}{cc} \\ \\ 16 \end{array}
   \right\}
  \end{align*}
  \begin{align*}
   \big\downarrow \textrm{ adjustment }
  \end{align*}
  \begin{align*}
   \left\{ 
   \begin{array}{cc} & 2 \\ 1 & \\ 1 & \end{array} \right. 
   \begin{array}{cc} & \\ & 5 \\ 4 & \end{array} 
   \begin{array}{cc} & \mathbf{8} \\ \mathbf{7} & \\ \mathbf{7} & \end{array}
   \begin{array}{cc} \\ \\ 10 \end{array}
   \begin{array}{cc} \\ \\ 12 \end{array}
   \begin{array}{cc} \\ 14 \\ 14 \end{array}
   \left. \begin{array}{cc} \\ \\ 16 \end{array}
   \right\}
  \end{align*}
  \begin{align*}
   \big\downarrow \textrm{ one more backward move on the larger 3-cluster }
  \end{align*}
  \begin{align*}
   \left\{ 
   \begin{array}{cc} & 2 \\ 1 & \\ 1 & \end{array} \right. 
   \underbrace{
   \begin{array}{cc} & \\ & 5 \\ 4 & \end{array} 
   \begin{array}{cc} & \mathbf{7} \\ \mathbf{6} & \\ \mathbf{6} & \end{array}
   }_{!}
   \begin{array}{cc} \\ \\ 10 \end{array}
   \begin{array}{cc} \\ \\ 12 \end{array}
   \begin{array}{cc} \\ 14 \\ 14 \end{array}
   \left. \begin{array}{cc} \\ \\ 16 \end{array}
   \right\} 
  \end{align*}
  \begin{align*}
   \big\downarrow \textrm{ adjustment }
  \end{align*}
  \begin{align*}
   \left\{ 
   \begin{array}{cc} & 2 \\ 1 & \\ 1 & \end{array} \right. 
   \begin{array}{cc} & 5 \\ 4 & \\ 4 & \end{array}
   \begin{array}{cc} & \\ & \mathbf{8} \\ \mathbf{7} & \end{array} 
   \begin{array}{cc} \\ \\ 10 \end{array}
   \begin{array}{cc} \\ \\ 12 \end{array}
   \begin{array}{cc} \\ 14 \\ 14 \end{array}
   \left. \begin{array}{cc} \\ \\ 16 \end{array}
   \right\}
  \end{align*}
  At this point, we deduce that $\nu_2 = 3 \cdot 4 = 12$.  
  Also, looking at the smallest 2-cluster, $\eta_1 = 0$ can be seen.  
  Because with one more backward move on the smallest 2-cluster, 
  the intermediate partition becomes 
  \begin{align*}
   \left\{ 
   \begin{array}{cc} & 2 \\ 1 & \\ 1 & \end{array} \right. 
   \begin{array}{cc} & 5 \\ 4 & \\ 4 & \end{array}
   \begin{array}{cc} \\ 7 \\ 7 \end{array} 
   \begin{array}{cc} \\ \\ 10 \end{array}
   \begin{array}{cc} \\ \\ 12 \end{array}
   \begin{array}{cc} \\ \mathbf{14} \\ \mathbf{14} \end{array}
   \left. \begin{array}{cc} \\ \\ 16 \end{array}
   \right\}
   \begin{array}{cc} \\ \\ . \end{array}
  \end{align*}
  This must be the extra move.  
  \begin{align*}
   \big\downarrow \textrm{ five backward moves on the larger 2-cluster }
  \end{align*}
  \begin{align*}
   \left\{ 
   \begin{array}{cc} & 2 \\ 1 & \\ 1 & \end{array} \right. 
   \begin{array}{cc} & 5 \\ 4 & \\ 4 & \end{array}
   \begin{array}{cc} \\ 7 \\ 7 \end{array} 
   \begin{array}{cc} \\ & 10 \\ 9 & \end{array}
   \begin{array}{cc} \\ \\ \mathbf{12} \end{array}
   \begin{array}{cc} \\ \\ 14 \end{array}
   \left. \begin{array}{cc} \\ \\ 16 \end{array}
   \right\}
  \end{align*}
  This yields $\eta_2 = 5$.  
  Finally, it is clear that $\mu = 1+1+1$, 
  so that the partition becomes \eqref{baseptnKRc1-2-1}.  
  \begin{align*}
   \left\{ 
   \begin{array}{cc} & 2 \\ 1 & \\ 1 & \end{array} \right. 
   \begin{array}{cc} & 5 \\ 4 & \\ 4 & \end{array}
   \begin{array}{cc} \\ 7 \\ 7 \end{array} 
   \begin{array}{cc} \\ & 10 \\ 9 & \end{array}
   \begin{array}{cc} \\ \\ 11 \end{array}
   \begin{array}{cc} \\ \\ 13 \end{array}
   \left. \begin{array}{cc} \\ \\ 15 \end{array}
   \right\}
  \end{align*}
  In other words, the base partition for $n_2 > 0$.  
  The weight of $\lambda$ is indeed
  \begin{align*}
   \vert \lambda \vert = 116 = 89 + 3 + (1 + 5) + 18 
   = \vert \beta \vert + \vert \mu \vert + (\textrm{ extra move } + \vert \eta \vert) + \vert \nu \vert.  
  \end{align*}
  
% end example 

\begin{theorem}
\label{thmKRc2-2}
  For $n, m \in \mathbb{N}$, let $kr^c_{2-2}(n, m)$ be the number of partitions of $n$ into $m$ parts
  with difference at least three at distance three such that 
  if parts at distance two differ by at most one, 
  then their sum, together with the intermediate part, 
  is $\equiv 2 \pmod{3}$.  
  Then, 
  \begin{align}
  \nonumber
  & \sum_{m,n \geq 0} kr^c_{2-2}(n,m) q^n x^m \\
  \label{eqGF-KRc2-2}
  = & \sum_{\substack{n_1, n_3 \geq 0 \\ n_2 > 0}} \frac{ 
      q^{ ( 9n_3^2 + n_3 )/2 + 2n_2^2 + n_2 + n_1^2 + 6n_3n_2 + 3n_3n_1 + 2n_2n_1 - 1 } 
      (1 + q)(-q; q^2)_{n_2 - 1} \; x^{3n_3 + 2n_2 + n_1} }
    { (q; q)_{n_1} (q^2; q^2)_{n_2} (q^3; q^3)_{n_3} } \\
  \nonumber
  + & \sum_{n_1, n_3 \geq 0} \frac{ q^{ ( 9n_3^2 + n_3 )/2 + n_1^2 + 3n_3n_1 } x^{3n_3 + n_1} }
    { (q; q)_{n_1} (q^3; q^3)_{n_3} } \\ 
  \label{eqGF-KRc2-2-2nd}
  = & \sum_{n_1, n_2, n_3 \geq 0 } \frac{ 
      q^{ ( 9n_3^2 + n_3 )/2 + 2n_2^2 + n_2 + n_1^2 + 6n_3n_2 + 3n_3n_1 + 2n_2n_1 } 
      (-1/q; q^2)_{n_2} \; x^{3n_3 + 2n_2 + n_1} }
    { (q; q)_{n_1} (q^2; q^2)_{n_2} (q^3; q^3)_{n_3} }.  
  \end{align}
\end{theorem}

{\bf Remark: } A partition enumerated by $kr^c_{2-2}(n, m)$ may contain the 2-cluster 
$\begin{array}{cc} 1 \\ 1 \end{array}$, but not the 3-clusters 
$\begin{array}{cc} & 2\\ 1 & \\ 1 & \end{array}$ or $\begin{array}{cc} 1 \\ 1 \\ 1 \end{array}$, 
so it can have up to two occurrences of 1.  

\begin{proof}
 \eqref{eqGF-KRc2-2-2nd} follows from \eqref{eqGF-KRc2-2} 
 by standard algebraic manipulations, 
 so we demonstrate \eqref{eqGF-KRc2-2} only.  

 The proof is very similar to the proof of Theorem \ref{thmKRc1-2}.  
 The two base partitions are the following.  
 \begin{align}
 \nonumber
  & \left\{ \begin{array}{cc} \\ 1 \\ 1 \end{array} \right.
  \begin{array}{cc} & 4 \\ & 4 \\ 3 & \end{array}
  \begin{array}{cc} & 7 \\ & 7 \\ 6 & \end{array}
  \begin{array}{cc} \\ \\ \cdots \end{array}
  \begin{array}{cc} & 3n_3+1 \\ & 3n_3+1 \\ 3n_3 & \end{array}
  \begin{array}{cc} \\ & 3n_3 + 4 \\ 3n_3 + 3  & \end{array}
  \begin{array}{cc} \\ & 3n_3 + 6 \\ 3n_3 + 5 &  \end{array}
  \begin{array}{cc} \\ \\ \cdots \end{array} \\[10pt]
 \label{baseptnKRc2-2-1}
  & \begin{array}{cc} \\ & 3n_3 + 2n_2 \\ 3n_3 + 2n_2-1 &  \end{array}
  \begin{array}{cc} \\ \\ 3n_3 + 2n_2 + 1 \end{array}
  \begin{array}{cc} \\ \\ 3n_3 + 2n_2 + 3 \end{array}
  \begin{array}{cc} \\ \\ \cdots \end{array}
  \left. \begin{array}{cc} \\ \\ 3n_3 + 2n_2 + 2n_1-1 \end{array} \right\}
  \begin{array}{cc} \\ \\ , \end{array}
 \end{align} 
 whose weight is $( 9n_3^2 + n_3 )/2 + 2n_2^2 + n_2 + n_1^2 + 6n_3n_2 + 3n_3n_1 + 2n_2n_1 - 1 $, 
 for $n_1, n_3 \geq 0, n_2 > 0$.  
 
 \begin{align*}
  & \left\{ \begin{array}{cc} & 2 \\ & 2 \\ 1 & \end{array} \right.
  \begin{array}{cc} & 5 \\ & 5 \\ 4 & \end{array}
  \begin{array}{cc} \\ \\ \cdots \end{array}
  \begin{array}{cc} & 3n_3-1 \\ & 3n_3-1 \\ 3n_3-2 & \end{array}
  \begin{array}{cc} \\ \\ 3n_3  + 1 \end{array}
  \begin{array}{cc} \\ \\ 3n_3  + 3 \end{array}
  \begin{array}{cc} \\ \\ \cdots \end{array}
  \left. \begin{array}{cc} \\ \\ 3n_3 + 2n_1-1 \end{array} \right\}
  \begin{array}{cc} \\ \\ , \end{array}
 \end{align*} 
 whose weight is $( 9n_3^2 + n_3 )/2 + n_1^2 + 3n_3n_1$, for $n_1, n_3 \geq 0$.  
 This is not the case $n_2 = 0$ of \eqref{baseptnKRc2-2-1}.  
 
 The smallest 2-cluster in \eqref{baseptnKRc2-2-1} has one extra move forward to enter the game, 
 which entails a prestidigitation through the 3-clusters, and making \eqref{baseptnKRc2-2-1} into 
 \begin{align*}
  & \left\{ \begin{array}{cc} & 2 \\ & 2 \\ 1 & \end{array} \right.
  \begin{array}{cc} & 5 \\ & 5 \\ 4 & \end{array}
  \begin{array}{cc} \\ \\ \cdots \end{array}
  \begin{array}{cc} & 3n_3-1 \\ & 3n_3-1 \\ 3n_3-2 & \end{array}
  \begin{array}{cc} \\ & 3n_3 + 2 \\ 3n_3 + 1  & \end{array}
  \begin{array}{cc} \\ & 3n_3 + 4 \\ 3n_3 + 3  & \end{array} 
  \begin{array}{cc} \\ \\ \cdots \end{array} \\[10pt]
  & \begin{array}{cc} \\ & 3n_3 + 2n_2 \\ 3n_3 + 2n_2-1 &  \end{array}
  \begin{array}{cc} \\ \\ 3n_3 + 2n_2 + 1 \end{array}
  \begin{array}{cc} \\ \\ 3n_3 + 2n_2 + 3 \end{array}
  \begin{array}{cc} \\ \\ \cdots \end{array}
  \left. \begin{array}{cc} \\ \\ 3n_3 + 2n_2 + 2n_1-1 \end{array} \right\}
  \begin{array}{cc} \\ \\ . \end{array}
 \end{align*} 
 
 To tell the cases in which this extra move is made or not apart, 
 we simply check if $\lambda$ contains the 2-cluster 
 $\begin{array}{cc} 1 \\ 1 \end{array}$ as a sediment or not.  
\end{proof}

\begin{theorem}
\label{thmKRc2-1}
  For $n, m \in \mathbb{N}$, let $kr^c_{2-1}(n, m)$ be the number of partitions of $n$ into $m$ parts 
  with at most one occurrence of the part 1, 
  and difference at least three at distance three such that 
  if parts at distance two differ by at most one, 
  then their sum, together with the intermediate part, 
  is $\equiv 2 \pmod{3}$.  
  Then, 
  \begin{align}
  \nonumber
  & \sum_{m,n \geq 0} kr^c_{2-1}(n,m) q^n x^m \\
  \label{eqGF-KRc2-1}
  = & \sum_{n_1, n_2, n_3 \geq 0} \frac{ 
      q^{ ( 9n_3^2 + n_3 )/2 + 2n_2^2 + n_2 + n_1^2 + 6n_3n_2 + 3n_3n_1 + 2n_2n_1 } 
      (-q; q^2)_{n_2} \; x^{3n_3 + 2n_2 + n_1} }
    { (q; q)_{n_1} (q^2; q^2)_{n_2} (q^3; q^3)_{n_3} }.  
  \end{align}
\end{theorem}

\begin{proof}
 It suffices to observe that $kr^c_{2-1}(n+m, m) = kr_6(n, m)$.  
 Then the result becomes a corollary of Theorem \ref{thmKR6}.  
\end{proof}

By means of shifts of all parts of a partition, 
one can put restrictions on the size of the smallest part and its number of occurrences.  
Then, the generating functions of such partitions may be obtained as 
corollaries of Theorems \ref{thmKR5}-\ref{thmKRc2-1}.  

\section{Alternative Series for Kanade and Russell's Conjectures 5-6}
\label{secKR5-6Alt}

In \cite{K-Capparelli-GG}, it is shown that 
\begin{align}
\label{eqGG1alt}
  \sum_{n \geq 0} \frac{ q^{ n^2 } (-q; q^2)_n x^n}{ (q^2; q^2)_n } 
  = & \sum_{n_1, n_2 \geq 0} 
  \frac{ q^{ 4n_2^2 + (3n_1^2 - n_1)/2 + 4n_2n_1 } x^{2n_2 + n_1} }
    { (q; q)_{n_1} (q^4; q^4)_{n_2} }.  
\end{align}

Using this formula in \eqref{eqGF_KR5}, \eqref{eqGF_KR6}, \eqref{eqGF-KRc1-2}, 
\eqref{eqGF-KRc2-2} and \eqref{eqGF-KRc2-1}, 
and a little $q$-series algebra will yield the following.  

\begin{align}
\nonumber
& \sum_{m,n \geq 0} kr_5(n,m) q^n x^m \\
\nonumber
= & \sum_{n_1, n_2, n_3 \geq 0} \frac{ q^{ ( 9n_3^2 + 5n_3 )/2 + 2n_2^2 + n_2 + n_1^2 + 6n_3n_2 + 3n_3n_1 + 2n_2n_1 } 
    (-q; q^2)_{n_2} x^{3n_3 + 2n_2 + n_1} }
  { (q; q)_{n_1} (q^2; q^2)_{n_2} (q^3; q^3)_{n_3} } \\[10pt]
\label{eqGF_KR5alt}
= & \sum_{n_1, m_2, n_3, m_4 \geq 0} \frac{ 
      q^{ 8m_4^2 + 2m_4 + ( 9n_3^2 + 5n_3 )/2 + ( 5m_2 + m_2 )/2 + n_1^2 } }
  { (q; q)_{n_1} (q; q)_{m_2} (q^3; q^3)_{n_3} (q^4; q^4)_{m_4} } \\[10pt]
\nonumber
& \times q^{12m_4n_3 + 8m_4m_2 + 4m_4n_1 + 6n_3m_2 + 3n_3n_1 + 2m_2n_1 } 
\; x^{4m_4 + 3n_3 + 2m_2 + n_1}
\end{align}
  
\begin{align}
\nonumber
& \sum_{m,n \geq 0} kr_6(n,m) q^n x^m \\
\nonumber
= & \sum_{n_1, n_2, n_3 \geq 0} \frac{ 
    q^{ ( 9n_3^2 + 7n_3 )/2 + 2n_2^2 + 3n_2 + n_1^2 + n_1 + 6n_3n_2 + 3n_3n_1 + 2n_2n_1 } 
    (-q; q^2)_{n_2} x^{3n_3 + 2n_2 + n_1} }
  { (q; q)_{n_1} (q^2; q^2)_{n_2} (q^3; q^3)_{n_3} } \\[10pt]
\label{eqGF_KR6alt}
= & \sum_{n_1, m_2, n_3, m_4 \geq 0} \frac{ 
    q^{ 8m_4^2 + 6m_4 + ( 9n_3^2 + 7n_3 )/2 + (5m_2^2 + 5m_2)/2 + n_1^2 + n_1 } }
  { (q; q)_{n_1} (q; q)_{m_2} (q^3; q^3)_{n_3} (q^4; q^4)_{m_4} } \\[10pt]
\nonumber
& \times q^{12m_4n_3 + 8m_4m_2 + 4m_4n_1 + 6n_3m_2 + 3n_3n_1 + 2m_2n_1 } 
\; x^{4m_4 + 3n_3 + 2m_2 + n_1}
\end{align}
  
\begin{align}
\nonumber
& \sum_{m,n \geq 0} kr^c_{1-2}(n,m) q^n x^m \\
\nonumber
= & \sum_{n_1, n_2, n_3 \geq 0} \frac{ 
    q^{ ( 9n_3^2 - n_3 )/2 + 2n_2^2 + n_2 + n_1^2 + 6n_3n_2 + 3n_3n_1 + 2n_2n_1 } 
    (-1/q; q^2)_{n_2} \; x^{3n_3 + 2n_2 + n_1} }
  { (q; q)_{n_1} (q^2; q^2)_{n_2} (q^3; q^3)_{n_3} } \\[10pt]
\label{eqGF-KRc1-2alt}
= & \sum_{n_1, m_2, n_3, m_4 \geq 0} \frac{ 
    q^{ 8m_4^2 + 2m_4 + ( 9n_3^2 - n_3 )/2 + (5m_2^2 + m_2)/2 + n_1^2 } }
  { (q; q)_{n_1} (q; q)_{m_2} (q^3; q^3)_{n_3} (q^4; q^4)_{m_4} } \\[10pt]
\nonumber 
& \times (1 + x^2 q^{8m_4 + 6n_3 + 4m_2 + 2n_1 + 2} ) \\[5pt]
\nonumber 
& \times  q^{12m_4n_3 + 8m_4m_2 + 4m_4n_1 + 6n_3m_2 + 3n_3n_1 + 2m_2n_1 } 
\; x^{4m_4 + 3n_3 + 2m_2 + n_1}
\end{align}

\begin{align}
\nonumber
& \sum_{m,n \geq 0} kr^c_{2-2}(n,m) q^n x^m \\
\nonumber
= & \sum_{n_1, n_2, n_3 \geq 0} \frac{ 
    q^{ ( 9n_3^2 + n_3 )/2 + 2n_2^2 + n_2 + n_1^2 + 6n_3n_2 + 3n_3n_1 + 2n_2n_1 } 
    (-1/q; q^2)_{n_2} \; x^{3n_3 + 2n_2 + n_1} }
  { (q; q)_{n_1} (q^2; q^2)_{n_2} (q^3; q^3)_{n_3} } \\[10pt]
\label{eqGF-KRc2-2alt}
= & \sum_{n_1, m_2, n_3, m_4 \geq 0} \frac{ 
    q^{ 8m_4^2 + 2m_4 + ( 9n_3^2 + n_3 )/2 + (5m_2^2 + m_2)/2 + n_1^2 } }
  { (q; q)_{n_1} (q; q)_{m_2} (q^3; q^3)_{n_3} (q^4; q^4)_{m_4} } \\[10pt]
\nonumber 
& \times (1 + x^2 q^{8m_4 + 6n_3 + 4m_2 + 2n_1 + 2} ) \\[5pt]
\nonumber 
& \times q^{12m_4n_3 + 8m_4m_2 + 4m_4n_1 + 6n_3m_2 + 3n_3n_1 + 2m_2n_1 } 
\; x^{4m_4 + 3n_3 + 2m_2 + n_1}
\end{align}

\begin{align}
\nonumber
& \sum_{m,n \geq 0} kr^c_{2-1}(n,m) q^n x^m \\
\nonumber
= & \sum_{n_1, n_2, n_3 \geq 0} \frac{ 
    q^{ ( 9n_3^2 + n_3 )/2 + 2n_2^2 + n_2 + n_1^2 + 6n_3n_2 + 3n_3n_1 + 2n_2n_1 } 
    (-q; q^2)_{n_2} \; x^{3n_3 + 2n_2 + n_1} }
  { (q; q)_{n_1} (q^2; q^2)_{n_2} (q^3; q^3)_{n_3} }  \\[10pt]
\label{eqGF-KRc2-1alt}
= & \sum_{n_1, m_2, n_3, m_4 \geq 0} \frac{ 
    q^{ 8m_4^2 + 2m_4 + ( 9n_3^2 + n_3 )/2 + ( 5m_2 + m_2 )/2 + n_1^2 } }
  { (q; q)_{n_1} (q; q)_{m_2} (q^3; q^3)_{n_3} (q^4; q^4)_{m_4} } \\[10pt]
\nonumber
& \times q^{12m_4n_3 + 8m_4m_2 + 4m_4n_1 + 6n_3m_2 + 3n_3n_1 + 2m_2n_1 } 
\; x^{4m_4 + 3n_3 + 2m_2 + n_1}
\end{align}

The combinatorics of the new formulas is as follows.  
We focus on the 2-clusters only, 
as the incorporation of the 1- and 3-clusters in the discussion is routine.  
The 2-clusters are lined up as
\begin{align*}
 \left\{ 
  \begin{array}{cc} & 2 \\ 1 & \end{array}
  \begin{array}{cc} & 4 \\ 3 & \end{array}
  \begin{array}{cc}  \\ \cdots \end{array}
  \begin{array}{cc} & 2n_2 \\ 2n_2 - 1 & \end{array} 
 \right\}
  \begin{array}{cc}  \\ . \end{array}
\end{align*}
Then, we set $n_2 = 2m_4 + m_2$ for $m_2, m_4 \in \mathbb{N}$ 
and move the $i$th largest 2-cluster $m_2-i$ times forward for $i = 1, 2, \ldots, m_2$.  
\begin{align*}
 \left\{ 
  \begin{array}{cc} & 2 \\ 1 & \end{array}
  \begin{array}{cc} & 4 \\ 3 & \end{array}
  \begin{array}{cc}  \\ \cdots \end{array}
  \begin{array}{cc} & 4m_4 \\ 4m_4 - 1 & \end{array} 
  \begin{array}{cc} & 4m_4 + 2 \\ 4m_4 + 1 & \end{array} 
  \begin{array}{cc} 4m_4 + 4 \\ 4m_4 + 4 \end{array} 
  \begin{array}{cc} & 4m_4 + 7 \\ 4m_4 + 6 & \end{array} 
  \begin{array}{cc}  \\ \cdots \end{array}
 \right\}
\end{align*}
Next, we declare the consecutive 2-clusters 
\begin{align*}
  \begin{array}{cc} & 2 \\ 1 & \end{array}
  \begin{array}{cc} & 4 \\ 3 & \end{array} 
  \begin{array}{cc}  \\ , \end{array}
  \quad
  \begin{array}{cc} & 6 \\ 5 & \end{array}
  \begin{array}{cc} & 8 \\ 7 & \end{array} 
  \begin{array}{cc}  \\ , \end{array}
  \quad
  \begin{array}{cc}  \\ \cdots \end{array}
  \begin{array}{cc}  \\ , \end{array}
  \quad
  \begin{array}{cc} & 4m_4 - 2 \\ 4m_4 - 3 & \end{array}
  \begin{array}{cc} & 4m_4 \\ 4m_4 - 1 & \end{array} 
\end{align*}
\emph{2-cluster pair}s, and the others \emph{individual 2-cluster}s.  
One forward move on an individual 2-cluster still adds one to the total weight, 
but one forward move on a 2-cluster pair adds four.  

The procession of 2-cluster pairs through individual 2-clusters 
are defined similar to movement of pairs in 
\cite[section 3]{K-Capparelli-GG}.  
The procession of 2-cluster pairs through 1-clusters, 
or prestidigitation of 2-cluster pairs through the 3-clusters 
are defined in the obvious way.  

\section{$q$-series Versions of Kanade-Russell Conjectures}
\label{secKR-qSeries}

Given a partition counter, say $kr_1(n, m)$ in Theorem \ref{thmKR1}, 
we define
\begin{align*}
 KR_1(n) = \sum_{m \geq 0} kr_1(m, n).  
\end{align*}
Then, we have the following relation between the generating functions.  
\begin{align*}
 \sum_{n \geq 0} KR_1(n) q^n = \left. \sum_{n, m \geq 0} kr_1(m, n) x^m q^n \right\vert_{x = 1}.  
\end{align*}
In other words, substituting $x = 1$ renders the track of number of parts ineffective.  

Using this idea in the respective theorems above 
gives the following conjectured $q$-series identities, 
in conjunction with \cite{KR-conj}.

\begin{conj}
\label{conjKRAnalytic}
  \begin{align}
  \label{eqConjKR1Analytic} 
    \frac{1}{ ( q, q^3, q^6, q^8; q^{9})_\infty } & \stackrel{?}{=}
    \sum_{n_1, n_2 \geq 0} \frac{ q^{3n_2^2 + n_1^2 + 3n_1n_2} }
      { (q; q)_{n_1} (q^3; q^3)_{n_2} }
    \\[10pt]
  \label{eqConjKR2Analytic} 
    \frac{1}{ ( q^2, q^3, q^6, q^7; q^{9})_\infty } & \stackrel{?}{=}
    \sum_{n_1, n_2 \geq 0} \frac{ q^{3n_2^2 + 3n_2 + n_1^2 + n_1 + 3n_1n_2} }
      { (q; q)_{n_1} (q^3; q^3)_{n_2} } 
     \\[10pt]
  \label{eqConjKR3Analytic} 
    \frac{1}{ ( q^3, q^4, q^5, q^6; q^{9})_\infty } & \stackrel{?}{=}
    \sum_{n_1, n_2 \geq 0} \frac{ q^{3n_2^2 + 3n_2 + n_1^2 + 2n_1 + 3n_1n_2} }
      { (q; q)_{n_1} (q^3; q^3)_{n_2} } 
     \\[10pt]
  \label{eqConjKR4Analytic} 
    \frac{1}{ ( q^2, q^3, q^5, q^8; q^{9})_\infty } &  \stackrel{?}{=}
    \sum_{n_1, n_2 \geq 0} \frac{ q^{3n_2^2 + 2n_2 + n_1^2 + n_1 + 3n_1n_2} }
      { (q; q)_{n_1} (q^3; q^3)_{n_2} }
  \end{align}
  \begin{align}
    \nonumber
    & \frac{1}{ ( q, q^3, q^4, q^6, q^7, q^{10}, q^{11} ; q^{12})_\infty } \\[10pt]
  \label{eqConjKR5Analytic} 
    \stackrel{?}{=} & \sum_{n_1, n_2, n_3 \geq 0} \frac{ 
    q^{ ( 9n_3^2 + 5n_3 )/2 + 2n_2^2 + n_2 + n_1^2 + 6n_3n_2 + 3n_3n_1 + 2n_2n_1 } 
      (-q; q^2)_{n_2} }
    { (q; q)_{n_1} (q^2; q^2)_{n_2} (q^3; q^3)_{n_3} } 
     \\[10pt]
    \nonumber
    = & \sum_{n_1, m_2, n_3, m_4 \geq 0} \frac{ 
	  q^{ 8m_4^2 + 2m_4 + ( 9n_3^2 + 5n_3 )/2 + ( 5m_2^2 + m_2 )/2 + n_1^2 } }
      { (q; q)_{n_1} (q; q)_{m_2} (q^3; q^3)_{n_3} (q^4; q^4)_{m_4} } \\[10pt]
    \nonumber
    & \times q^{ 12m_4n_3 + 8m_4m_2 + 4m_4n_1 + 6n_3m_2 + 3n_3n_1 + 2m_2n_1 } 
  \end{align}
  \begin{align}
  \nonumber
    & \frac{1}{ ( q^2, q^3, q^5, q^6, q^7, q^8, q^{11} ; q^{12})_\infty } \\[10pt]
  \label{eqConjKR6Analytic} 
  \stackrel{?}{=} & \sum_{n_1, n_2, n_3 \geq 0} \frac{ 
      q^{ ( 9n_3^2 + 7n_3 )/2 + 2n_2^2 + 3n_2 + n_1^2 + n_1 + 6n_3n_2 + 3n_3n_1 + 2n_2n_1 } 
      (-q; q^2)_{n_2} }
    { (q; q)_{n_1} (q^2; q^2)_{n_2} (q^3; q^3)_{n_3} } \\[10pt]
  \nonumber
= & \sum_{n_1, m_2, n_3, m_4 \geq 0} \frac{ 
    q^{ 8m_4^2 + 6m_4 + ( 9n_3^2 + 7n_3 )/2 + (5m_2^2 + 5m_2)/2 + n_1^2 + n_1 } }
  { (q; q)_{n_1} (q; q)_{m_2} (q^3; q^3)_{n_3} (q^4; q^4)_{m_4} } \\[10pt]
\nonumber
& \times q^{12m_4n_3 + 8m_4m_2 + 4m_4n_1 + 6n_3m_2 + 3n_3n_1 + 2m_2n_1 } 
  \end{align}
\end{conj}

\eqref{eqConjKR1Analytic} is a combination of \eqref{eqGF_KR1} and \cite[$I_1$]{KR-conj}, 
\eqref{eqConjKR2Analytic} of \eqref{eqGF_KR2} and \cite[$I_2$]{KR-conj}, 
\eqref{eqConjKR3Analytic} of \eqref{eqGF_KR3} and \cite[$I_3$]{KR-conj}, 
\eqref{eqConjKR4Analytic} of \eqref{eqGF_KR4} and \cite[$I_4$]{KR-conj}, 
\eqref{eqConjKR5Analytic} of \eqref{eqGF_KR5}, \eqref{eqGF_KR5alt} and \cite[$I_5$]{KR-conj}, 
and \eqref{eqConjKR6Analytic} of \eqref{eqGF_KR6}, \eqref{eqGF_KR6alt} and \cite[$I_6$]{KR-conj}.  

\section{Comments and Further Work}
\label{secConclusion}

The series constructed in this paper are different from the series constructed in \cite{KR-newseries}.  
The approach is different, as well.  

The usage of Gordon marking in the proof of Theorem \ref{thmKR1}, 
or other theorems in section \ref{secKR1-4} does not make them immensely easier.  
One can simply declare, say, in Theorem \ref{thmKR1},  
$[3k, 3k]$ or $[3k+1, 3k+2]$ admissible pairs, 
other parts singletons, and imitate the proofs in \cite{K-Capparelli-GG}.  

However, Gordon marking is vital in the proof of Theorem \ref{thmKR5}, 
or other theorems in sections \ref{secKR5-6}-\ref{secKR5-6Alt}; 
and it is prudent to have all Kanade-Russell conjectures together.  
Without Gordon marking, the proof of Theorem \ref{thmKR5} becomes more tedious than it already is.  

Normally, an $r$-cluster cannot go through an $s$-cluster if $s \geq r$ \cite{K-GordonMarking}.  
The \emph{prestidigitation} is an exception without which the proofs are 
longer and less elegant, if not impossible (please see the appendix).  

Unfortunately, in sections \ref{secKR5-6}-\ref{secKR5-6Alt}, 
one cannot make the sum condition on the 3-clusters 
$\equiv 0 \pmod{3}$ instead of $\equiv 1 \pmod{3}$ or $\equiv 2 \pmod{3}$.  
It is not possible to define forward or backward moves compatible with both Gordon marking and 
the given difference conditions.  

For instance, let $kr^c_{3-3}(n, m)$ be the number of partitions of $n$ into $m$ parts 
with difference at least three at distance three such that if the difference at distance two 
is at most one, then the sum of those parts, together with the intermediate part, 
is divisible by three.  
The 3-clusters in a partition $\lambda$ enumerated by $kr^c_{3-3}(n, m)$ must be of the form
\begin{align*}
  \left\{ \begin{array}{cc} \\ \\ ( \textrm{ parts } \leq k - 3 ) \end{array}
  \begin{array}{cc} k\\ k \\ k \end{array}
  \begin{array}{cc} \\ \\ ( \textrm{ parts } \geq k + 3 ) \end{array} \right\}
  \begin{array}{cc} \\ \\ . \end{array}
\end{align*} 
One simply cannot make a forward move on the 3-cluster in the partition below.  
\begin{align*}
  \left\{ 
  \begin{array}{cc} \mathbf{1} \\ \mathbf{1} \\ \mathbf{1} \end{array}
  \begin{array}{cc}  \\ \\ 4 \end{array}
  \right\}
  \quad \stackrel{\textrm{ move }}{\longrightarrow} \quad
  \left\{ \begin{array}{cc}  \\  \\  \end{array} \right.
  \underbrace{
  \begin{array}{cc} \mathbf{2} \\ \mathbf{2} \\ \mathbf{2} \end{array}
  \begin{array}{cc}  \\ \\ 4 \end{array}
  }_{!}
  \left. \begin{array}{cc}  \\  \\  \end{array} \right\}
  \stackrel{\textrm{ adjustment }}{\longrightarrow} 
  \left\{ \begin{array}{cc}  \\  \\  \end{array} \right.
  \underbrace{
  \begin{array}{cc}  \\ \\ 1 \end{array}
  \begin{array}{cc} \mathbf{3} \\ \mathbf{3} \\ \mathbf{3} \end{array}
  }_{!}
  \left. \begin{array}{cc}  \\  \\  \end{array} \right\}
\end{align*} 
The violation of the difference condition persists after the adjustment.  
To resolve it, we should either compromise the invariance of the number of $r$-clusters for fixed $r$, 
or define some other kind of moves.  
In short, the $\equiv 0 \pmod{3}$ case cannot be treated with the machinery developed in this paper.  

It should be possible to incorporate differences at distance four, 
so that 4-clusters enter the stage.  
However, such a venture is not advisable before we have partition identities, 
or conjectures, pertaining to difference at distance four 
as natural extensions of Kanade-Russell conjectures \cite{KR-conj}.  

Of course, the biggest open problem is the proof of Kanade-Russell conjectures.  
Using the series constructed here or in \cite{KR-newseries}, 
and Bailey pairs, will it be possible to give at least an analytic proof of the conjectures?  
A good starting point might be \cite{Stanton}.  
 
\noindent
{\bf Acknowledgements: }
We thank George E. Andrews, Alexander Berkovich, Karl Mahlburg and Dennis Stanton 
for useful discussions, suggesting references or terminology 
during the preparation of the manuscript.  
The term \emph{prestidigitation} and the story in the appendix
is due to the historian and my friend Emre Erol of Sabanc{\i} University.  

\bibliographystyle{amsplain}

\begin{thebibliography}{normalsize}
%    Insert the bibliography data here. % yapma be?

\bibitem{AAG-Capparelli} K. Alladi, G. E. Andrews, and B. Gordon, 
Refinements and generalizations of Capparelli’s conjecture on partitions,
\emph{J. Algebra} {\bf 174}, 636--658 (1995).  

%% bu kaliyor mu?
\bibitem{Andrews-Gordon} George E. Andrews, 
An analytic generalization of the Rogers-Ramanujan identities for odd moduli,  
\emph{Proc. Nat. Acad. Sci.} {\bf 71(10)}, 4082--4085, (1974).  
DOI: 10.1073/pnas.71.10.4082

%% bu herhalde kalacak
\bibitem{Andrews-bluebook} George E. Andrews, 
\emph{The Theory of Partitions}, 
Cambridge University Press, (1984).  

\bibitem{Capparelli-Conj} S. Capparelli, 
On some representations of twisted affine Lie algebras and combinatorial identities, 
\emph{J. Algebra} {\bf 154}, 335--355 (1993).  

\bibitem{KR-conj} S. Kanade and M.C. Russell, 
IdentityFinder and Some New Identities of Rogers-Ramanujan Type, 
\emph{Experimental Mathematics} {\bf 24}(4), 419--423 (2015).  

\bibitem{KR-newseries} S. Kanade and M.C. Russell, 
Staircases to analytic sum-sides for many new integer partition identities of Rogers-Ramanujan type, 
\emph{submitted}, 

{\tt http://arxiv.org/abs/1803.02515v1}.  

\bibitem{K-GordonMarking} K. Kur\c{s}ung\"{o}z, 
Parity considerations in Andrews–Gordon identities, 
\emph{Eur. J. Comb.} {\bf 31}(3), 976--1000 (2010).  

\bibitem{K-Clusters} K. Kur\c{s}ung\"{o}z, 
Cluster parity indices of partitions, 
\emph{The Ramanujan J} {\bf 23}(1-3), 195--213 (2010).  

\bibitem{K-Capparelli-GG} K. Kur\c{s}ung\"{o}z, 
Andrews-Gordon Type Series for Capparelli's and G\"{o}llnitz-Gordon Identities, 
\emph{submitted}, 

{\tt http://arxiv.org/abs/1807.11189}.  

\bibitem{Stanton} D. Stanton, 
The Bailey-Rogers-Ramanujan Group, 
\emph{Contem. Math} {\bf 291}
($q$-series with Applications to Combinatorics, Number Theory, and Physics), 55--70 (2000).  

\end{thebibliography}

\noindent
{\bf Appendix: }
Now, let's try to visualize this process with a metaphor. 

Imagine a person walking into a fancy cupcake store to taste the delicacies 
that he heard so much about from his colleagues at work. 
The cupcakes are neatly arranged in a large display case with one shelf over another. 
Each shelf has different kinds of cupcakes put into boxes of different sizes. 
There is certain logic to the way the boxes are displayed. 
The shelves have boxes with three cupcakes at the first two rows 
followed by a box with a single cupcake or two cupcakes at the back of the shelves. 

The hypothetical cupcake enthusiast starts gazing colorful cupcakes of various types 
until his eyes are fixated towards a single box with a single cupcake in it. 
The box is located behind two bigger boxes with three cupcakes in each 
at a middle shelve as per the logic of display and there is hardly any space for one 
to grab the box with the single cupcake from the back of the shelf. 
The cupcake enthusiast is certain of his choice and makes a move 
towards the box in the back to grab it. 
The shop owner at the register sees the customer’s move and immediately interrupts him: 
‘I am afraid you can’t move the box at the back of the self without my help sir! 
It’s impossible for you to squeeze your hand through the narrow space between the shelves 
without ruining the cupcakes.  

The cupcake enthusiast stops for a brief moment, 
listens to the shop owner’s warning and then he confidently keeps moving 
towards the box with the single cupcake behind the two larger boxes with three cupcakes in each. 
He thrusts his hand towards the narrow middle shelf and magic happens in the blink of an eye. 
The customer is able bring both the single-size box and the single cupcake of his choice 
to the front of the shelf albeit separately. 
The customer turned out to be a prestidigitator and performed some masterly sleight-of-hand. 
He retrieved the single cupcake of his choice by relocating it 
through the two other boxes with three cupcakes. 
The cupcake was swiftly put in an out of these larger boxes and united 
at the very front of the shelf with its original box in the end. 
The impossible became possible under this rare circumstance that allowed 
different cake to be put in and out of the boxes of three. 

The shop owner was awed. 
He asked if the same trick could be done with another middle shelf 
that had a box with two cupcakes at the back as well. 
The cupcake enthusiast tried his trick there too and it worked again!  
The box of two and the cupcakes are separately delivered to the front 
while the to cupcakes got in an out of the boxes of three. 
Not only that, he was able to put back all the boxes that he retrieved 
from the back of the middle shelf to their original places reversing his trick. 
The shopkeeper, now amused, decided to offer his cupcakes free of charge to the customer. 

Needless to say, 
each cupcake represents individual numbers and each box represents 
a free cluster of a particular size in this metaphor. 
I can only hope that the ‘prestidigitator cupcake enthusiast’s proof of his 
‘sleight-of-hand’ would also prove to be as ‘amusing’ 
for his fellow mathematicians in real life as it does in the metaphor.  

\end{document}